\documentclass[12pt]{article}
\usepackage{amsmath,amsthm,amsfonts,amssymb,mathrsfs,bm}
\usepackage[usenames]{color}
\usepackage{amssymb}

\usepackage{bbm}
\RequirePackage{amsthm,amsmath,amsfonts,amssymb,mathtools}
\RequirePackage{enumitem}
\RequirePackage[numbers,sort&compress]{natbib}
\RequirePackage[
  colorlinks=true,
  linkcolor=blue,
  citecolor=blue,
  urlcolor=blue,
  linktoc=page
]{hyperref}
\RequirePackage{graphicx}

\parskip 	\bigskipamount

\theoremstyle{plain}

\newtheorem{theorem}{Theorem}[section]
\newtheorem{lemma}[theorem]{Lemma}
\newtheorem{corollary}[theorem]{Corollary}
\newtheorem{proposition}[theorem]{Proposition}
\theoremstyle{definition}
\newtheorem{definition}[theorem]{Definition}

\newtheorem*{remark}{Remark}

\newenvironment{acks}
  {\par\noindent\textbf{Acknowledgement. }\ignorespaces}
  {\par}

\newcommand{\exclude}{~\backslash~}
\DeclarePairedDelimiterX{\infdivx}[2]{(}{)}{%
  #1\;\delimsize\|\;#2}
\DeclareMathOperator{\var}{\mathrm{Var}}

\newcommand{\pspace}[2]{\mathcal{M}(#2;#1)}
\newcommand{\optloss}{\hat{L}_n}
\newcommand{\loss}[1]{L_n(#1)}

\newcommand{\de}{\mathrm{d}}
\newcommand{\argmin}{\mathrm{argmin}}
\newcommand{\argmax}{\mathrm{argmax}}
\newcommand{\indicator}{\mathbbm{1}}

\newcommand{\supp}{\mathrm{Supp}}
\newcommand{\E}{\mathbb{E}}
\newcommand{\R}{\mathbb{R}}
\newcommand{\abs}[1]{\left\lvert#1\right\rvert}

\newcommand{\norm}[1]{\left\lVert#1\right\rVert}
\newcommand{\normnorm}[1]{\lVert#1\rVert}

\newcommand{\Z}{\mathbb{Z}}

\newcommand{\mc}[1]{\mathcal{#1}}

\newcommand{\nbr}{N_{[]}}
\newcommand{\est}{\tilde{f}_n}
    \newcommand{\KL}{\mathrm{KL}}
   \newcommand{\He}{\mathrm{H}}
\newcommand{\comp}[1]{{#1}^{\complement}}
\newcommand{\dist}{\mathrm{dist}}
\newcommand{\npmle}{\hat{f}_n}

\newcommand{\In}{\mathrm{in}}
\newcommand{\Out}{\mathrm{out}}
\newcommand{\TV}{\mathrm{TV}}

\def\qt#1{\qquad\text{#1}}

\def\argmin{\mathop{\rm argmin}}
\def\argmax{\mathop{\rm argmax}}

\newenvironment{funding}
  {\par\noindent\textbf{Funding. }\ignorespaces}
  {\par}

\makeatletter
\renewcommand*{\@cite@ofmt}{\hbox}
\makeatother

\begin{document}
\title{Gaussian mixtures and non-parametric likelihoods \\ through the lens of statistical mechanics}

\author{
\small
\begin{tabular}{c}{Subhroshekhar Ghosh\thanks{Authors listed in alphabetical order of surnames}}\\ Department of Mathematics\\National University of Singapore\\subhrowork@gmail.com
\end{tabular}
\and
\small
\begin{tabular}{c}{Aditya Guntuboyina}\\
 Department of Statistics\\
University of California, Berkeley
\\ aditya@stat.berkeley.edu
\end{tabular}
\and
\small
\begin{tabular}{c}
{Satyaki Mukherjee \thanks{Corresponding author}}\\ Department of Mathematics\\National University of Singapore \\satyaki.mukhyo@gmail.com
\end{tabular}
\and 
\small
\begin{tabular}{c}
{Hoang-Son Tran \thanks{Corresponding author}}\\ Department of Mathematics\\National University of Singapore \\ hoangson.tran@nus.edu.sg 
\end{tabular}
}

\date{} 
\maketitle

\begingroup
\renewcommand{\thefootnote}{}
\endgroup

\begin{abstract}
In this work, we investigate Gaussian Mixture Models ({\it abbreviated} GMMs) and the related problem of nonparametric maximum likelihood estimation ({\it abbreviated} NPMLE) from the perspective of statistical mechanics. In particular, we establish stability guarantees for the NPMLE procedure that improve and extend existing guarantees. Crucially, we obtain guarantees on the Kullback--Leibler divergence between NPMLE estimators and the ground truth, a type of result which has been known to be  challenging in the literature on this problem. 

In particular, we provide high-probability upper bounds on the KL divergence between the NPMLE and the true density that are of order  $\min\big\{\frac{(\log n)^{d+2}}{n} , \frac{\log n}{\sqrt n}\big\}$, which cover a wide range of scenarios for the relative sizes of $n$ and $d$. We obtain similar guarantees for approximate solutions to the NPMLE problem, addressing realistic situations where optimization algorithms need to be stopped in finite time, allowing access only to approximations to the true NPMLE. 

A cornerstone of our approach is an analysis of the function class complexity of logarithms of gaussian mixture densities, which is able to handle their unboundedness, and could be of wider interest. Our methods lead to novel confidence-interval guarantees for entropy estimation in Gaussian mixtures, demonstrating their wider impact.

We also establish correspondences between stability phenomena in the NPMLE problem and concepts such as chaos and multiple valleys in random energy landscapes of statistical mechanics models. While these correspondences are largely of a conceptual nature at this point, we believe that these connections, especially those with concentration phenomena and Langevin dynamics, may be developed into a toolbox for studying a wide variety of random optimization problems in statistics and machine learning.
\end{abstract}





\tableofcontents





\newpage

\section{{ Introduction}}
A classical (standard, homoskedastic) Gaussian Mixture Model (GMM) is a probability distribution on $\mathbb{R}^d$ that is given by a probability density of the form

\begin{equation} \label{eq:GMM-discrete}
    \rho(x) = \sum_{k=1}^m \pi_k \cdot \frac{1}{(2\pi)^{d/2}} \exp\Big (-\frac{\|x-\theta_k\|^2}{2}\Big ),
\end{equation}
where the $\{\theta_k\}_{k=1}^m$ are referred to as {\it centres} and the entries of the probability vector $(\pi_k)_{k=1}^m$ are called the mixing weights. 

GMMs have long been a classical modeling paradigm for a wide range of applications in statistics and machine learning, many of which are detailed in standard references such as \cite{everitt2013finite, titterington1981finite, mclachlan2019finite, lindsay1995mixture}. They are also a core example motivating the development of the EM algorithm \cite{dempster1977maximum}. Other prominent applications include clustering and unsupervised learning \cite{mclachlan2019finite, bishop2006pattern, lindsay1995mixture}, nonparametric density estimation \cite{mclachlan2019finite, lindsay1995mixture}, novelty detection \cite{zong2018deep}, and empirical Bayes problems \cite{lindsay1995mixture}.

The general Gaussian location mixture model is an extension of the classical discrete GMM where the discrete mixture of Gaussians (over their centres) is relaxed to a general mixture against a probability distribution (referred to as a {\it mixing measure}) on $\R^d$. In more concrete terms, the general Gaussian Mixture Model is a probability distribution  on $\mathbb{R}^d$ whose density can be written as
\begin{equation} \label{eq:GMM-density}
    f_\mu(x) = \frac{1}{(2\pi)^{d/2}} \int_{\mathbb{R}^d} \exp\Big (-\frac{\|x-\theta\|^2}{2}\Big )\mu(\de\theta), \quad \mu \in \mathcal{P}(\mathbb{R}^d),
\end{equation}
where $\mathcal{P}(\mathbb{R}^d)$ denotes the set of all probability measures on $\mathbb{R}^d$. The sense in which \eqref{eq:GMM-density} relaxes \eqref{eq:GMM-discrete} is clear; in fact, \eqref{eq:GMM-discrete} is a special case of \eqref{eq:GMM-density} by considering the probability measure $\mu$ that puts mass $\pi_k$ on the point $\theta_k$, for $1 \le k \le m$. 

We will denote by $\mathcal{M}$ the class of all Gaussian location mixture densities on $\mathbb{R}^d$, i.e.
\[ \mathcal{M}:= \{f: f \text{ is of the form given by \eqref{eq:GMM-density} for some $\mu \in \mathcal{P}(\mathbb{R}^d)$}\}.\]
To set up notation, for any mixing measure $\mu \in \mathcal{P}(\mathbb{R}^d)$, we shall denote by $f_{\mu}\in \mathcal{M}$ the corresponding Gaussian location mixture density given by \eqref{eq:GMM-density}; conversely, given a density $f\in \mathcal{M}$, we let $\mu_f$ be its corresponding mixing measure.



Given $n$ data points $X_1,\ldots,X_n \in \mathbb{R}^d$, we consider the following (empirical) log-likelihood function
\begin{equation} \label{eq:log-likelihood}
 L_n(f) := \frac{1}{n} \sum_{i=1}^n \log f(X_i), \quad  f\in \mathcal{M}.     
\end{equation}

Assuming that the data $X_1,\ldots,X_n$ are drawn independently from a GMM density $f_* \in \mathcal{M}$, in {\it Non-Parametric Maximum Likelihood Estimation}, one  estimates $f_*$ by maximizing $L_n(f)$ over all $f \in \mathcal{M}$, the class of all Gaussian location mixture densities. Concretely, a {\it Non-Parametric Maximum Likelihood Estimator (NPMLE)}, denoted by $\hat{f}_n$, is defined as any maximizer of the log-likelihood over $\mathcal{M}$:
\begin{equation} \label{eq:NP-opt}
\hat{f}_n \in \argmax_{f\in \mathcal{M}} \: L_n (f).
\end{equation}
For convenience, we set $\optloss:= \sup_{f\in \mathcal{M}} L_n(f)$, which equals $L_n(\hat{f}_n)$ by definition.

In practice, exact maximization of $L_n(f)$ may be computationally
challenging, which motivates the study of approximate NPMLEs.
Accordingly, given a sequence $(\varepsilon_n)_{n \ge 1}$, we consider estimators $\est$ satisfying
\[
\loss{\est} \ge \optloss - \varepsilon_n .
\]
That is, $\est$ achieves log-likelihood within $\varepsilon_n$ of the optimal value.

\cite{kiefer1956consistency} and \cite{robbins1950generalization} appear to be the earliest works to study this estimator, and it has since attracted substantial attention.  Several basic properties of \eqref{eq:NP-opt} are discussed in the monographs \cite{bohning2000computer} and \cite{lindsay1995mixture}. The optimization problem \eqref{eq:NP-opt} is convex because the objective function is concave in $f$ and the constraint set $\mathcal{M}$ is a convex (though infinite-dimensional) class of densities. It can be shown that the estimator $\hat{f}_n$ always exists, and the first order optimality conditions for \eqref{eq:NP-opt} can be explicitly characterized (see, for example, \cite[Chapter 5]{lindsay1995mixture}). In one dimensional setting, $\hat{f}_n$ is also unique, as proved by \cite{lindsay1993uniqueness}. In contrast, when $d > 1$, \cite{soloffjakeguntuboniya} demonstrated that there exist $n = 4$ points $X_1, X_2, X_3, X_4$ in $\mathbb{R}^2$ for which $\hat{f}_n$ is not unique. Nevertheless, it remains an open question whether uniqueness holds almost surely when the data $X_1, \dots, X_n$ are drawn from $f_*$. 

There also exist several algorithms for computing an approximate maximizer of \eqref{eq:NP-opt}. These include vertex direction and vertex exchange methods \cite{bohning2000computer}, expectation-maximization based methods \cite{jiang2009general}, black-box convex optimization applied to a grid-based discretization \cite{koenker2014convex}, exemplar-based methods \cite{soloffjakeguntuboniya} and specialized convex optimization methods tailored to this problem \cite{kim2020fast, zhang2024efficient}. 

The statistical properties of $\hat{f}_n$ (and the corresponding mixing measure $\hat{\mu}_n$) as estimators of $f_*$ (and $\mu_*$), have been investigated in a number of works. \cite{kiefer1956consistency} (see also \cite{chen2017consistency, pfanzagl1988consistency}) proved consistency of $\hat{\mu}_n$. \cite{zhang2009generalized} proved rates of convergence of $\hat{f}_n$ to $f_*$ under the Hellinger metric in the univariate case; these results were extended to $d \geq 2$ by \cite{saha2020nonparametric} (see also \cite{soloffjakeguntuboniya}). Structural properties of $\hat{\mu}_n$ (such as bounds on the number of atoms) are established in \cite{polyanskiy2020self} for $d = 1$. Minimax bounds for estimating densities in the class $\mathcal{M}$ have been proved by \cite{kim2014minimax} for $d = 1$ and by \cite{kim2022minimax} for $d \ge 1$.




\section{{ Key ideas and main results}}
\label{sec: ideas and results}

\subsection{{ NPMLE and statistical mechanics: a brief overview}} 
\label{sec:intro-statmech}
A conceptual aspect of this work is to look at the NPMLE problem through the lens of statistical mechanics. We will provide a detailed account of this in Section \ref{sec:stat-mech}; restricting ourselves only to an overview of the key ideas in this introductory section. The NPMLE procedure, as laid out in the optimization problem \eqref{eq:NP-opt} and \eqref{eq:GMM-density}, may be viewed as an optimization problem in a {\it random environment} (in the sense of statistical mechanics), with the randomness being that of the data. Many models in statistical mechanics can be posed as (solutions to) such random optimization problems -- this includes the famous models of first and last passage percolation, random polymer models, spin glasses and their ground states, traveling salesman problem (TSP) on random point sets, among others. For a detailed account, we refer the reader to Section \ref{sec:stat-mech}, and more generally to the monograph \cite{chatterjee2014superconcentration}; for wider connections between statistical inference and statistical physics, we refer the reader to the excellent survey \cite{zdeborova2016statistical} and the references therein.  

Over the years, a large body of tools and techniques have been developed in the statistical mechanics literature in order to investigate the fine properties of the solutions, and more generally, the landscapes of such optimization problems in random environments. In particular, \cite{chatterjee2014superconcentration} established, in the context of a large class of classical statistical mechanical models, a general theory connecting the {\it fluctuation behaviour} of the solutions of such random optimization problems to the more physical phenomena of {\it chaos} and {\it multiple valleys}.  

Broadly speaking, {\it chaos} refers to the sensitivity of the solution of the optimization problem to perturbations in the random environment (in the statistical context, the input data), whereas {\it multiple valleys} refers to spurious local near-optima pervading the landscape of the random optimization problem. The notions of chaos and multiple valleys have clear implications in a statistical and machine learning context, pertaining to the robustness and stability of data analytical procedures. 

On the other hand, fluctuations and related concentration phenomena are fundamental to classical mathematical statistics; c.f. the classical central limit theorem for MLEs connecting their fluctuation order to the Fisher information under the truth. Investigation of fluctuation phenomena is underpinned by an array of powerful analytical tools which can potentially be brought to bear on any model at hand. The theory developed in \cite{chatterjee2014superconcentration} establishes mutual implications between the order of fluctuations of the objective function in several random optimization problems in statistical mechanics with the occurrence of chaos and multiple valleys in the landscape of these models.

In this work, we investigate the NPMLE problem in the light of these statistical mechanical ideas. It turns out that the theory and the toolbox developed in \cite{chatterjee2014superconcentration} is not directly applicable in the NPMLE setup, largely due to the former being focused on the discrete setup of lattice models of statistical mechanics; whereas the NPMLE problem, by definition, resides in the continuum. However, intriguingly, we are able to separately analyse phenomena of chaos, multiple valleys (or the lack thereof), and fluctuation orders of the objective function in the setting of the NPMLE model, using direct methods centered on the information geometry of Gaussian mixtures. 

Our analysis leads to natural consequences in terms of the robustness and the stability of the NPMLE procedure. We clarify that the correspondences with statistical mechanics that we observe are of a conceptual and motivational nature at this point, rather than a direct application of theorems established in one field to problems in another. This raises the natural question of extending the general theory of equivalences between chaos, multiple valleys and fluctuation orders from a discrete statistical mechanical setup to more continuous settings that are suitable for studying random optimization problems arising in statistics and machine learning.

\subsection{{ Main results}} 
\label{sec: main results}
In this section, we discuss the main technical results that we obtain in this paper, borne out of the analysis motivated by the statistical mechanical considerations outlined in Section \ref{sec:intro-statmech}.

To formulate these results, we recall that the (squared) Hellinger distance and Kullback-Leibler divergence between density functions $f$ and $g$ on $\R^d$ are defined respectively as:
\begin{eqnarray*}
    \He^2(f,g) &:=& \int_{\R^d} \Big (\sqrt{f(x)} - \sqrt{g(x)} \Big )^2 \de x,\\
    \KL(f\|g) &:=& \int_{\R^d} \log \Big (\frac{f(x)}{g(x)}\Big ) f(x) \de x.
\end{eqnarray*}

\subsubsection{{ Stability phenomena}} \label{sec:intro-stability}
Our first result is Theorem \ref{thm:stability}, which concerns the algorithmic stability of the NPMLE. In particular,
let $\est$ be an estimator that nearly achieves the optimal value $\hat{L}_n$, we are interested in quantifying the closeness of $\est$ to the true density $f_*$. 


\begin{theorem}\label{thm:stability}
     Let $X_1, ..., X_n$ be i.i.d. samples from a GMM density $f_* \in \mathcal{M}$ whose mixing measure $\mu_*$ has support contained in a compact set $\Theta \subset \mathbb{R}^d$. 
     Let $(\varepsilon_n)_{n \ge 1}$ be any sequence of positive numbers and let $\est \in \mathcal{M}$ be any estimator (depending on the data $X_i$'s) which nearly achieves the optimal value, i.e., 
     \begin{align} \label{eq:epsLnucrit}
         \loss{\est} \ge \optloss- \varepsilon_n.
     \end{align}
     Then we have:
     \begin{enumerate}[label=(\roman*)]
        \item\label{it:stabilitypolylognovern} There exist positive constants $C_1,C_2$ depending only on $\Theta$ and $d$ such that
        \begin{equation} \label{eq:thm-stability-Hellinger-1}
             \He^2(f_*, \est)\le  \varepsilon_n+C_1 \frac{(\log n)^{d+1}}{n} 
         \end{equation}
         with probability at least $1-C_2\exp (- (\log n)^{d+1}  )$.
         \item \label{it:stabilitylognsqrtn}
        Let $n$ be large enough such that
         \begin{equation} \label{eq:thm-stability-condition}
             \varepsilon_n+C_1 \frac{(\log n)^{d+1}}{n} \le \frac{1}{2}(1-e^{-1})^2,
         \end{equation}
         then there exists $C_3 = C_3 (d,\Theta)>0$ such that 
             \begin{equation} \label{eq:thm-stability-KL-1}
                 \KL(f_*\|\est) \le  C_3 \Big (\varepsilon_n \log (\min \{\varepsilon_n^{-1},n \})+ \frac{(\log n)^{d+2}}{n} \Big ) 
             \end{equation}
              with probability at least $1-C_2\exp (- (\log n)^{d+1}  )$.

              \medskip
      \end{enumerate}
\end{theorem}

Theorems~\ref{thm:stability} advances the existing literature on the NPMLE for Gaussian location mixture models in two important ways, which we discuss below.

\noindent {\bf Approximate NPMLEs.} \quad First, inequality \eqref{eq:thm-stability-Hellinger-1} applies to approximate NPMLEs, namely estimators that maximize the likelihood only up to an accuracy level $\varepsilon_n$ (as defined in \eqref{eq:epsLnucrit}). Some earlier results—such as \cite[Remark 2]{zhang2009generalized},  \cite[Theorem 2.1]{saha2020nonparametric} and \cite[Corollary 8]{soloffjakeguntuboniya}—also allow for approximate NPMLEs. However, those results require $\varepsilon_n$ to converge to zero at a rate faster than the upper bound they establish for $\He^2(f_*, \est)$. Under our assumptions on the mixing distribution, this requirement effectively amounts to $\varepsilon_n = O \left((\log n)^{d+1}/n \right)$.

In contrast, inequality \eqref{eq:thm-stability-Hellinger-1} imposes no such restriction: it applies to arbitrary sequences $\varepsilon_n$, including those that decay slowly to zero or even remain at a small but fixed constant level. This is significant in view of the fact that there is no known algorithm to exactly compute the global optimum for the NPMLE problem, even in one dimension. In particular, optimization protocols to compute the NPMLE (e.g. iterative algorithms such as gradient descent or its variants) typically need to be stopped in finite time.  As such, every known algorithm yields only an approximate NPMLE, a situation that is addressed by Theorem \ref{thm:stability}, in contrast to existing literature. 

For practical algorithms used to compute an approximate NPMLE $\tilde{f}_n$, it is possible to obtain an explicit value of $\varepsilon_n$ satisfying \eqref{eq:epsLnucrit}. For example, a standard approach restricts the mixing measure to a finite set
$\mathcal{A}\subset\mathbb{R}^d$, turning the problem into a finite-dimensional
convex program over the mixing weights; this program can be solved to global
optimality up to negligible numerical error, and the resulting $\varepsilon_n$
is controlled by the fineness of $\mathcal{A}$. Concretely, let $\mathcal{K}_n$
denote the convex hull of $X_1,\dots,X_n$ and take $\mathcal{A}$ to be the set
of corners of a cover of $\mathcal{K}_n$ by hypercubes of side $\delta$. Every NPMLE is supported on $\mathcal{K}_n$ (see e.g., \cite[Proposition 4]{soloffjakeguntuboniya}) so $\mathcal{K}_n$ is precisely the region that must be
discretized. Then \cite[Proposition 6]{soloffjakeguntuboniya} shows that, for
all sufficiently small $\delta$, \eqref{eq:epsLnucrit} holds with
\begin{align*}
    \varepsilon_n = d \left(2 D_n^2 + \frac{1}{2} \right) \delta^2
\end{align*}
where $D_n$ is the diameter of the convex hull $\mathcal{K}_n$. 

Another approach for obtaining $\varepsilon_n$ that works for any approximate NPMLE $\tilde{f}_n$ proceeds via the inequality:
\begin{align*}
    \hat{L}_n - L_n(\tilde{f}_n)
    &= \frac{1}{n}\sum_{i=1}^n \log\frac{\hat{f}_n(X_i)}{\tilde{f}_n(X_i)}
     \;\le\; \frac{1}{n}\sum_{i=1}^n \frac{\hat{f}_n(X_i)}{\tilde{f}_n(X_i)} - 1\\
    &= \int\!\left[\frac{1}{n}\sum_{i=1}^n
        \frac{\phi(X_i-\theta)}{\tilde{f}_n(X_i)}\right]d\hat{\mu}_n(\theta) - 1
     \;\le\; \sup_{\theta\in\mathcal{K}_n}\left[\frac{1}{n}\sum_{i=1}^n
        \frac{\phi(X_i-\theta)}{\tilde{f}_n(X_i)}\right] - 1,
\end{align*}
where $\hat{\mu}_n$ is the mixing measure of an exact NPMLE $\hat{f}_n$ and the
last step uses that $\hat{\mu}_n$ is supported on $\mathcal{K}_n$. Hence \eqref{eq:epsLnucrit} holds
with
\begin{align}\label{vareps_comment}
    \varepsilon_n := \sup_{\theta\in\mathcal{K}_n}\left[\frac{1}{n}\sum_{i=1}^n
        \frac{\phi(X_i-\theta)}{\tilde{f}_n(X_i)}\right] - 1.
\end{align}
This bound clearly depends only on $\tilde{f}_n$ and the data, not on the unknown $\hat{f}_n$. The map $\theta\mapsto \tfrac{1}{n}\sum_i \phi(X_i-\theta)/\tilde{f}_n(X_i)$
is smooth with uniformly bounded gradient on the compact set $\mathcal{K}_n$, so
this supremum can be approximated by evaluating it on a fine grid and adding the
corresponding Lipschitz correction. Note also that (via the first order optimality condition) $\tilde{f}_n$ is an NPMLE (i.e., it maximizes $L_n(f)$ over all $f \in \mathcal{M}$) if and only if the right hand side of \eqref{vareps_comment} is nonpositive (see e.g.,  \cite[Theorem 2.1]{bohning2000computer}). As a result, if $\tilde{f}_n$ is an approximate NPMLE, it is reasonable to expect $\varepsilon_n$ as defined by \eqref{vareps_comment} to be a small positive quantity. Note however that Theorem \ref{thm:stability} does not place restriction on how small $\varepsilon_n$ needs to be. 

\color{black}

\noindent{\bf KL Risk Bounds.} The second way in which Theorem \ref{thm:stability} advances existing literature on the NPMLE for Gaussian location mixture models is the following. Inequality \eqref{eq:thm-stability-KL-1} yields bounds on the Kullback–Leibler risk $\KL(f_* \| \est)$. To the best of our knowledge, such risk bounds for the NPMLE under KL loss are new—even in the case $\varepsilon_n = 0$, corresponding to the exact NPMLE $\hat{f}_n$—and we therefore record this result as a separate corollary below. While \cite[Remark 5]{ma2025best} does present a KL-type bound, it appears to apply only to a \textit{restricted} NPMLE where the mixing measure is confined to a compact interval of $\mathbb{R}$.

\begin{corollary} \label{cor:npmle} 
Let $C_1$ be the constant as in Theorem~\ref{thm:stability}. Suppose $n$ is large enough such that 
         \begin{equation} \notag 
             C_1 \frac{(\log n)^{d+1}}{n} \le \frac{1}{2}(1-e^{-1})^2.
         \end{equation}         
         Then there exist constants $C_2, C_3$ depending on $\Theta$, and $d$ such that 
             \begin{equation} \label{cor_kl}
                 \KL(f_*\|\hat{f}_n) \le  C_3 \frac{(\log n)^{d+2}}{n}, 
             \end{equation}
              with probability at least $1 -C_2\exp (- (\log n)^{d+1})$.   
\end{corollary}
Establishing bounds for KL divergence (which dominates the squared Hellinger distance) is technically more challenging than deriving bounds for the squared Hellinger loss. The bound $(\log n)^{d+2}/n$ implied by inequality \eqref{cor_kl} is only slightly weaker (by a factor of $\log n$) than the corresponding bound $(\log n)^{d+1}/n$ in \eqref{eq:thm-stability-Hellinger-1} for the squared Hellinger distance.

The KL bound provided by inequality \eqref{eq:thm-stability-KL-1} relies on the condition \eqref{eq:thm-stability-condition}, which may fail to hold in certain regimes of $n$ and $d$. We are nevertheless able to establish a bound on $\KL(f_* \| \est)$ in such cases, albeit under an additional assumption on the mixing measure corresponding to $\est$. Specifically, we assume that $\est \in \mathcal{M}(\Theta; \tau)$ for some compact set $\Theta$ and constant $\tau > 0$, where both $\Theta$ and $\tau$ are independent of $n$. The class $\mathcal{M}(\Theta; \tau)$ is defined as follows:




\begin{definition} \label{def:M-Theta-tau}
For $ \Theta \subset \mathbb{R}^d$ and $\tau \in [0,1]$, define
\begin{equation} \label{eq:M-Theta-tau}
    \mathcal{M}(\Theta;\tau):= \Big \{ f\in \mathcal{M}: \mu_f(\Theta) \ge \tau \Big \},
\end{equation}
where $\mu_f$ denotes the mixing measure of the GMM density $f$.
\end{definition}

\begin{theorem}[Restricted approximate NPMLE] \label{thm:restricted-npmle}
    Let $X_1, ..., X_n$ be i.i.d. samples from a GMM density $f_* \in \mathcal{M}$ whose mixing measure $\mu_*$ has support contained in a compact set $\Theta \subset \R^d$. Let $\tau>0$ and consider the class $\mathcal{M}(\Theta;\tau)$ defined in \eqref{eq:M-Theta-tau}.
     Let $(\varepsilon_n)_{n \ge 1}$ be any sequence of positive numbers and let $\est$ be an estimator such that
     \begin{align}\label{rest_npmle}
         \loss{\est} \geq \loss{f_*} - \varepsilon_n ~~\text{ and } ~~\est \in \mathcal{M}(\Theta;\tau).
     \end{align}
      Then there exists a constant $C_0$ (depending only on $d, \Theta$ and $\tau$) such that 
     \begin{equation} \label{eq:expected-KL}
          \mathbb{E}[\KL(f_*\|\est) ] \le \varepsilon_n + \frac{C_0}{\sqrt{n}} \cdot
     \end{equation}
     In particular, we have 
     \begin{equation} \label{eq:probability-KL}
         \mathbb{P}\Big (\KL(f_*\|\est) \le \varepsilon_n + \frac{\log n}{\sqrt{n}} \Big ) \ge 1 - \frac{C_0}{\log n} \cdot
     \end{equation}
\end{theorem}

\begin{remark}
The assumption that \(\Theta\) contains \(\operatorname{supp}(\mu_*)\) is made only for notational convenience. A slightly more general version of Theorem~\ref{thm:restricted-npmle} can be stated in which the set \(\Theta\) appearing in \(\mathcal M(\Theta;\tau)\) is not required to contain \(\operatorname{supp}(\mu_*)\). In that case, however, the constant \(C_0\) depends on the pair of sets \((\Theta,\operatorname{supp}(\mu_*))\). We impose the inclusion \(\operatorname{supp}(\mu_*)\subseteq \Theta\) to avoid carrying this additional dependence in the notation.
\end{remark}

\noindent \textbf{The conditions in Theorem \ref{thm:restricted-npmle}.}  We begin by observing that the \emph{restricted class assumption}, i.e., the condition $\est \in \mathcal{M}(\Theta;\tau)$, is quite mild and reasonable, as it only requires the total mass of $\mu_{\est}$ on the set $\Theta$ to be at least $\tau$ (rather than imposing any pointwise lower bound on densities). In particular, the mixing measure need only place a fraction $\tau$ of its mass on a fixed compact set $\Theta$, and may still assign positive mass to locations arbitrarily far away. The only distributions excluded are those whose mixing measures concentrate nearly all of their mass outside every sufficiently large compact set, which would be a rather pathological situation. 

More importantly, in Lemma \ref{lm:assumption-M-Theta-tau}, we show that the restricted class assumption is satisfied by any estimator $\tilde{f}$ that is sufficiently close  to $f_*$ in Hellinger distance (in a mild sense, i.e., smaller than a suitable constant); namely, there exist $\Theta$ and $\tau$ such that $\tilde{f} \in \mathcal{M}(\Theta;\tau)$. An important consequence is that, if $\est$ is an estimator satisfying the sample size condition \eqref{eq:thm-stability-condition} in Theorem \ref{thm:stability}, then $\est$ automatically satisfies the restricted class condition with high probability. Thus, in this sense, the restricted class condition is weaker than the sample size condition \eqref{eq:thm-stability-condition} in Theorem \ref{thm:stability}.

\noindent\textbf{Comparison of the rates in Theorems \ref{thm:stability} and \ref{thm:restricted-npmle}.}
Inequality \eqref{eq:probability-KL} says that, under the condition \eqref{rest_npmle}, we have with high probability
\[
\KL(f_*\|\est) \le \varepsilon_n + \frac{\log n}{\sqrt{n}}.
\]
It may be noted that, for a fairly wide range of $n$ and $d$,  $(\log n)/\sqrt{n}$ can be substantially smaller than $(\log n)^{d+2}/n$, in which case Theorem \ref{thm:restricted-npmle} provides a sharper convergence guarantee than Theorem \ref{thm:stability}. For example, even for $d=10$ and $n$ on the order of $10^6$ (a relatively mild dependence between dimension and sample size for modern large-scale problems), $(\log n)/\sqrt{n}$ may be seen to be smaller than $(\log n)^{d+2}/n$ by an order of magnitude. 

Further, it is notable that in Theorem \ref{thm:restricted-npmle}, in expectation $\KL(f_*\|\est)$ does not even pay the logarithmic price, i.e., it is of the order $\varepsilon_n + C_0 n^{-1/2}$ (without any log factors). This may be of technical significance in the context of the existing literature on this problem. 


\noindent \textbf{Theorems \ref{thm:stability} and \ref{thm:restricted-npmle} and statistical mechanics.}
We can interpret Theorems \ref{thm:stability} and \ref{thm:restricted-npmle} from the perspective of viewing the NPMLE as a random optimization problem, connecting the NPMLE paradigm to statistical mechanics. Since functionals (e.g., energy functional) of a disordered system can be viewed as random functions due to the randomness of the environment, stability in optimizing random functions with respect to the changes in environment is of fundamental interest in the statistical mechanics literature. 

In the monograph \cite{chatterjee2014superconcentration}, Chatterjee investigated classical statistical mechanical models as random optimization problems, and introduced the notions of \emph{superconcentration, chaos}, and \emph{multiple valleys} to characterize the sensitivity of a system with respect to changes in the environment. The notions of superconcentration, chaos and multiple valleys as in \cite{chatterjee2014superconcentration} will be discussed in detail in Section \ref{sec:stat-mech}; here we limit ourselves to a bare-bones discussion for the purpose of understanding and interpreting our main results.


Roughly speaking, a random optimization problem is said to have multiple valleys if there are many vastly dissimilar near-optimal solutions. The setting where such multiple valleys do not occur in the optimization landscape is also known to be significant in the random optimization problems, and is broadly referred to as the phenomenon of \emph{Asymptotic Essential Uniqueness} ({\it abbrv.} AEU). The AEU property is known to occur in many random optimization problems, a principal example being the famous Travelling Salesman Problem ({\it abbrv.} T.S.P.) over random locations.

We observe that the stability of NPMLE in Theorem \ref{thm:stability} demonstrates that any near optimal solution to the NPMLE problem must be close to the optimum; in other words, the landscape of the log-likelihood function is devoid of multiple valleys, and thus may be seen to exhibit the phenomenon of asymptotic essential uniqueness.


\subsubsection{{ The complexity of logarithmic GMM densities}} \label{sec:intro-bracket}
A crucial ingredient in our analysis, for instance, in the proof of Theorem \ref{thm:stability} (ii) and Theorem \ref{thm:restricted-npmle}, is the complexity of the function class $\{ \log f: f \in \mathcal{M}\}$ consisting of logarithmic GMM densities. To this end, we will provide an upper bound on the so-called \emph{bracketing entropy} of this function class. 

To begin with this discussion, let us recall some standard notions in empirical process theory.
For two $\R$-valued functions $l(x), u(x)$ defined on a measurable space $(\mathcal{X},\mathcal{A})$, the \emph{bracket} $[l,u]$ is defined to be the set 
\[\Big\{f:\mathcal{X} \rightarrow \R \text{ measurable} \: \Big|\: l(x) \le f(x) \le u(x), \forall \: x\in \mathcal{X} \Big\}.\] Let $\mathbf{P}$ be a probability measure on $(\mathcal{X},\mathcal{A})$ and $\varepsilon>0$. An \emph{$\varepsilon$-bracket} in $L^2(\mathbf{P})$ is a bracket $[l,u]$ such that $\int |l-u|^2\de \mathbf{P} \le \varepsilon^2$. Let $\mathcal{F} \subset L^2(\mathbf{P})$ be some class of functions, the $\varepsilon$-\emph{bracketing number} $N_{[]}(\varepsilon, \mathcal{F}, L^2(\mathbf{P}))$ is the minimum number of $\varepsilon$-brackets in $L^2(\mathbf{P})$ needed to cover $\mathcal{F}$. 
The \emph{bracketing entropy} of $\mathcal{F}$, by definition, is the logarithm of the bracketing number.

As remarked earlier, we are interested in the bracketing numbers of the class $\{\log f: f\in \mathcal{M}\}$. 
Our primary motivation for considering the function class of logarithmic GMM densities rather than the class of GMM densities themselves (as in, for example, the entropy results of \cite{saha2020nonparametric}) lies in their direct relationship with the log-likelihood function $L_n$ defined in \eqref{eq:log-likelihood} and with the Kullback–Leibler divergence.
We note, however, that working with log-densities is substantially more delicate than working with densities. Heuristically, the log-density may diverge to infinity as the density approaches zero. Since bracketing entropy requires pointwise control, logarithmic densities pose additional technical challenges and are therefore significantly more difficult to handle.


The key to resolve the difficulties mentioned above is the subclass of GMM densities $\mathcal{M}(\Theta; \tau) $, which is defined in Definition \ref{def:M-Theta-tau}.
In the setting of the NPMLE, this condition is in fact mild and easy to satisfy. We show that if a GMM density $f$ is sufficiently close to $f_*$, then $f \in \mathcal{M}(\Theta;\tau)$ for some $\tau>0$ (see Lemma~\ref{lm:assumption-M-Theta-tau}).
The class $\mathcal{M}(\Theta;\tau)$ is thus large enough to be applicable in practice, yet sufficiently restricted to allow effective control of the bracketing entropy via a subtle splitting argument. Our main result on the bracketing entropy of logarithmic GMM densities over this class is stated below.

\begin{theorem}\label{thm:FiniteBracketingIntegral}  
Let $\Theta \subset \mathbb{R}^d$ be a compact set and $\tau> 0$. 
Let $f_*$ be a GMM density whose mixing measure $\mu_*$ is supported in $\Theta$.
Consider the class of densities $\mathcal{M}(\Theta;\tau)$ defined in \eqref{eq:M-Theta-tau} and let
\begin{equation}
    \log \mathcal{M}(\Theta;\tau):=\Big \{\log f :f\in \mathcal{M}(\Theta;\tau)\Big\}.
\end{equation}
There exists a threshold $\eta=\eta(\Theta,d,\tau)>0$ and a constant $C_{\ref{thm:FiniteBracketingIntegral}}= C_{\ref{thm:FiniteBracketingIntegral}}(\Theta,d,\tau)>0$ such that
\[\log N_{[]}(\varepsilon,\log \mathcal{M}(\Theta;\tau), {L^2(f_*)}) \le C_{\ref{thm:FiniteBracketingIntegral}}  |\log \varepsilon|^{d+1}, \quad \forall \: 0<\varepsilon\le \eta.\]
\end{theorem}


\subsubsection{{ Concentration phenomena for NPMLE}}

A key ingredient in our analytical approach is the following technical result on controlling fluctuation of the optimal log-likelihood $\optloss$ up to second moment. This result, to the best of our knowledge, is not known in the literature and is of independent interest. 

\begin{theorem}[Moment bounds of $\optloss$]
\label{thm:L2mu_bound}    
Let $X_1,\ldots,X_n$ be i.i.d. samples from a Gaussian location mixture density $f_*\in \mathcal{M}$ whose mixing measure is compactly supported. Then we have
    \begin{equation}\label{eq:L2mu_bound.thm}
        \mathbb{E} [ |\optloss - L_n(f_*)  |^p ] = o(n^{-p/2}), \quad \text{for } p =1,2.
    \end{equation}
\end{theorem}

We quickly note that as $f_*$ is a mixture of Gaussians, it has the property that $\log f_*$ is always non-positive. As such by applying Theorem~\ref{thm:L2mu_bound} with $p=1$, we observe that $\mathbb{E} [|\optloss|]$ is asymptotically equivalent to $\mathbb{E}[|\log f_*(X_1)|]$ upto an $o(n^{-1/2})$ term. The case $p=2$ provides an analogous bound on the fluctuations of $|\optloss|$.


Distributional properties of the maximum log-likelihood play a significant role in classical statistics. For instance, an early example of this is provided by the Wilks' Theorem, which provides a description of the distributional asymptotics for the log-likelihood under the null in a hypothesis testing scenario. 

The above result is also related to existing results in \cite{AzaisGassiatMercadier2009} and \cite{Hartigan1985}, as well as to a recent result of \cite{ZhangVolgushev2026}. To draw the connection, set 
\begin{align*}
  \Lambda_n := n\bigl(\optloss - L_n(f_*)\bigr) = n\Bigl(\sup_{f \in \mathcal{M}} L_n(f) - L_n(f_*)\Bigr),
\end{align*}
so that \eqref{eq:L2mu_bound.thm} immediately yields
\begin{align*}
  \E |\Lambda_n|^p = o(n^{p/2}), \qquad p = 1, 2.
\end{align*}
\cite{ZhangVolgushev2026} consider the variant of $\Lambda_n$ in which the supremum is restricted to those $f \in \mathcal{M}$ whose mixing measure is supported on a known compact set $\Theta$. For this restricted statistic, their Theorem~3.1 shows that, when $\mu_*$ (the mixing measure corresponding to $f_*$) is finitely discrete, the statistic converges in distribution to a tight limit, while it diverges to infinity in probability when $\mu_*$ is not finitely discrete. Since the restricted supremum is bounded above by
$\optloss$, our Theorem~\ref{thm:L2mu_bound} complements their result by showing that this divergence is $o_P(n^{1/2})$ for every compactly supported $\mu_*$, discrete or not.

\subsubsection{Implication of Theorem~\ref{thm:L2mu_bound} to differential entropy estimation}\label{subsec:stat_impl_L2mu}

Theorem~\ref{thm:L2mu_bound} has the following immediate consequence.
It identifies the first-order fluctuations of the optimized
log-likelihood value \(\hat L_n\): at the \(n^{-1/2}\) scale,
\(\hat L_n\) is asymptotically equivalent to the oracle empirical
log-likelihood \(L_n(f_*)\).

\begin{corollary}[Central limit theorem for the optimized log-likelihood]
\label{cor:clt_Lhat}
Let \(X_1,\ldots,X_n\) be i.i.d.\ samples from a Gaussian location mixture
density \(f_* \in \mathcal M\) whose mixing measure \(\mu_*\) is
compactly supported.  Define
\[
    \sigma_*^2 :=
     \operatorname{Var}_{f_*}\big(\log f_*(X)\big).
\]
Then \(0<\sigma_*^2<\infty\), and
\[
    \sqrt n \big(\hat L_n - \mathbb E_{f_*}\big[\log f_*(X)\big]\big)
    \rightsquigarrow
    N(0,\sigma_*^2).
\]
\end{corollary}

The proof of Corollary~\ref{cor:clt_Lhat} is deferred to
Appendix~\ref{app:prf_clt_Lhat}.  The result gives a direct
inferential interpretation of the optimized likelihood.  Although
\(\hat L_n\) is obtained by maximizing over the infinite-dimensional
class \(\mathcal M\), the maximization does not contribute an additional
first-order stochastic fluctuation to the optimized likelihood value.

Equivalently, since
\[
    H(f_*) = -\mathbb E_{f_*}\log f_*(X),
    \qquad X\sim f_*,
\]
is the \textit{differential entropy} of \(f_*\) (see e.g., \cite{cover2006elements}), Corollary~\ref{cor:clt_Lhat} implies that \(-\hat L_n\) is a root-\(n\) consistent estimator of \(H(f_*)\).  

Now if we can consistently estimate the variance $\sigma_*^2 := \text{Var}_{f_*} \log f_*(X)$, then we can get confidence intervals for $H(f_*)$. The following lemma, proved in Appendix~\ref{prf:appsigma_est},  does precisely that by describing a consistent estimator of $\sigma_*^2$.

\begin{lemma}\label{sigma_est}
  Let $X_1,\dots, X_n$ be i.i.d.\ samples from a Gaussian location
  mixture density $f_* \in \mathcal{M}$ whose mixing measure $\mu_*$
  is compactly supported. Let $\hat{f}_n$ be any NPMLE over
  $\mathcal{M}$. Then
\[
  \hat{\sigma}_n^2 :=
  \frac{1}{n} \sum_{i=1}^n
  \left(\log \hat{f}_n(X_i) - \hat{L}_n \right)^2
\]
converges in probability to
\[
  \sigma_*^2 := \operatorname{Var}_{f_*}\big(\log f_*(X)\big)
\]
as $n \rightarrow \infty$.
\end{lemma}

Combining Corollary~\ref{cor:clt_Lhat} with Lemma~\ref{sigma_est},
Slutsky's theorem gives the feasible studentized central limit theorem
\[
  \sqrt{n}\,
  \frac{\hat L_n-\mathbb E_{f_*}[\log f_*(X)]}{\hat\sigma_n}
  \;\rightsquigarrow\; N(0,1).
\]
Thus an asymptotically valid \((1-\alpha)\) confidence interval ($\mathrm{CI}_{1-\alpha}$) for
the differential entropy \(H(f_*)\) is
\begin{align}\label{eq:entropy_CI}
  \mathrm{CI}_{1-\alpha}(H(f_*))
  =
  \left[
      -\hat L_n
      -
      z_{1-\alpha/2}\frac{\hat\sigma_n}{\sqrt n},
      \;
      -\hat L_n
      +
      z_{1-\alpha/2}\frac{\hat\sigma_n}{\sqrt n}
  \right].
\end{align}
Both endpoints are computable from the NPMLE, with no additional
bandwidth or smoothing parameter required.

This provides a simple inferential interpretation of the optimized
likelihood.  Although \(\hat L_n\) is obtained by maximizing over the
infinite-dimensional class \(\mathcal M\), the estimator \(-\hat L_n\)
has the same first-order behavior as the oracle estimator
\[
    -L_n(f_*)
    =
    -\frac1n\sum_{i=1}^n \log f_*(X_i).
\]
Thus the nonparametric optimization affects the entropy estimator only
at smaller order.  The leading uncertainty is the ordinary sampling
fluctuation of \(\log f_*(X)\), and the limiting variance can be
estimated by the empirical variance of the fitted log-density values,
\[
    \hat\sigma_n^2
    =
    \frac1n\sum_{i=1}^n
    \left\{
        \log \hat f_n(X_i)-\hat L_n
    \right\}^2.
\]
In particular, for each fixed compactly supported mixing measure
\(\mu_*\), the procedure yields root-\(n\) Gaussian inference for
\(H(f_*)\), regardless of whether \(\mu_*\) is discrete, continuous, or
mixed.

The confidence interval \eqref{eq:entropy_CI} is related to, but
distinct from, existing work on entropy estimation under Gaussian
smoothing.  \cite{GoldfeldGreenewaldPolyanskiy2018,
GoldfeldGreenewaldWeedPolyanskiy2019,
GoldfeldGreenewaldNilesWeedPolyanskiy2020} study estimation of
\(h(S+Z)\), where \(Z\sim N(0,\sigma^2 I_d)\) and the distribution of
\(S\) is unknown.  In that setting, the statistician observes i.i.d.\
samples from the latent distribution of \(S\), and the plug-in estimator
is the entropy of \(\hat P_n*\varphi_\sigma\), where \(\hat P_n\) is
the empirical distribution of the \(S\)-samples.  By contrast, here the
data are sampled from the convolved density
\(f_*=\mu_* * \varphi\) itself, the mixing distribution \(\mu_*\) is
latent, and the estimator is based on the NPMLE over the full
nonparametric mixture class.  The result above therefore supplies a
likelihood-based, studentized central limit theorem for the entropy
estimator, rather than only a rate bound.

Mixture-based entropy estimation has also been studied by
\cite{RobinScrucca2023}, who fit finite mixture models, with Gaussian
mixtures as a main example.  A separate numerical literature studies
approximations and bounds for the entropy of a known Gaussian mixture;
see, for example, \cite{HuberBaileyDurrantWhyteHanebeck2008} and
\cite{MelbourneTalukdarBhabanMadimanSalapaka2022}.  Finally, the
likelihood-ratio results of \cite{ZhangVolgushev2026} show that related
excess-likelihood statistics can have qualitatively different behavior
depending on whether the true mixing measure is finitely discrete.  The
studentized entropy estimator above has a different form: after centering
by \(H(f_*)\) and scaling by \(\hat\sigma_n\), it has a standard normal
limit for every compactly supported \(\mu_*\) covered by
Corollary~\ref{cor:clt_Lhat} and Lemma~\ref{sigma_est}.

The proof of Lemma \ref{sigma_est} is crucially reliant on the
following result which can be seen as a corollary of Theorem
\ref{thm:L2mu_bound} and is fully proved in Appendix~\ref{prf:appL2prob}. 

\begin{corollary}\label{cor:L2prob}
  Let $X_1,\dots, X_n$ be i.i.d samples from a Gaussian location
  mixture density $f_* \in \mathcal{M}$ whose mixing measure $\mu_*$
  is compactly supported. Let $\hat{f}_n$ be any NPMLE over
  $\mathcal{M}$. 
  Then
  \begin{align}\label{cor:L2prob.eq}
    \frac{1}{n} \sum_{i=1}^n \left(\log \hat{f}_n(X_i) - \log f_*(X_i)
    \right)^2 \rightarrow 0 
  \end{align}
  in probability as $n \rightarrow \infty$. 
\end{corollary}

\subsubsection{{ Fluctuations in NPMLE}} \label{sec:intr-fluct}

It was demonstrated in \cite{chatterjee2014superconcentration} that, in the context of discrete statistical mechanical models, the phenomena of \emph{superconcentration} and \emph{chaos} are in fact equivalent, and each of them imply the so-called \emph {multiple valleys} phenomenon. Roughly speaking,
\emph{superconcentration} entails that the Poincar\'{e} inequality is not tight in terms of order; in other words, a (sequence of) functions $g_n$ of Gaussian random vectors $Z_n$ superconcentrates if $\var[g_n(Z_n)] = o(\E[\|\nabla g_n(Z_n)\|^2])$.

As already discussed in Section \ref{sec:intro-stability}, Theorem \ref{thm:stability} can be interpreted as an \emph{asymptotic essential uniqueness} (AEU) phenomenon for the NPMLE model, which is the antithesis of multiple valleys. Following the intuitive parallel with statistical mechanics, it is therefore reasonable to expect that the contra-positive of superconcentration would hold for the maximum (log)-likelihood (which is the optimal objective function) in the NPMLE. In other words, we are led to conjecture that the  Poincar\'{e} inequality should be tight for the maximum log-likelihood . 

Put differently, parallels with statistical mechanics anticipate that the variance $\var[\optloss]$ should be comparable in magnitude to $\mathbb{E}[\|\nabla \optloss\|^2]$, being equivalent upto factors that are independent of the data size $n$. This indeed turns out to be the case, and is confirmed by the following theorem:


\begin{theorem}[Fluctuations in NPMLE] \label{thm:anti-supercon}
    Let $X_1,\ldots,X_n$ be i.i.d. samples from a Gaussian location mixture density $f_* \in \mathcal{M}$ whose mixing measure $\mu_*$ is compactly supported. Then
    $\optloss$, as of a function of the data $(x_1,\ldots,x_n)$, is differentiable almost everywhere in $\R^{dn}$. 
    Moreover, $\optloss$ is anti-superconcentrated, namely, there exists $C\ge 1$ depending on the density $f_*$ but  independent of $n$ such that
        \begin{equation} \label{eq:Fluct_thm}
        C^{-1} \cdot\mathbb{E}[\|\nabla \optloss \|^2] \le \var[\optloss] \le C \cdot \mathbb{E}[\|\nabla \optloss\|^2] .
        \end{equation} 
\end{theorem}

We note in passing that, in \eqref{eq:Fluct_thm}, the gradient on the empirical log-likelihood $\hat{L}_n$ is in the {\it space variables} $X_1,\ldots,X_n$, and not in the {\it mixing measure} which is the variable with respect to which the optimization is taking place.

{
Inequalities of the form \eqref{eq:Fluct_thm} are closely related to
Poincaré-type functional inequalities, which play a central role in the
theory of concentration of measure. For distributions satisfying a
Poincaré inequality, the variance of a sufficiently regular functional
can be controlled by the expected squared norm of its gradient, leading
to quantitative bounds on fluctuations and stability. Such inequalities
serve as a fundamental tool in understanding concentration phenomena for
high-dimensional random objects, for instance c.f. \cite{ledoux2001concentration,BoucheronLugosiMassart2013}.
}

We remark that the upper bound in \eqref{eq:Fluct_thm} follows from the fact that the measure $\prod_{i=1}^nf_*(x_i)\de x_i$ (which is the equilibrium distribution of the Langevin dynamics \eqref{eq:Langevin}) satisfies the Poincar\'e inequality whose the Poincar\'e constant is independent of $n$. Indeed, by the tensorisation property of PI, the Poincar\'e constant of the product measure $\prod_{i=1}^n f_*(x_i)\de x_i$ is the same as the Poincar\'e constant of $f_*(x) \de x$. In \cite{Bardet18}, it was shown that the class of Gaussian convolutions with compactly supported measures satisfies PI. Since $f_*(x)\de x$ is the convolution of the standard Gaussian with the mixing measure $\mu_*$ which is compactly supported, it follows that $f_*(x)\de x$ satisfies PI.

We would like to emphasize that the rigorous implications between chaos, superconcentration and multiple valleys, as expounded in \cite{chatterjee2014superconcentration}, have been provably demonstrated only in the context of the class of discrete statistical mechanical models discussed therein. In particular, the theorems and the techniques of \cite{chatterjee2014superconcentration} do not apply to the NPMLE problem, and we develop from scratch a dedicated analysis tailored to this setting, in order to  establish the phenomena that would be motivated by the parallels to statistical mechanics.

\subsubsection{{ Chaos and Langevin dynamics}} \label{sec:robustness}
We complete the circle of ideas motivated by parallels with statistical mechanics with an investigation into the chaotic properties of the NPMLE procedure. Intuitively, in machine learning terms this can be perceived as a certain kind of stability property against perturbations in the input data to the random optimization problem. Chaos in statistical mechanical terms, broadly speaking, refers to the sensitivity of a functional of a system to perturbations in its random environment. In statistical terms, it may be viewed as the response of the solution of a random optimization problem to the perturbations in the input data.

Following the theoretical framework developed in \cite{chatterjee2014superconcentration}, it is natural to conjecture that the NPMLE procedure should be \emph{non-chaotic}; in other words, in a sense it would be stable to perturbations in the input data. In this section, we confirm that this indeed the case. 

In order to state the result in technical terms, we need to clarify the notion of perturbation of inputs. It turns out that in statistical mechanics there is a standard notion of perturbing a random environment, which is defined using canonical stochastic dynamics that leaves distributional properties of the model intact. The perturbation is then defined by a coupling of the original realization of the system with its path-wise evolution under this dynamics, allowed to evolve for a small time. 

To give a simple example of such dynamical perturbation, we can consider a Gaussian random variable $Z_0$ evolving under an Ornstein-Uhlenbeck flow $(Z_t)_{t \ge 0}$. For small times $t$, the variable $Z_t$ can be coupled pathwise to the original $Z_0$ as a small perturbation, yet retain its distributional  properties. In the setting of more general distributions, such as Gaussian mixture models, there is an analogous \emph{Langevin dynamics} which can be used to achieve a similar perturbation; this is detailed in technical terms in Section \ref{sec:stat-mech}; see especially \eqref{eq:Langevin}. 

In statistical mechanics, chaotic effect on optimal solutions to random (discrete) optimization problems can be articulated directly in terms of the {\it overlap} between such solutions (e.g., number of common elements in two sets). In statistical and machine learning applications, the background space is often continuous (e.g., for the NPMLE it is the space of probability measures), which leads to some consideration regarding what an appropriate metric might be. 

It turns out that a suitable notion of overlap or similarity between probability densities is the so-called {\it Bhattacharyya coefficient}, defined for two densities $f$ and $g$ by the expression
\[
\mathrm{BC}(f,g) = \int \sqrt{f(x)g(x)}\,\mathrm{d}x.
\]
Clearly, $0 \le \mathrm{BC}(\cdot,\cdot) \le 1$, with BC near 1 indicating strong similarity between the densities under consideration. Thus, under perturbation of input data, the BC between the respective optima being close to 0 would indicate onset of {\it chaos}, whereas the same BC being close to 1 would indicate the lack thereof.
For a detailed discussion on the appropriateness of this metric in the context of perturbations and chaos via analogies with statistical mechanics, we refer the reader to Appendix \ref{sec:BC}.



We are now ready to state the non-chaotic stability of the NPMLE in the form of the following result.
\begin{corollary} \label{cor:robust}
 Let the environment $(X_1^{(t)},\ldots,X_n^{(t)})$ evolve via the Langevin dynamics \eqref{eq:Langevin} and $\npmle^{(t)}$ be the NPMLE with respect to the perturbed environment at a time $t \ge 0$. 
Then  $\mathbb{E}[\mathrm{BC}(\npmle, \npmle^{(t)})] \to 1$ as $n\to \infty$.
\end{corollary}

\begin{remark}
    We define the perturbation $X_i^{(t)}$ via the Langevin dynamics in order to keep parity with the standard practice in statistical mechanics theory. However, we record that similar results can also be shown to hold when the perturbation changes the true distribution $f_*$ by a small amount, although we do not venture into a detailed proof within the limited scope of this article. A more in-depth, quantitative analysis of the above result, especially the dependency between the degree of perturbation (captured by $t$), the system size $n$ and the specific model of perturbations, would be an interesting direction to pursue as an extension of the present work.
\end{remark}

The proof of the above result, in the setting of GMMs, follows via a direct argument pursuant to the main theorems established earlier in this paper. As such, it is stated as a corollary to the main results. 


\section{ NPMLE through the lens of statistical mechanics}
\label{sec:stability-NPMLE}
Motivated by the overlaps with disordered systems, we now view the NPMLE through the lens of random optimization problems from statistical mechanics. We clarify at the outset that the general theory developed in \cite{chatterjee2014superconcentration} between the various notions of stability and fluctuations in random optimization problems applies to the setting of discrete statistical mechanical models discussed therein; it is not oriented towards the NPMLE and more generally towards statistical and machine learning problems, which usually reside in a continuous space. Thus, we leverage the stability theory of disordered statistical mechanics models in a motivational sense, and establish the different notions of stability and robustness for the NPMLE separately via direct arguments. 

In the NPMLE setting, the analogue of the energy functional of the system is the negative log-likelihood function $-L_n(f)$, over the configuration space $\mathcal{M}$ of all GMM densities. We let the environment $(X_1,\ldots,X_n)$ evolve via the Langevin dynamics \eqref{eq:Langevin}. 

Theorem \ref{thm:stability} and Theorem \ref{thm:restricted-npmle} say that the random optimization problem \eqref{eq:npmle-optimization} in this setting is stable, which is similar in spirit to the asymptotic essential uniqueness (AEU) phenomenon, and the contrapositive of the multiple valleys/peaks. Hence, it is natural to expect that superconcentration and chaos do not occur in NPMLE. In other words, the log-likelihood function should satisfy the Poincar\'e inequality with an $O(1)$ constant, and the optimal mixing measure should be robust under small perturbations via the Langevin dynamics. 

This is indeed the case, as our Theorem \ref{thm:anti-supercon} confirms that there exists a constant $C \ge 1$, independent of $n$, such that
\begin{equation} \label{eq:anti-supercon}
    C^{-1} \cdot \mathbb{E} [ \|\nabla \optloss\|^2] \le \var[\optloss] \le C \cdot \mathbb{E} [ \|\nabla \optloss\|^2].
\end{equation}


Since superconcentration and chaos are equivalent in the discrete statistical mechanics setup, we are motivated to conjecture that $\optloss$ is also not chaotic in the sense of  Definition \ref{def:chaos}. However, as we remarked earlier, it is generally unclear how chaos in Definition \ref{def:chaos} relates the usual notion of chaos in statistical physics (in the sense of sensitivity of output to perturbations in input), especially when the NPMLE involves an optimization problem over the space of measures, which is infinite-dimensional and far more complicated than the Gaussian polymer model. We believe that establishing such implications would be a natural objective of future research in this direction. 

However, the more direct, physical interpretation of chaos in the sense of the sensitivity of the maximizing measure to small perturbations in the input data, can be formalized into a technical statement in terms of Bhattacharyya Coefficient between the optimal and near-optimal measures (Corollary \ref{cor:robust}).


\section{{ Proof ideas and main proofs}}
\label{sec:proofs}
In this section, we sketch the proof of our theorems except Theorem~\ref{thm:FiniteBracketingIntegral}. Since the proof of Theorem~\ref{thm:FiniteBracketingIntegral} is long and technical, its proof has been moved to Appendix~\ref{app:prfFiniteBracketingIntegral}. Some technical lemmas will also be used along the way. Again due to page constraints, we will only state the statements of these lemmas and refer the reader to the Appendix for the complete proofs of these results.


\subsection{ Stability in NPMLE: Proof of Theorem \ref{thm:stability}} \label{subsec:proof-thm-stability}

\subsubsection{ Proof of Theorem \ref{thm:stability} \ref{it:stabilitypolylognovern}}

Let $t>0$ be suitably chosen later, we bound the probability $\mathbb{P} (\He^2(\est, f_*) \ge  t)$. 
Let $M:= 3\sqrt{\log n}$ and $\Theta_M:= \{x \in \mathbb{R}^d: \dist (x,\Theta) \le M\}$, where $\dist(x,\Theta):= \inf_{y \in \Theta} \|x-y\|$. 
We define accordingly a pseudonorm on $\mathcal{M}$ as $\|f\|_{\infty,\Theta_M}:= \sup_{x\in \Theta_M} |f(x)|$.

Let $N:=N(n^{-2}, \mathcal{M}, \|\cdot\|_{\infty,\Theta_M})$ be the $n^{-2}$-covering number of $\mathcal{M}$ with respect to the pseudonorm $\|\cdot\|_{\infty,\Theta_M}$. 
The following result, a known consequence of \cite{saha2020nonparametric} and detailed in Appendix~\ref{prf:appguntu-entropy}, gives an upper bound on $N$.

\begin{lemma} \label{lm:guntu-entropy}
    Let $M=3\sqrt{\log n}$ and $N:=N(n^{-2}, \mathcal{M}, \|\cdot\|_{\infty,\Theta_M})$.
    There exists a constant $C_{\ref{lm:guntu-entropy}}$ (depending only on $d$ and $\Theta$) such that
    $\log N \le C_{\ref{lm:guntu-entropy}} (\log n)^{d+1}$.
\end{lemma}

Let $\{g_1,\ldots,g_N\}$ be a fixed $n^{-2}$-covering set of $\mathcal{M}$ with respect to $\|\cdot\|_{\infty, \Theta_M}$. 
We define the following subset of indices $\mathbf{J} \subset [N] :=\{1,2,\ldots,N\}$ as
\begin{equation}
    \mathbf{J} := \Big \{j \in [N]: \exists \: h_j \in \mathcal{M} \text{ such that } \|h_j - g_j\|_{\infty, \Theta_M} \le n^{-2} \text{ and } \He^2(h_j, f_*) \ge t \Big \}.
\end{equation}
We will use the following technical lemma, whose proof can be found at \ref{prf:appHellinger-technical}:
\begin{lemma} \label{lm:Hellinger-technical}
    Assume that $|L_n(\est) - \optloss| \le \varepsilon_n$, 
    then on the event $\{\He^2(\est, f_*) \ge t\}$ we have the following inequality 
    \[ 
    \Big [\max_{j\in \mathbf{J}}\prod_{i=1}^n \frac{h_j(X_i) + 2n^{-2}}{f_*(X_i)}\Big ] 
    \cdot \Big [\prod_{i=1}^n  \Big (1+n^2{\mathbf{1}(X_i\notin\Theta_M)}\Big ) \Big] \ge \exp (-n \varepsilon_n).\]
\end{lemma}

Given Lemma \ref{lm:Hellinger-technical}, $\mathbb{P} (\He^2(\est, f_*)\ge t )$ is then upper bounded by 
\begin{eqnarray} \label{eq:2terms}
    &&\mathbb{P} \Big (\Big [\max_{j \in \mathbf{J}}\prod_{i=1}^n \frac{h_j(X_i) + 2n^{-2}}{f_*(X_i)} \Big ] 
    \cdot \Big [\prod_{i=1}^n  \Big (1+n^2{\mathbf{1}(X_i\notin\Theta_M)}\Big ) \Big] 
    \ge e^{-n\varepsilon_n} \Big ) \notag \\
    &\le& \mathbb{P} \Big (\max_{j \in \mathbf{J}}\prod_{i=1}^n \frac{h_j(X_i) + 2n^{-2}}{f_*(X_i)}  \ge e^{-n\varepsilon_n - \lambda}\Big ) + \mathbb{P} \Big (\prod_{i=1}^n  \Big (1+n^2{\mathbf{1}(X_i\notin\Theta_M)}\Big )  \ge e^\lambda \Big ),
\end{eqnarray}
where $\lambda>0$ will be chosen later.

We separately control each term in \eqref{eq:2terms}, using the following technical results whose proofs can be found in the Appendix (in ~\ref{prf:app1-term} and ~\ref{prf:app2-term} respectively). For the first term, we have:
\begin{lemma} \label{lm:1-term}
    For any $h\in \mathcal{M}$ and $\gamma\in \R$, we have
    \[ \mathbb{P} \Big (\prod_{i=1}^n \frac{h(X_i) + 2n^{-2}}{f_*(X_i)}  \ge e^{-\gamma}\Big ) 
    \le \exp \Big (\frac{\gamma}{2} - \frac{n}{2}\He^{2}(h,f_*) +\sqrt{2}\int \sqrt{f_*}\Big ). \]
\end{lemma}
Since $\He^2(h_j,f_*) \ge t$ for every $j \in \mathbf{J}$, applying  Lemma \ref{lm:1-term} with $\gamma = n \varepsilon_n+ \lambda$ gives
    \[ \mathbb{P} \Big (\prod_{i=1}^n \frac{h_j(X_i) + 2n^{-2}}{f_*(X_i)}  \ge e^{-\gamma}\Big ) 
    \le \exp \Big (\frac{n\varepsilon_n+\lambda}{2} - \frac{nt}{2}+\sqrt{2}\int \sqrt{f_*}\Big ), \quad \forall j \in \mathbf{J}. \]
Combining with Lemma \ref{lm:guntu-entropy}, the first term in \eqref{eq:2terms} is then bounded by
\begin{eqnarray*}
    \mathbb{P} \Big (\max_{j \in \mathbf{J}}\prod_{i=1}^n \frac{h_j(X_i) + 2n^{-2}}{f_*(X_i)}  \ge e^{-n\varepsilon_n - \lambda}\Big ) 
    &\le& N \cdot \max_{j\in \mathbf{J}} \mathbb{P} \Big (\prod_{i=1}^n \frac{h_j(X_i) + 2n^{-2}}{f_*(X_i)}  \ge e^{-n\varepsilon_n - \lambda}\Big ) \\
    &\lesssim& \exp\Big (C_{\ref{lm:guntu-entropy}}(\log n)^{d+1} +\frac{n\varepsilon_n + \lambda}{2} - \frac{nt}{2} \Big ) \cdot
\end{eqnarray*}

To control the second term in equation~\eqref{eq:2terms}, we consider the following lemma:
\begin{lemma} \label{lm:2-term}
    For $M=3\sqrt{\log n}$, we have
    $\E \Big [\prod_{i=1}^n  \Big (1+n^2{\mathbf{1}(X_i\notin\Theta_M)}\Big ) \Big ] \le e^{e^{d}}.$
\end{lemma}
By Lemma \ref{lm:2-term}, the second term in \eqref{eq:2terms} can be bounded by
\begin{eqnarray*}
     \mathbb{P} \Big (\prod_{i=1}^n  \Big (1+n^2{\mathbf{1}(X_i\notin\Theta_M)}\Big )  \ge e^{\lambda} \Big ) \le  e^{-\lambda}  \mathbb{E} \Big [\prod_{i=1}^n  \Big (1+n^2{\mathbf{1}(X_i\notin\Theta_M)}\Big )  \Big ] \lesssim e^{-\lambda}.
\end{eqnarray*}
Combining these inequalities, we deduce that for all $t\ge 0$ and $\lambda>0$
\begin{equation} \label{eq:control-probability-t}
    \mathbb{P}(\He^2(\est, f_*)\ge t) \lesssim\exp\Big (C_{\ref{lm:guntu-entropy}}(\log n)^{d+1} +\frac{n\varepsilon_n + \lambda}{2} - \frac{nt}{2} \Big ) +  \exp (-\lambda),
\end{equation}
where the implicit constant only depends on $f_*$ and $d$.
Let
\[ t = \varepsilon_n+(2C_{\ref{lm:guntu-entropy}}+3  )\frac{(\log n)^{d+1}}{n} \quad,\quad 
\lambda = (\log n)^{d+1}\]
and plug into \eqref{eq:control-probability-t}, we get
\[ \mathbb{P}\Big (\He^2(\est, f_*)\ge   \varepsilon_n+(2C_{\ref{lm:guntu-entropy}}+3  )\frac{(\log n)^{d+1}}{n} \Big ) \lesssim \exp \Big ( - (\log n)^{d+1}  \Big ).\]
This yields \eqref{eq:thm-stability-Hellinger-1}
as desired $\square$


\subsubsection{ Proof of Theorem \ref{thm:stability} \ref{it:stabilitylognsqrtn}} \label{subsection:extra-assumption}
 Let $\gamma:= \varepsilon_n + C_1 (\log n)^{d+1}/n$, where $C_1,C_2$ are constants in Theorem \ref{thm:stability} (i). We now consider the event $\mathcal{E} = \{ \He^2(f_*,\est) \le \gamma\}$, which occurs with probability at least $1-C_2\exp(- (\log n)^{d+1})$.  
We have the following lemma:
\begin{lemma} \label{lm:assumption-M-Theta-tau}
    Let $Z$ be the standard Gaussian vector in $\mathbb{R}^d$, and let $r>0$ be such that $\mathbb{P}(\|Z\| \le r ) \ge 0.9$. Let 
    $$\widetilde{\Theta} := \{ x \in \mathbb{R}^d: \dist(x, \supp(\mu_*)) \le 2r \}.$$ 
    Then, if $\He^2(f_*,f) \le \frac{1}{2}(1-e^{-1})^2$, we have $f\in \mathcal{M}(\widetilde{\Theta};\tau)$ for $\tau>0.16$.
    In particular, on the event $\mathcal{E}$, we have $\est \in \mathcal{M}(\widetilde{\Theta};\tau)$. 
\end{lemma}

 
Given Lemma \ref{lm:assumption-M-Theta-tau}, 
we now show \eqref{eq:thm-stability-KL-1}. To this end, we cite here a result in \cite{wong1995probability}, which relates $\He^2$ and $\KL$.

\begin{lemma}[Theorem 5 in \cite{wong1995probability}] \label{prop:shenwong} 
    Let $p,q$ be two densities such that 
    \[\He^2(p,q) \le \gamma \le   \frac{1}{2}(1-e^{-1})^2.\] 
    Suppose that
    \[  \int_{\{p/q \ge e^{1/\delta}\}} p \Big (\frac{p}{q}\Big )^{\delta} \le M_{\delta}< \infty, \quad \text{for some } \delta \in (0,1].\]
    Then there exists $C_{\ref{prop:shenwong}}>0$ (depending on $\delta$ and $M_{\delta}$) such that  
    \[ \KL(p\|q) =\int p \log \Big (\frac{p}{q}\Big ) \le C_{\ref{prop:shenwong}}  \gamma \log (\gamma^{-1}).\]
\end{lemma}

Roughly speaking, Lemma \ref{prop:shenwong} says that if two densities $p,q$ are close in $\He$-distance and satisfy the condition $M_{\delta}$ in Lemma \ref{prop:shenwong}, then $\KL(p\|q)$ is comparable to $\He^2(p,q)$. Our idea is to apply this lemma for $p=f_*$ and $q=\est$.  
The condition $M_{\delta}$ in Lemma \ref{prop:shenwong} will be verified by the following lemma, whose proof is given in Appendix~\ref{prf:appcondition-M-delta}:
\begin{lemma}\label{lm:condition-M-delta}
    There exists $C_{\ref{lm:condition-M-delta}}<\infty$ depending only on $\widetilde{\Theta}, d $ and $\tau$ such that
    \[ \int_{\mathbb{R}^d} f_*(x) \Big ( \frac{f_*(x)}{f(x)} \Big )^{1/2} \de x \le C_{\ref{lm:condition-M-delta}}, \quad \forall f \in \mathcal{M}(\widetilde{\Theta};\tau).\]
\end{lemma}

By Lemma \ref{lm:assumption-M-Theta-tau}, on the event $\mathcal{E}$ we have $\est \in \mathcal{M}(\widetilde{\Theta};\tau)$. Thus, the condition $M_{\delta}$ in Lemma \ref{prop:shenwong} for $f_*=p, \est=q$ and $\delta=1/2$ is satisfied by Lemma \ref{lm:condition-M-delta}.
By our assumption, $\He^2(f_*,\est) \le \gamma \le \frac{1}{2}(1-e^{-1})^2$. Applying Lemma \ref{prop:shenwong} yields:
\[ \KL(f_* \| \est) \lesssim \gamma \log (\gamma^{-1})  \lesssim \varepsilon_n \log (\min \{ \varepsilon_n^{-1}, n \}) + \frac{(\log n)^{d+2}}{n} , \]
here we used that $\log((a+b)^{-1}) \le \log (\min \{a^{-1},b^{-1}\})$. This yields \eqref{eq:thm-stability-KL-1}
$\square$


\subsection{ Stability in the restricted NPMLE: Proof of Theorem \ref{thm:restricted-npmle}}\label{subsubsec:stability_KL_dudley}

We observe that
\begin{eqnarray*}
    \KL(f_* \| \est) &=& - \int \log \Big (\frac{\est(x)}{f_*(x)}\Big) f_*(x)\de x \\
   &=&  -\frac{1}{n} \sum_{i=1}^n \log \Big (\frac{\est(X_i)}{f_*(X_i)} \Big ) 
   + \frac{1}{n} \sum_{i=1}^n \log \Big (\frac{\est(X_i)}{f_*(X_i)} \Big ) - \int \log \Big (\frac{\est(x)}{f_*(x)} \Big ) f_*(x)\de x \\
   &\le& L_n(f_*)- L_n(\est) + \Big |\frac{1}{n} \sum_{i=1}^n \log \Big (\frac{\est(X_i)}{f_*(X_i)} \Big ) - \int \log \Big (\frac{\est(x)}{f_*(x)} \Big ) f_*(x)\de x \Big | \\
   &\le& \varepsilon_n  + \sup_{f\in \mathcal{M}({\Theta};\tau)} \Big |\frac{1}{n} \sum_{i=1}^n \log \Big (\frac{f(X_i)}{f_*(X_i)} \Big ) - \int \log \Big (\frac{f(x)}{f_*(x)} \Big ) f_*(x)\de x \Big | ,
\end{eqnarray*}
where we used the assumption that $\est \in \mathcal{M}(\Theta;\tau)$ in the last line.
In view of Theorem \ref{thm:FiniteBracketingIntegral}, the bracketing entropy of the class $\mathcal{M}(\Theta;\tau)$
is well controlled. Therefore, a standard argument from empirical process theory, using the so-called Dudley's integral, yields the following:

\begin{lemma} \label{lm:dudley-bound}
    There exists a constant $C_{\ref{lm:dudley-bound}}=C_{\ref{lm:dudley-bound}}(d, {\Theta}, \tau)$ such that
    \[\mathbb{E} \Big [\sup_{f\in \mathcal{M}({\Theta};\tau)} \Big |\frac{1}{\sqrt{n}} \sum_{i=1}^n \log \Big (\frac{f(X_i)}{f_*(X_i)} \Big ) - \int \log \Big (\frac{f(x)}{f_*(x)} \Big ) f_*(x)\de x \Big | \Big ] \le C_{\ref{lm:dudley-bound}}.\]
\end{lemma}
The detailed proof of this lemma can be found at Appendix~\ref{prf:appdudley-bound}.
Taking the expecation on both sides and utilizing Lemma \ref{lm:dudley-bound}, the theorem follows $\square$

\subsection{ Moment bounds in NPMLE: Proof of Theorem \ref{thm:L2mu_bound}} \label{subsec:proof-thm-L2mu}

\subsubsection{ A conditioning argument}

We begin with a simple observation that
\begin{equation} \label{eq:sketch-optloss}
   \sqrt{n}( \optloss - L_n(f_*)) = \frac{1}{\sqrt{n}} \sum_{i=1}^n \Big ( \log \npmle(X_i) - \log f_*(X_i) \Big ),
\end{equation}
where $\npmle$ is any NPMLE. 
Note that
the summands in \eqref{eq:sketch-optloss} are strongly dependent due to the dependence of $\npmle$ on the data $X_i$'s. This is essentially the source of the difficulty. 
Our idea is to consider the following subset 
$\mathcal{Q}_n \subset \mathcal{M}$ 
\begin{equation} \label{eq:def-Qn}
    \mathcal{Q}_n:=\Big \{f\in \mathcal{M}: \He^2(f,f_*) \le C_1 \frac{(\log n)^{d+1}}{n} \Big \},
\end{equation}
where $C_1$ is the constant in Theorem \ref{thm:stability} \ref{it:stabilitypolylognovern}. A direct consequence of Theorem \ref{thm:stability} \ref{it:stabilitypolylognovern} with $\varepsilon_n=0$ is that $\mathbb{P}(\npmle \notin \mathcal{Q}_n) \lesssim \exp (-(\log n)^{d+1})$.
This observation, together with Cauchy-Schwarz inequality, implies the following lemma:
\begin{lemma} \label{lm:not-in-qn}
    For $p=1,2$, we have
    \[\E\Big [|\sqrt{n} (\optloss - L_n(f_*)) |^p \mathbf{1}(\npmle \notin \mathcal{Q}_n)\Big ]  = o(1).\]
\end{lemma}
The detailed proof of this lemma can be found at Appendix~\ref{prf:appnot-in-qn}.
We now consider the event $\{\npmle \in \mathcal{Q}_n\}$.
On this event, we have
\begin{equation} \label{eq:inq}
    \sqrt{n}|\optloss - L_n(f_*)|\mathbf{1}(\npmle \in \mathcal{Q}_n)\le \sup_{f\in \mathcal{Q}_n} \Big |\frac{1}{\sqrt{n}} \sum_{i=1}^n \Big ( \log f(X_i) - \log f_*(X_i) \Big ) \Big | .
\end{equation}
Thus, it suffices to show that
\begin{equation} \label{eq:empirical-process}
     \E \Big [ \sup_{f\in \mathcal{Q}_n} \Big |\frac{1}{\sqrt{n}} \sum_{i=1}^n \Big ( \log f(X_i) - \log f_*(X_i) \Big ) \Big |^p \Big ] = o(1), \quad p=1,2.
\end{equation}

\subsubsection{ The case $p=1$}

We first show \eqref{eq:empirical-process} for $p=1$.
For simplicity of notation, let $g_{f}(x):= \log f(x) - \log f_*(x)$ for $f\in \mathcal{M}$. Then
\begin{eqnarray} \label{eq:good-event}
    &&\E \Big [\sup_{f\in \mathcal{Q}_n} \Big |\frac{1}{\sqrt{n}} \sum_{i=1}^n g_{f}(X_i) \Big | \Big ] \nonumber\\
     &\le& \E \Big [ \sup_{f\in \mathcal{Q}_n} \Big |\frac{1}{\sqrt{n}} \sum_{i=1}^n \Big (g_{f}(X_i)  - \mathbb{E} g_{f}(X_i) \Big )\Big | \Big ] + \sqrt{n}\sup_{f\in \mathcal{Q}_n} |\mathbb{E} g_{f}(X_1)| \nonumber\\
      &=& \E \Big [\sup_{f\in \mathcal{Q}_n} \Big |\frac{1}{\sqrt{n}} \sum_{i=1}^n \Big (g_{f}(X_i)  - \mathbb{E} g_{f}(X_i) \Big )\Big | \Big ] + \sqrt{n}\sup_{f \in \mathcal{Q}_n}\KL (f_*\|f) .
\end{eqnarray}
The KL divergence term in \eqref{eq:good-event} is controlled by following technical lemma, whose proof is deferred to Appendix~\ref{prf:appKLdiv-Qn}.
\begin{lemma} \label{lem:KLdiv-Qn}
    There exists a constant $C>0$ depending only on $f_*$ and $d$ such that
    \[ \KL(f_*\|f) \le C \frac{(\log n)^{d+2}}{n}\]
    for all $n$ sufficiently large and for all $f \in \mathcal{Q}_n$.
\end{lemma}
In particular, by Lemma \ref{lem:KLdiv-Qn}, we have $ \sqrt{n}\sup_{f \in \mathcal{Q}_n}\KL (f_*\|f) = o(1)$.
Thus, it suffices to show
 \[\mathbb{E} \Big [ \sup_{f\in \mathcal{Q}_n} \Big |\frac{1}{\sqrt{n}} \sum_{i=1}^n \Big (g_{f}(X_i)  - \mathbb{E} g_{f}(X_i) \Big )\Big | \Big ] = o(1).\]
To this end, we utilize the following classical result in empirical process theory, which is known as Dudley's entropy integral bound:

\begin{proposition}[Corollary 19.35 in \cite{van2000asymptotic}] \label{thm:sup-van2000}
    Let $(\mathcal{X}, \mathcal{A},\mathbf{P})$ be a probability space and $\mathcal{F}$ be a class of $L^2(\mathbf{P})$ functions on $\mathcal{X}$.  Let $X_1,\ldots,X_n$ be i.i.d. samples from $\mathbf{P}$. If $F$ is an envelope function of $\mathcal{F}$, then
    \[ \mathbb{E}  \Big [ \sup_{f\in\mathcal{F}} \Big |\frac{1}{\sqrt{n}}\sum_{i=1}^n \Big (f(X_i) - \mathbb{E}f(X_i)\Big )\Big | \Big ] \lesssim \int_0^{\|F\|_{L^2(\mathbf{P})}} \sqrt{\log \nbr(\varepsilon,\mathcal{F},L^2(\mathbf{P})) }\de \varepsilon,\]
    where $\nbr(\varepsilon,\mathcal{F},L^2(\mathbf{P}))$ denotes the bracketing number of $\mathcal{F}$.
\end{proposition}

Recall that function $F$ is called an \emph{envelope function} of a class $\mathcal{F}$ if $F\ge |f|$ pointwisely for all $f\in \mathcal{F}$. Consider the function class $\mathcal{G}_n:=\{g_f(x)=\log(f(x)/f_*(x)): f \in \mathcal{Q}_n\}$. 
A natural candidate for an envelope function of $\mathcal{G}_n$ is
\begin{equation} \label{eq:envelope}
    G_n(x):= \sup_{f \in \mathcal{Q}_n} |g_{f}(x)|, \quad x\in \mathbb{R}^d.
\end{equation}
The properties of this envelope function are encoded in the following lemma, whose proof is detailed in Appendix~\ref{prf:appenvelopeFunction}:
\begin{lemma} \label{lem:envelopeFunction}
    Let $G_n$ be defined as in \eqref{eq:envelope}. Then $G_n \in L^2(f_*)$, nonnegative, decreasing and $\lim_{n\rightarrow \infty} G_n(x) = 0$ everywhere. In particular, $\|G_n\|_{L^2(f_*)}$ decreases to $0$.
\end{lemma}

By the Dudley's entropy bound (i.e., Proposition \ref{thm:sup-van2000}), we have 
\begin{equation} \label{eq:sup-proc}
    \mathbb{E} \Big [ \sup_{f\in \mathcal{Q}_n} \Big |\frac{1}{\sqrt{n}} \sum_{i=1}^n \Big (g_{f}(X_i)  - \mathbb{E} g_{f}(X_i) \Big )\Big | \Big ] \lesssim \int_0^{\|G_n\|_{L^2(f_*)}} \sqrt{\log \nbr(\varepsilon, \mathcal{G}_n, L^2(f_*))} \de \varepsilon.
\end{equation}
To control the bracketing entropy of $\mathcal{G}_n$, we need the following lemma, argued in Appendix~\ref{prf:appbracketing-entropy}:
\begin{lemma} \label{lm:bracketing-entropy}
    Let $\Theta$ be a compact set such that its interior contains the support of $\mu_*$. Then, for any $\tau\in(0,1)$, there exists $N$ such that $\mathcal{Q}_n \subset \mathcal{M}(\Theta;\tau)$ for all $n\ge N$. 
\end{lemma}

By Lemma \ref{lm:bracketing-entropy}, we would have $\mathcal{Q}_n \subset \mathcal{M}(\Theta;\tau)$ for large enough $n$. Consequently, for large enough $n$
    \[\log \nbr(\varepsilon,\mc{G}_n, {L^2(f_*)}) \le \log \nbr(\varepsilon,\log \mathcal{M}(\Theta;\tau), {L^2(f_*)}) \lesssim \Big (\log\frac{1}{\varepsilon} \Big)^{d+1}. \]
Thus, for large enough $n$
\begin{align} \label{eq:expectationlesslogeps}
    \mathbb{E} \Big [ \sup_{f\in \mathcal{Q}_n} \Big |\frac{1}{\sqrt{n}} \sum_{i=1}^n \Big (g_{f}(X_i)  - \mathbb{E} g_{f}(X_i) \Big )\Big | \Big ] \lesssim \int_0^{\|G_n\|_{L^2(f_*)}}  \bigg(\log\frac{1}{\varepsilon}\bigg)^{d+1} \de \varepsilon.
\end{align}
Since the RHS in \eqref{eq:expectationlesslogeps} is integrable near $0$ and $\|G_n\|_{L^2(f_*)}$ converges to $0$ as $n$ goes to $\infty$ (by Proposition~\ref{lem:envelopeFunction} ), the RHS in \eqref{eq:expectationlesslogeps} is $o(1)$. This proves the claim \eqref{eq:empirical-process} for $p=1$.

\subsubsection{ The case $p=2$}
Now we prove the claim \eqref{eq:empirical-process} for $p=2$. Let
\[Y_n := \sup_{f\in \mathcal{Q}_n} \Big |\frac{1}{\sqrt{n}} \sum_{i=1}^n  g_f(X_i)  \Big |.\]
Note that $\E[Y_n^2]=\E[Y_n]^2 + \var[Y_n]$. Since $\E[Y_n]=o(1)$, to show $\E[Y_n^2] =o(1)$ it suffices to show $\var[Y_n] = o(1)$. Thus, we only need to show
\[ \var \Big [  \sup_{f\in \mathcal{Q}_n} \Big |\frac{1}{\sqrt{n}} \sum_{i=1}^n  g_f(X_i)  \Big |\Big ] = o(1).\]
By Theorem 11.1 in \cite{boucheron2013concentration}, we have
    \begin{eqnarray*}
        \var  \Big [ \sup_{f \in \mathcal{Q}_n} \Big | \frac{1}{\sqrt{n}} \sum_{i=1}^n g_{f}(X_i) \Big |  \Big ] \le \mathbb{E}[\sup_{f\in \mathcal{Q}_n} |g_{f}(X_1)|^2] = \mathbb{E}|G_n(X_1)|^2 = \|G_n\|^2_{L^2(f_*)} = o(1),
    \end{eqnarray*}
    by Lemma \ref{lem:envelopeFunction}. This completes the proof. $\square$


\subsection{ Fluctuations in NPMLE: Proof of Theorem \ref{thm:anti-supercon}} \label{subsec:proof-thm-anticon}
 
Recall that for given data $(x_1,\ldots,x_n)\in \mathbb{R}^{dn}$, we have
\[ \optloss (x_1,\ldots,x_n) = \sup_{f\in \mathcal{M}} \Big \{L_n(f)(x_1,\ldots,x_n) \Big \}= \sup_{f\in\mathcal{M}} \Big \{ \frac{1}{n} \sum_{i=1}^n \log f(x_i) \Big \}.\]
We first show that $\optloss$ is differentiable in $(x_1,\ldots,x_n)$.
Note that while each each $L_n(f)$ is differentiable (in the data $(x_1,\ldots, x_n)$), there is no a-priori reason why their supremum is also differentiable. To this end, we observe in Appendix~\ref{prf:appconvex} that

\begin{lemma}\label{lm:convex}
    The function 
    \[(x_1,\ldots,x_n) \mapsto \optloss(x_1,\ldots,x_n) + \frac{1}{2n}\sum_{i=1}^n \|x_i\|^2\] 
    is convex on $\mathbb{R}^{dn}$. Consequently, $\nabla \optloss$ makes sense a.e. on $\mathbb{R}^{dn}$.
\end{lemma}

Given Lemma \ref{lm:convex}, $\nabla \optloss$ actually makes sense. We now proceed to show the desired inequalities, i.e., show that there exists $C\ge 1$ independent of $n$ such that
\[ C^{-1}\cdot \mathbb{E}[\|\nabla \optloss\|^2] \le \var[\optloss] \le C \cdot  \mathbb{E}[\|\nabla \optloss\|^2].\]
The second inequality automatically follows from the Poincar\'e inequality for the Langevin dynamics. The more interesting and challenging task is the first inequality. 
By Theorem \ref{thm:L2mu_bound}, we have
\[ \Big |\var[\sqrt{n}\optloss]^{1/2} - \var[\sqrt{n}L_n(f_*)]^{1/2} \Big | \le \sqrt{\mathbb{E}[|\sqrt{n}(\optloss-L_n(f_*)|^2]} = o(1).\]
Note that
\[ \var[\sqrt{n}L_n(f_*)] = \var\Big [\frac{1}{\sqrt{n}} \sum_{i=1}^n \log f_*(X_i) \Big ] = \var [\log f_*(X_1)].\]
Hence, $\var[\sqrt{n} \optloss] \asymp 1$. Thus, it suffices to show that $\E[ n\|\nabla \optloss\|^2] = O(1)$. To this end, we consider again the set $\mathcal{Q}_n$ introduced in \eqref{eq:def-Qn}
\begin{equation*} 
    \mathcal{Q}_n=\Big \{f\in \mathcal{M}: \He^2(f,f_*) \le C_1 \frac{(\log n)^{d+1}}{n} \Big \} \cdot
\end{equation*}
Then one can write
\begin{eqnarray} \label{eq:Dirichlet-control}
    \E[ n\|\nabla \optloss\|^2] &=&  \E[ n\|\nabla \optloss\|^2 \mathbf{1}(\npmle \notin \mathcal{Q}_n)] + \E[ n\|\nabla \optloss\|^2 \mathbf{1}(\npmle \in \mathcal{Q}_n)] \notag \\
    &\le& \Big (\E[ n^2\|\nabla \optloss\|^4] \cdot \mathbb{P}(\npmle \notin \mathcal{Q}_n) \Big )^{1/2} + \E[ n\|\nabla \optloss\|^2 \mathbf{1}(\npmle \in \mathcal{Q}_n)].
\end{eqnarray}

The first term in \eqref{eq:Dirichlet-control} can be effectively bounded by the fact that $\mathbb{P}(\npmle \notin \mathcal{Q}_n)$ is exponentially small. In particular, we have the following lemma, proved in Appendix~\ref{prf:appDirichlet-control-1}.

\begin{lemma} \label{lm:Dirichlet-control-1}
    We have $\E[n^2\|\nabla \optloss \|^4] \cdot \mathbb{P}(\npmle \notin \mathcal{Q}_n) = o(1)$.
\end{lemma}

For the second term in \eqref{eq:Dirichlet-control}, we note that
$\optloss = \sup_{f\in \mathcal{Q}_n} L_n(f)$ on the event $\{\npmle \in \mathcal{Q}_n\}$. We have the following simple lemma, whose proof can be found at Appendix~\ref{prf:appderivsup}: 
\begin{lemma} \label{lem:derivsup}
        Let $\mathcal{G}$ be a family of real-valued differentiable functions $g$ on $\mathbb{R}^m$, and $\hat{g}(x) := \sup_{g \in \mathcal{G}} g(x)$. If $\hat{g}$ is differentiable at $x\in\mathbb{R}^m$, then 
        \[ \|\nabla \hat{g}(x)\| \le \sup_{g\in \mathcal{G}} \|\nabla g(x)\|.\]
\end{lemma}
By Lemma \ref{lem:derivsup}, we have
\[ n\|\nabla \optloss\|^2 \mathbf{1}(\npmle \in \mathcal{Q}_n) \le n \sup_{f\in \mathcal{Q}_n} \|\nabla L_n(f)\|^2. \]
For any $f\in \mathcal{Q}_n$, we have $n \|\nabla L_n(f)\|^2 = \frac{1}{n} \sum_{i=1}^n \|\nabla \log f(X_i)\|^2$. To control this term, we will use the following lemma (derived in Appendix~\ref{prf:appDirichlet-control-2}):

\begin{lemma} \label{lm:Dirichlet-control-2}
    There exist constants $C_1,C_2$ (depending only on $f_*$ and $d$) such that for $n$ sufficiently large, we have
    \[ \|\nabla \log f(x) \| \le C_1\|x\| +C_2, \quad \forall x\in \R^d, \forall f\in \mathcal{Q}_n. \]
\end{lemma}

By Lemma \ref{lm:Dirichlet-control-2}, for every $f\in \mathcal{Q}_n$, we have
\begin{eqnarray*}
n    \|\nabla L_n(f)\|^2 = \frac{1}{n} \sum_{i=1}^n \|\nabla \log f(X_i)\|^2 \le \frac{2C_1^2}{n}\sum_{i=1}^n \|X_i\|^2 + 2C_2^2.
\end{eqnarray*}
Therefore
\begin{eqnarray*}
    \E [n\|\nabla \optloss\|^2 \mathbf{1}(\npmle \in \mathcal{Q}_n)] &\le& \E \Big [  \sup_{f\in \mathcal{Q}_n} n \|\nabla L_n(f)\|^2 \Big ] 
    \le \E \Big [\frac{2C_1^2}{n} \sum_{i=1}^n \|X_i\|^2 + 2C_2^2 \Big ] \\
    &=& 2 C_1^2 \mathbb{E}[\|X_1\|^2] + 2 C_2^2 = O(1).
\end{eqnarray*}
Thus $\E [n\|\nabla \optloss\|^2 ] = O(1)$ as desired. This completes the proof. $\square$

\subsection{Bracketing entropy for logarithmic GMM densities: Proof sketch of Theorem~\ref{thm:FiniteBracketingIntegral}}
\label{subsec:log-gmm-bracketing}

 We now provide the main ideas in the bracketing entropy estimate for logarithmic Gaussian mixture densities. For the purposes of Section~\ref{subsec:log-gmm-bracketing} we use $\pi$ as the mixing measure to avoid notational conflicts with statements of the main results.

For ordinary density bracketing, one typically discretizes the mixing distribution
and obtains an approximation \(f_{\pi} \approx f_{\widetilde \pi}\) uniformly or in
a suitable integral norm. This directly gives brackets for the densities. In the
present theorem, however, the empirical-process object is not \(f\), but
\(\log f\), as this is the object appearing in the likelihood and in the
Kullback--Leibler divergence. This introduces a new difficulty: the logarithm is
not uniformly Lipschitz on \((0,\infty)\). Indeed,
\[
    |\log u-\log v|
    \leq \frac{|u-v|}{u\wedge v},
    \qquad u,v>0.
\]
Thus, a small absolute error \(|f_{\pi}-f_{\widetilde \pi}|\) is not sufficient;
the error must be small relative to the local size of the density. This is
problematic because Gaussian mixture densities can be exponentially small on a
large ball, and the log-density is unbounded in the tails.

The proof resolves these two problems separately. First, we choose
\(R \asymp \sqrt{|\log \varepsilon|}\) and split
\(\mathbb R^d = B_R \cup B_R^{\complement}\). On the tail \(B_R^{\complement}\), we do not attempt a
fine approximation. Instead, the condition \(f\in\mathcal M(\Theta;\tau)\)
gives a deterministic envelope
\[
    0 \leq -\log f(x) \leq \|x\|^2 + C,
\]
and the \(L_2(f_*)\)-contribution of this envelope over \(B_R^{\complement}\) is
\(O(\varepsilon)\), since \(f_*\) is a Gaussian mixture with compactly
supported mixing measure. Second, on \(B_R\), we build a finite approximating
class for the log-densities. The non-standard part here is that we split the mixing
measure \(\pi\) into the mass inside \(\Theta\) and the mass outside \(\Theta\).
The component inside \(\Theta\) is discretized while keeping its support in
\(\Theta\); this preserves a uniform lower bound on the approximating density.
The outside component can then be discretized separately in a larger ball.
After recombining the two discretizations, the preserved lower bound allows us
to pass from density approximation to log-density approximation.

We provide a more formal proof sketch below with the full details reserved for Appendix~\ref{app:prfFiniteBracketingIntegral}.

\begin{proof}[Proof sketch of Theorem~\ref{thm:FiniteBracketingIntegral}]
Throughout the proof, unless otherwise mentioned, constants depend only on
\(\Theta,d,\tau\). Let
\[
    \|\Theta\|_\infty := \sup_{\theta\in\Theta}\|\theta\|.
\]

In Lemma~\ref{lm:bound-density-Mc}, proved in Appendix~\ref{prf:appbound-density-Mc}, we record the following elementary envelope bound:
\begin{equation}\label{eq:log-envelope}
0 \le - \log f(x) \le \|x\|^2 +C_{\ref{lm:bound-density-Mc}}, \quad \forall x\in \R^d, \forall f \in \mathcal{M}(\Theta;\tau).    
\end{equation} 

Let \(0<\varepsilon<1\) be small, and set
\[
    R := \sqrt{80|\log \varepsilon|}.
\]
Taking \(\varepsilon\) small enough, we may assume
\(R\geq 2\|\Theta\|_\infty\). Let $B_R$ be the ball of radius $R$ centered at $0$. And let $B_R^{\complement} := \mathbb R^d\setminus B_R.$
We construct brackets separately on \(B_R\) and \(B_R^{\complement}\).

\paragraph{Tail bracket.}
On \(B_R^{\complement}\), the envelope \eqref{eq:log-envelope} gives the universal bracket
\[
    \ell_{\mathrm{out}}(x)
    :=
    -\|x\|^2-C_{\ref{lm:bound-density-Mc}},
    \qquad
    u_{\mathrm{out}}(x)
    :=
    0,
    \qquad x\in B_R^{\complement}.
\]
This bracket contains \(\log f\) on \(B_R^{\complement}\), uniformly over
\(f\in\mathcal M(\Theta;\tau)\). Moreover, if \(X\sim f_*\), then
\(X=Y+Z\), where \(Y\sim\mu_*\) is supported on \(\Theta\) and
\(Z\sim N(0,I_d)\). Hence
\[
    \|X\|\leq \|\Theta\|_\infty+\|Z\|.
\]
Using \(R\geq 2\|\Theta\|_\infty\) and the Gaussian tail bound for \(\|Z\|\),
we obtain
\begin{equation}
\label{eq:tail-control}
    \int_{B_R^{\complement}}
        \bigl(\|x\|^2+C_{\ref{lm:bound-density-Mc}}\bigr)^2 f_*(x)\,dx
    \leq
    C_1 \varepsilon^2 .
\end{equation}

\paragraph{Inside approximation.}
It remains to control the log-densities on \(B_R\). We will push all the difficulties into the following approximation result: 
\begin{lemma} \label{lm:approx}
    Let $\varepsilon > 0$ be small enough such that 
    \[R:= \sqrt{80|\log \varepsilon|} \ge 2  \sup_{\theta\in \Theta} \|\theta\|.\] 
    There exist constants $C_{\ref{lm:approx}}, D_{\ref{lm:approx}}$ depending only on $\Theta, d$ and $c$, and a class of functions $\mathcal{H}$ with 
    \[\log |\mathcal{H}| \le D_{\ref{lm:approx}} |\log \varepsilon|^{d+1}\] 
    satisfying the following property: for every density $f \in \mathcal{M}(\Theta;c)$, there exists $h\in \mathcal{H}$ such that
    \[ \sup_{x\in B_R} |\log f(x) - h(x)| \le C_{\ref{lm:approx}} \varepsilon.\]
\end{lemma}

Let us now complete the proof of the theorem assuming Lemma~\ref{lm:approx}. Since the proof of Lemma~\ref{lm:approx} contains non-standard technical ideas, we will give a proof sketch after completing the proof sketch of Theorem~\ref{thm:FiniteBracketingIntegral}.


Assuming Lemma~\ref{lm:approx}, we complete the proof of the
theorem. For each \(h\in\mathcal H\), define functions \(\ell_h,u_h\) on
\(\mathbb R^d\) by
\[
    \ell_h(x)
    :=
    \begin{cases}
        h(x)-C_{\ref{lm:approx}}\varepsilon, & x\in B_R,\\
        -\|x\|^2-C_{\ref{lm:bound-density-Mc}}, & x\in B_R^{\complement},
    \end{cases}
    \qquad
    u_h(x)
    :=
    \begin{cases}
        h(x)+C_{\ref{lm:approx}}\varepsilon, & x\in B_R,\\
        0, & x\in B_R^{\complement}.
    \end{cases}
\]
By Lemma~\ref{lm:approx}, the pair \([\ell_h,u_h]\) brackets
\(\log f\) on \(B_R\) for a suitable \(h\). By \eqref{eq:log-envelope}, the same
pair brackets \(\log f\) on \(B_R^{\complement}\). Hence
\[
    \bigl\{[\ell_h,u_h]: h\in\mathcal H\bigr\}
\]
covers \(\log\mathcal M(\Theta;\tau)\).

It remains to bound the \(L_2(f_*)\)-width of these brackets. By
\eqref{eq:tail-control},
\[
\begin{aligned}
    \|\ell_h-u_h\|_{L_2(f_*)}^2
    &=
    \int_{B_R} |\ell_h(x)-u_h(x)|^2 f_*(x)\,dx
    +
    \int_{B_R^{\complement}} |\ell_h(x)-u_h(x)|^2 f_*(x)\,dx \\
    &\leq
    4C_{\ref{lm:approx}}^2\varepsilon^2
    +
    C_1\varepsilon^2
    =
     C^2 \varepsilon^2.
\end{aligned}
\]
Thus the above collection is a family of
\( C\varepsilon\)-brackets. Since
\(\log|\mathcal H|\leq D_3|\log\varepsilon|^{d+1}\), we have
\[
    \log N_{[]}
    \bigl(
         C\varepsilon,
        \log\mathcal M(\Theta;\tau),
        L_2(f_*)
    \bigr)
    \leq
    D_3|\log\varepsilon|^{d+1}.
\]
Rescaling \(\varepsilon\) and adjusting constants gives the desired bound.
\end{proof}

\subsubsection{Main idea of Lemma~\ref{lm:approx}}

Finally to complete the proof we consider the following proof sketch of Lemma~\ref{lm:approx}.  Below we use $C$ to be a constant which can change from line to line but depends only on $\Theta, d, \tau$.  The full details (including tracked constants) are in Appendix~\ref{prf:applmapprox}.
\begin{proof}[Proof Sketch of Lemma~\ref{lm:approx}]
We explain the construction, since this is the part where the logarithmic case
differs from ordinary density bracketing. Write \(f=f_\pi\), where \(\pi\) is
the mixing measure. A naive discretization of \(\pi\) gives an approximation
\(f_{\widetilde \pi}\) to \(f_\pi\), but this is not enough for log-densities:
we need
\[
    |f_\pi(x)-f_{\widetilde\pi}(x)|
    \ll
    f_\pi(x)\wedge f_{\widetilde\pi}(x)
\]
uniformly over \(x\in B_R\).

The condition \(f\in\mathcal M(\Theta;\tau)\) gives a way to preserve such a
lower bound. Decompose the mixing measure into two parts,
\[
    \pi
    =
    \pi(\Theta)\pi_{\Theta}
    +
    \pi(\Theta^c)\pi_{\Theta^c},
\]
where \(\pi_\Theta\) and \(\pi_{\Theta^c}\) denote the normalized restrictions
whenever the corresponding masses are nonzero. We discretize these two pieces
separately.

First, because \(\pi_\Theta\) is supported on the compact set \(\Theta\), in Proposition~\ref{prop:discretizationcompact} we construct a discrete approximation \(\widetilde\pi_1\) that is still supported
inside \(\Theta\), has at most \(C|\log\varepsilon|^d\) atoms, and satisfies
\[
    \sup_{x\in B_R}
    |f_{\pi_\Theta}(x)-f_{\widetilde\pi_1}(x)|
    \leq
    C\varepsilon^{81}.
\]
Keeping the support inside \(\Theta\) is crucial. It guarantees that, after
recombination, the approximating mixture still has at least \(\tau\) mass on
\(\Theta\), and therefore inherits the same type of lower bound as \(f\).

Second in Proposition~\ref{prop:discretizationnoncompact}, we discretize the outside part \(\pi_{\Theta^c}\) by a standard Gaussian mixture approximation on an enlarged ball. This gives a discrete approximation
\(\widetilde\pi_2\), again with at most \(C|\log\varepsilon|^d\) atoms, such that
\[
    \sup_{x\in B_R}
    |f_{\pi_{\Theta^c}}(x)-f_{\widetilde\pi_2}(x)|
    \leq
    C\varepsilon^{81}.
\]
Now define the recombined discrete measure
\[
    \widetilde\pi
    :=
    \pi(\Theta)\widetilde\pi_1
    +
    \pi(\Theta^c)\widetilde\pi_2 .
\]
Because \(\pi(\Theta)\geq \tau\) and \(\widetilde\pi_1\) is supported in
\(\Theta\), we still have \(\widetilde\pi(\Theta)\geq \tau\). Consequently,
for \(x\in B_R\),
\[
    f_{\widetilde\pi}(x)
    \geq
    c\exp(-R^2)
    \geq
    c\varepsilon^{80}.
\]
The same lower bound holds for \(f_\pi\). Combining this with the
\(\varepsilon^{81}\)-level density approximation and the inequality
\[
    |\log u-\log v|
    \leq
    \frac{|u-v|}{u\wedge v},
\]
we obtain
\[
    \sup_{x\in B_R}
    |\log f_\pi(x)-\log f_{\widetilde\pi}(x)|
    \leq
    C\varepsilon .
\]

Finally, the discrete measures \(\widetilde\pi\) have only
\(O(|\log\varepsilon|^d)\) atoms and are supported in a ball whose radius is a
fixed multiple of \(R\). We place a sufficiently fine lattice on the atom
locations and weights. The resulting finite collection of discrete mixtures
induces a finite class \(\mathcal H\) of logarithmic densities. The lattice mesh
is chosen so that replacing \(\widetilde\pi\) by the nearest lattice-supported
measure changes the log-density on \(B_R\) by at most \(C\varepsilon\). The
number of possible atoms and weights is bounded by
\[
    \log |\mathcal H|
    \leq
    C|\log\varepsilon|^{d+1}.
\]
This proves Lemma~\ref{lm:approx}.
\end{proof}


\begin{acks}
We would like to thank Jingbin Pan for supportive discussions at a preliminary stage of this project.
\end{acks}

\begin{funding}
SG was supported in part by the Singapore MOE grants R-146-000-312-114, A-8002014-00-00, A-8003802-00-00, E-146-00-0037-01 and A-8000051-00-00. AG was supported by the US National Science Foundation (NSF) grant DMS-2515470. HST was supported by the NUS  grant A-8003576-00-00.
\end{funding}

\bibliographystyle{plain} 
\bibliography{Arxiv/bibpapr}       


\newpage




\begin{appendix}

\section{{ Stability in random optimization problems: a statistical mechanics perspective}}
\label{sec:stat-mech}
In this section, we undertake a detailed discussion of random optimization problems in statistical mechanics and the toolbox developed therein to analyze these problems. This, in turn, motivates our approach in the present paper.

\subsection{ Physical systems in random environments}

In statistical physics, the environment in which a system exists may not be deterministic but instead may exhibit randomness or disorder. Such systems are typically referred to as  \emph{disordered systems}. Observables of disordered systems (e.g., energy) can thus be viewed as random functions, due to this randomness of the environment. Consequently, the optimization of random functions is of broad interest in the statistical mechanics literature and has been extensively studied in various models, such as Gaussian random polymers, percolation models, extrema of Gaussian fields, and others.

A simple example of disordered systems is the {\it $(1+1)$-Gaussian random polymers}, through the lens of which we are going to discuss the ideas and methods for random optimization problems in statistical mechanics. In this model, the \emph{configuration space} consists of all one dimensional {\it paths} of length $n$ starting at $0$, i.e. 
\[\mathcal{P}_n := \Big \{p=((i,p(i))_{0 \le i \le n} : p(i) \in \mathbb{Z} \; \forall i, \; p(0)=0, |p(i+1)-p(i)|\le 1, \: 0\le i \le n-1 \Big \}.\]
The random \emph{environment} is given by $(G_{v})_{v\in \mathbb{Z}^2}$, a family of i.i.d. standard Gaussians indexed by the vertices of the integer lattice $\mathbb{Z}^2$. Given a path $p\in \mathcal{P}_n$, the \emph{energy of the path} is defined by 
\begin{equation}
    E_n(p):= -\sum_{v\in p} G_v.
\end{equation}
The configuration which minimizes the energy of the system is called a \emph{ground state polymer}
\begin{equation}
    \hat{p}_n \in \argmin_{p\in \mathcal{P}_n} E_n(p),
\end{equation}
and the minimized energy is called \emph{ground state energy}
\begin{equation}
     \hat{E}_n := \min_{p\in \mathcal{P}_n}E_n(p) .
\end{equation}

Note that the minimization of energy as defined above is equivalent to maximization of the sum of Gaussian weights (from the random environment) along a path, and the path varies over a set of possible choices (which, in statistical mechanical models, is usually combinatorial in nature). Thus, minimization of energy, a natural concept in physics, is equivalent in this setting to maximization of overall weight of an object, which is more canonical in statistics and machine learning.

To draw the parallel between NPMLE and disordered systems, we view $\mathcal{M}$, the space of all GMM densities, as the configuration space. Given a `configuration' $f\in \mathcal{M}$, the proxy for the energy is the negative log-likelihood function
\[ -L_n(f) = -\frac{1}{n} \sum_{i=1}^n \log f(X_i),\]
where $X_1,\ldots,X_n$ can be viewed as the environment of the system. The ground state energy, by definition, is the minimized energy over all configuration space
\begin{equation} \label{eq:npmle-optimization}
    -\optloss = \min_{f\in \mathcal{M}}\{-L_n(f)\} = - \max_{f\in\mathcal{M}}L_n(f)
\end{equation}
and the ground states are minimizers of \eqref{eq:npmle-optimization}, i.e., the (non-parametric) Maximum Likelihood Estimator $\npmle$.

\subsection{ Perturbation of environments and Langevin dynamics}

In the study of disordered systems, it is natural to investigate environments that undergo perturbations, and study the response of the solution of random optimization problems in response to such perturbations. The parallels with machine learning are clear; it is important to understand the response of algorithmic procedures to perturbations in their inputs.

A canonical approach to perturbations in statistical mechanics is to evolve the random environment for a small time via an invariant dynamics. This serves a two-fold purpose -- on one hand, its leaves the distribution of the random environment unchanged (so, from a statistical point of view, the model remains invariant); but on the other hand, it couples the time-evolved environment with the original one pathwise as a small perturbation. 

To give an example, we go back to the Gaussian polymer model. Here, the environment $(G_v)_{v \in \mathbb{Z}^2}$ can be made to evolve according to the Ornstein-Uhlenbeck ({\it abbrv.} OU) dynamics
\begin{equation} \label{eq:OU}
    \de G_v^{(t)} = - G_v^{(t)} \de t + \sqrt{2} \de B_v^{(t)},
\end{equation}
where $(B_v^{(t)})_{t\ge 0}$ is a standard Brownian motion. It is well-known that the standard Gaussian distribution is the equilibrium measure of the OU diffusion. Thus, for any time $t \ge 0$, $(G_v^{(t)})_{v\in \mathbb{Z}^2}$ are i.i.d. standard Gaussians. In fact, far beyond the Gaussian polymer, this is a standard invariant dynamics that may be used to perturb the environment in any stat mech model where the disorder constitutes independent Gaussians.

In the setting of NPMLE, the data $X_1,\ldots,X_n$ are i.i.d. samples from the density $f_*$. The canonical way to evolve $X_i$'s is via the so-called \emph{Langevin dynamics}
\begin{equation} \label{eq:Langevin}
    \de X^{(t)}_i = \nabla \log f_*(X^{(t)}_i) \de t + \sqrt{2} \de B^{(t)}_i,\quad 1\le i \le n,
\end{equation}
where $B^{(t)}_i$ are independent copies of the standard Brownian motion in $\R^d$ and $X^{(0)}_i=X_i \; \forall i$. The key feature of the Langevin dynamics \eqref{eq:Langevin} is that it preserves the distribution $f_*(x)\de x$. Namely, for any given time $t\ge 0$, we have $X_1^{(t)},\ldots,X_n^{(t)}$ are still i.i.d. samples from $f_*$. Thus, for small time $t$, the samples $(X_i^{(t)})_{i=1}^n$ can be regarded to be another i.i.d. sample of size $n$ from the distribution $f_*(x)\de x$, while being coupled to be pathwise close with the $(X_i)_{i=1}^n$.

\subsection{ Poincar\'e inequality and Dirichlet forms}

 A powerful tool to study Markov processes is the notion of Dirichlet forms. We include a brief description herein, mostly via examples, incorporating a more detailed discussion in Appendix \ref{appB}. Let $L$ be the generator of a Markov process with equilibrium measure $\nu$. Then the Dirichlet form of the process is then defined by
\begin{equation} \mathcal{E}(g,h) = -\int_S g(x)Lh(x) \nu(\de x). \end{equation}
The process is said to satisfy a Poincar\'e inequality (PI) with constant $C$ if 
\begin{equation} \label{eq:Poincare}
\var_{\nu}(g) \le C \mathcal{E} (g,g),
\end{equation}
for all $g\in L^2(\nu)$ such that the Dirichlet energy $\mathcal{E}(g,g)$ is well-defined. The smallest such a constant $C$ is called the optimal constant in the Poincar\'e inequality, denoted by $C_{\mathrm{PI}}$.

To give an example, in the Gaussian polymer model, the invariant dynamics on the random environment of i.i.d. Gaussians is given by the classic Ornstein-Uhlenbeck flow \eqref{eq:OU}. The generator of the process $(G_v^{(t)}: v\in V_n)_{t\ge 0}$ is 
\[L=  \Delta - x\cdot \nabla, \quad x \in \R^{|V_n|}\] 
where $\Delta$ is the Laplacian and $\nabla$ is the gradient operator in $\mathbb{R}^{|V_n|}$. The equilibrium measure of this diffusion is the standard Gaussian measure $\gamma_n$ in $\mathbb{R}^{|V_n|}$, and the Dirichlet energy is given by
\begin{equation*}
    \mathcal{E}_{\gamma_n} (g,h) =  \int_{\mathbb{R}^{|V_n|}} \nabla g (x)\cdot \nabla h(x) \gamma_n (\de x).
\end{equation*}
It is known that the OU dynamics satisfies the Poincar\'e inequality \eqref{eq:Poincare} with $C_{\mathrm{PI}}=1$.

In the NPMLE model, it is natural to let each $X_i$ evolve independently via the Langevin dynamics \eqref{eq:Langevin}. A simple computation shows that the generator of the process $(X_1^{(t)},\ldots,X_n^{(t)})_{t\ge 0}$ is 
\begin{equation} \label{eq:generator} L = \Delta + \sum_{i=1}^n \nabla_{x_i}\log f_{*}(x_i) \cdot \nabla_{x_i}, \end{equation}
where $\Delta$ is the Laplacian in $\mathbb{R}^{dn}$ and $\nabla_{x_i}$ is the gradient acting on the $x_i$-coordinate. The equilibrium measure of this diffusion is $f_*(x_1)\ldots f_*(x_n)\de x_1\ldots \de x_n$. Using integration by parts, the Dirichlet energy in this case is given by
\begin{equation} \label{eq:Dirichlet-energy}
     \mathcal{E}_{n}(g,h) = \int_{\mathbb{R}^{dn}} \nabla g(x_1,\ldots,x_n) \cdot \nabla h(x_1,\ldots,x_n) f_*(x_1)\ldots f_*(x_n)\de x_1\ldots \de x_n.
\end{equation}


\subsection{ Concentration and chaos in disordered systems}
\label{sec:conc-chaos}
A fundamental question in statistical mechanics is to study the sensitivity of a physical system (e.g. ground state energy, ground state, etc) to changes in the underlying random environments. In the seminal monograph \cite{chatterjee2014superconcentration}, Chatterjee investigated the concepts of \textit{superconcentration, chaos} and \textit{multiple valleys} describe the instability of disordered systems.

We discuss these notions below.
\begin{definition}\label{def:chaos}
    Let $X_t$ be a reversible Markov process with semigroup $P_t$, generator $L$, equilibrium measure $\nu$ and Dirichlet form $\mathcal{E}$. Let $\varepsilon, \delta >0$ and $f$ be a function on the state space of the Markov process. We say that:
    \begin{enumerate}[label=(\roman*)]
        \item (Superconcentration) $f$ is $\varepsilon$-superconcentrated if 
    \[ \var_{\nu}(f) \le \varepsilon \cdot C_{\mathrm{PI}} \mathcal{E}(f,f).\]
        \item (Chaos) $f$ is $(\varepsilon,\delta)$-chaotic if for all $t \ge \delta$
        \[ \mathcal{E}(f,P_tf) \le \varepsilon \cdot e^{-t/C_{\mathrm{PI}}} \mathcal{E}(f,f).\]
    \end{enumerate}
\end{definition}

To understand the notions of superconcentration and chaos in Definition \ref{def:chaos}, one should think of $X_t$ as the environment of a physical system at time $t$, $f$ is an observable of interest (e.g., the ground state energy), and $\varepsilon, \delta$ are small parameters that tend to $0$ as the system size increases. In that case, $f$ will be simply called superconcentrated (respectively, chaotic). 

The heuristics behind the definition of superconcentration can be understood as follows: if the semigroup $P_t$ of $X_t$ satisfies a Poincar\'e inequality with constant $C_{\mathrm{PI}}$, one should have $\var_{\nu}(f) \le C_{\mathrm{PI}} \mathcal{E}(f,f)$ for all observables $f$ of finite Dirichlet energy. Superconcentration refers to a tighter fluctuation (the existence of the small parameter $\varepsilon$) compared to the usual order obtained from the classical Poincar\'e inequality, in the specific case of the superconcentrated function f. 

The notion of chaos, as outlined in Definition \ref{def:chaos} (ii), is analytical in nature, and perhaps more convenient to handle in a technical sense. Its connections to the more intuitive notion of chaos, namely lack of stability of output in response to minor perturbations in the input, might not be superficially obvious. However, it does turn out that in many classical models of statistical mechanics,  Definition \ref{def:chaos} (ii) does imply a suitable phenomenon that is aligned to the more classical notion of chaos outlined here. For an illuminating discussion through a concrete example, we refer to Section \ref{sec:stability-GP}; for a version that is relevant to the NPMLE problem at hand, we direct the reader to Corollary \ref{cor:robust} and the related discussion in Section \ref{sec:stability-NPMLE}.  

These concepts are perhaps the most relevant in the setting where we have a sequence of models of growing size, parametrized by the quantity $n \to \infty$, and a sequence of real-valued functions $f_n$ defined on the configuration space $\Theta_n$ of the $n$-th model (for instance, consider the ground state energy as $f_n$ in the Gaussian polymer model of size $n$, which is defined on $\Theta_n$ that is the set of walks of length $n$). Then the sequence of functions $f_n$ superconcentrates if Definition \ref{def:chaos} (i) holds for $f_n$ with a corresponding sequence $\varepsilon_n \to 0$, and is chaotic  when Definition \ref{def:chaos} (ii) holds with parameters $(\tilde{\varepsilon}_n,\delta_n)$ with both $\tilde{\varepsilon}_n,\delta_n \to 0$.  

\subsection{ The multiple valleys and AEU phenomena}
\label{sec:multiple valleys}
For the notion of \emph{multiple valleys}, we will only discuss heuristics and refer the reader to \cite{chatterjee2014superconcentration} Chap. 1 Section 3 (esp. Subsections 3.1 and 3.2) for the formal definition. An optimization problem is called \textit{stable} if any near-optimum is close to the optimum in an appropriate metric. The \textit{multiple valleys} phenomenon is in some sense the opposite, namely, a random optimization problem is said to have multiple valleys if there are many significantly  dissimilar near-optimal solutions. The  monograph \cite{chatterjee2014superconcentration} demonstrates that superconcentration and chaos, as defined above are equivalent, and they imply the multiple valleys/multiple peaks phenomenon.

A sequential understanding for the multiple valleys/multiple peaks phenomenon, similar to concentration and chaos, can be easily obtained. However, as will be clear from the results in this paper, it is in fact the contrapositive of the multiple valleys/multiple peaks phenomenon that is conceptually relevant for our purposes. In the sequential setting of models of growing size, this may be understood as the so-called \emph{Asymptotic Essential Uniqueness} ({\it {abbrv.} AEU}) property. An early rigorous instance of this phenomenon was established by Aldous in the context of the so-called \emph{random assignment problem} \cite{aldous2001zeta} wherein he proved a well-known conjecture of Parisi \cite{parisi1998conjecture}; see also the discussion in Chap. 1 Section 3.1 in \cite{chatterjee2014superconcentration}. 

To understand this, we think of $f_n$ as discussed above as the optimal value of a random optimization problem; i.e.  $f_n=\min_{\theta \in \Theta_n} F_n(\theta)$ for functions $F_n$ defined on the configuration space $\Theta_n$, with the optimizing configuration being $\hat{\theta}_n$ (which is random, due to the random environment). Broadly speaking, AEU entails that for any (random) sequence of arguments $\theta_n \in \Theta_n$ such that $\mathbb{E}[s_n(\theta_n,\hat{\theta}_n)] \ll 1$ for a suitable \emph{measure of similarity} $s_n$ on the configuration space $\Theta_n$, we must have $\mathbb{E}[F_n(\theta_n)] \gg \mathbb{E}[f_n]$. Here the comparison inequalities are appropriately defined with due regard to the scaling of the quantities such as $\mathbb{E}[f_n]$ with $n$.

\subsection{ An example: Instability in the Gaussian polymer model}
\label{sec:stability-GP}
To shed light on these discussions, we once again look at the example of the $(1+1)$-dimensional Gaussian random polymer model. In this section, we will briefly address the main ideas, relegating a detailed discussion to Appendix \ref{appC}. 

To this end, let the Gaussian environment $(G_v)_{v\in \mathbb{Z}^2}$ in the $(1+1)$-dimensional Gaussian random polymer model evolve according to the Ornstein-Uhlenbeck diffusion \eqref{eq:OU}. 
We denote by $\hat{p}_n^{(t)}$ the ground state polymer of the system under the $t$-perturbed environment $(G_v^{(t)})_{v\in \mathbb{Z}^2}$.

In \cite{chatterjee2014superconcentration}, it was demonstrated that the ground state energy $\hat{E}_n$ is superconcentrated; in fact, it may be shown that $\var(\hat{E}_n) \le C n/ \log n$. To see how this bound connects to the Poincar\'e inequality, we note that in this model the optimal energy $\hat{E}_n$ satisfies $\hat{E}_n = \sum_{v\in \text{optimal path}} G_v$; therefore $\partial \hat{E}_n/\partial G_v = \mathbf{1}\{v\in \text{optimal path}\}$. This immediately leads to 
\[\|\nabla \hat{E}_n\|^2 = \sum_{v} \mathbf{1}\{v\in \text{optimal path}\} \le n+1,\]
which implies superconcentration of $\hat{E}_n$ with parameter $1/\log n$. For a detailed version of this analysis, we refer the reader to the Appendix \ref{appC}.

By the equivalence of superconcentration and chaos (established in \cite{chatterjee2014superconcentration} under rather general conditions in the setting of discrete statistical mechanical models), it may be deduced that that $\hat{E}_n$ is also chaotic (in the sense of  Definition \ref{def:chaos}). While the notion of chaos as in Definition \ref{def:chaos} is somewhat less intuitive due to the emergence of the Dirichlet form,  we remark that the Dirichlet form can be thought as a way to measure correlations in many probabilistic models, whereby we may interpret $\mathcal{E}(f,P_tf)$ as a kind of "correlation" between  $f(X)$ and $P_tf(X)$, where $X\sim \nu$. For a detailed discussion of this perspective with further examples, we refer the reader to Appendix \ref{appC}.

For the Gaussian polymer model,  it was shown in \cite{chatterjee2014superconcentration}  that chaos of the ground state energy $\hat{E}_n$ (in the sense of Definition \ref{def:chaos}) implies that there exist small times $t_n \to 0$ such that the ground state polymers $\hat{p}_n$ and $\hat{p}_n^{(t_n)}$ (at times 0 at $t_n$ resp.) are almost disjoint (relative to their lengths) even when $t_n$ is arbitrarily small (in the sense that $\mathbb{E}[\frac{1}{n}|\hat{p}_n \cap \hat{p}_n^{(t_n)}|] \to 0$ (where $|p\cap p'|$ denotes the number of vertices in the intersection of the two paths $p,p'$). This could be seen as a manifestation of the usual notion of chaos in statistical physics: the ground state polymers are highly sensitive to small changes in the environment. 

For the multiple valleys phenomenon, we recall that the ground state energy $\hat{E}_n$ is obtained by optimizing the energy functional $E_n$, which is a random function of the environment. Since the ground state energy $\hat{E}_n$ has been shown to be superconcentrated and chaotic, it follows from the theory established in \cite{chatterjee2014superconcentration}, that the multiple valleys phenomenon also occurs in this setting. This is understood most naturally via the sequential interpretation as in Section \ref{sec:multiple valleys}, with the similarity measure $s_n(p,p'):= |p\cap p'|/n$  between two paths $p$ and $p'$. It entails that, there is a large collection of {\it near-optimal} paths $p$ such that $|p\cap \hat{p}_n|/n$ is typically small, but nonetheless $E_n(p)$ is close to $\hat{E}_n$ at an appropriate scale for all such $p$ in this collection.


\subsection{ Bhattacharyya Coefficient as similarity measure and proof of Corollary \ref{cor:robust}} 
\label{sec:BC}

In this section, we discuss why the Bhattacharyya Coefficient ({\it abbrv} BC) is a natural measure of similarity of probability distributions for studying chaos phenomena in the context of statistical and machine learning problems, and complete the proof of Corollary \ref{cor:robust}.
To make this analogy precise, we examine chaos phenomena in statistical mechanics through the motivating example of the Gaussian polymer model, and identify a path $p$ with the probability
density of the normalized counting measure
\[
\mathrm{1}_p(x) := \frac{\indicator\{x \in p\}}{|p|},
\]
where $|p|$ denotes the number of vertices in the path. Then, for two
paths $p$ and $q$,
\[
\mathrm{BC}(\mathrm{1}_p,\mathrm{1}_q)
= \int \sqrt{\frac{\indicator\{x \in p\}\indicator\{x \in q\}}{|p||q|}}
\,\de x.
\]

Since all polymer paths have the same length, the above expression shows that $\mathrm{BC}(\mathrm{1}_p,\mathrm{1}_q)$ coincides exactly with the normalized intersection size $|p \cap q|$.

We complete this section with a direct proof of Corollary \ref{cor:robust}.

\begin{proof}[Proof of Corollary \ref{cor:robust}]
     Noting that
\[
\mathrm{BC}(\npmle,\npmle^{(t)}) = 1 - \frac{1}{2}\He^2(\npmle,\npmle^{(t)}),
\]
it therefore suffices to show that
$\He^2(\npmle,\npmle^{(t)}) \to 0$. 
     As $X_1, ..., X_n$ are i.i.d. samples from $f_{*}$ we have via \cite{zhang2009generalized, saha2020nonparametric} that 
     \begin{align}\mathbb{E}\left[\He^2(\npmle, f_{*})\right] \leq O\left(\frac{(\log n)^{d+1}}{n}\right).\end{align} 
     Since the perturbed samples $X_1^{(t)}, ..., X_n^{(t)}$ are also i.i.d. samples from the density $f_{*}$, the same result applies to $\npmle^{(t)}$. 
     Using sub-additivity of $\He$ gives that  
     \[\E\left[\He^2(\npmle, \npmle^{(t)})\right] \leq O \Big (\frac{(\log n)^{d+1}}{n} \Big ),\]
     as desired.
\end{proof}


\section{{ Basics of Dirichlet forms and Poincar\'e inequality}}\label{appB}

We briefly introduce the basics of the theory of Dirichlet forms here, with a view to providing a more or less self-contained account. For an in-depth account, we refer the interested reader to the excellent texts \cite{bakry2013, fukushima2010}, and the references therein. 
 
 Let $(X_t)_{t \ge 0}$ be a continuous-time Markov process taking values in some abstract state space $S$. By definition, the Markov semigroup $(P_t)_{t\ge 0}$ of the process is a family of linear operators acting on functions on $S$, given by
\[ P_tg(x) = \mathbb{E} [g(X_t) | X_0 =  x], \quad g:S\rightarrow \mathbb{R} \text{ and } x\in S.\]
As the name suggests, they satisfy the semi-group property: $P_0=\mathrm{Id}$ and $P_t\circ P_s = P_{t+s}$ for all $t,s \ge 0$.
The generator $L$ of the process is formally given by 
\[ L g = \lim_{t\rightarrow 0} \frac{P_t g - g}{t},\]
whenever the right-hand side makes sense. 

Let $\nu$ be a measure on $S$. We say that $\nu$ is \emph{invariant} under $(P_t)_{t\ge 0}$ if 
\[ \int f \de \nu = \int P_t f \de \nu , \quad \forall f \in L^1(\nu).\]
We further say that $\nu$ is \emph{reversible} if 
\[ \int (P_tf) g \de \nu = \int f P_t g \de \nu, \quad \forall f, g \in L^2(\nu).\]
If in addition $\nu$ is a probability measure, we also call $\nu$ an equilibrium measure of the Markov process.

Given an equilibrium measure $\nu$, the \emph{Dirichlet form} of the process is then defined by
\[ \mathcal{E}(f,g) = -\int_S f(x)Lg(x) \nu(\de x). \]
The process is said to satisfy the Poincar\'e inequality with constant $C$ if 
\[\var_{\nu}(g):= \int\Big (g-\int g\de \nu \Big )^2 \de \nu \le C \mathcal{E} (g,g),\]
for all $g\in L^2(\nu)$ such that the Dirichlet energy $\mathcal{E}(g,g)$ is well-defined. The smallest such a constant $C$ is called the optimal constant in the Poincar\'e inequality, denoted by $C_{\mathrm{PI}}$.

We now look at some examples.
In the Gaussian polymer model, we first notice that the system (particularly, the ground state energy $\hat{E}_n$ and the ground state polymer $\hat{p}_n$) only depends on the environment $G_v$ on the set
\[V_n:=\{(i,j)\in \mathbb{Z}_2: 0\le i \le n, -n\le j \le n\}.\]
Thus, it suffices to consider $(G_v: v\in V_n)$ as the environment of the system. For each $v\in V_n$, $G_v$ evolves independently according to the Ornstein-Uhlenbeck equation \eqref{eq:OU}. Hence, the generator of the process $(G_v^{(t)}: v\in V_n)_{t\ge 0}$ is 
\[L=  \Delta - x\cdot \nabla, \quad x \in \R^{|V_n|}\] 
where $\Delta$ is the Laplacian and $\nabla$ is the gradient operator in $\mathbb{R}^{|V_n|}$. The equilibrium measure of this diffusion is the standard Gaussian measure $\gamma_n$ in $\mathbb{R}^{|V_n|}$, and the Dirichlet energy is given by
\begin{eqnarray*}
    \mathcal{E}_{\gamma_n} (g,h) &=& -\int_{\mathbb{R}^{|V_n|}} g(x) ( \Delta h(x) - x\cdot \nabla h(x)) \gamma_n (\de x) \\
    &=& \int_{\mathbb{R}^{|V_n|}} \nabla g (x)\cdot \nabla h(x) \gamma_n (\de x),
\end{eqnarray*}
where the last equality follows by integration by parts. It is known that the OU semigroup satisfies Poincar\'e inequality with $C_{\mathrm{PI}}=1$.

In NPMLE, each $X_i$ evolves independently via the Langevin equation \eqref{eq:Langevin}. Thus, the generator of the process $(X_1^{(t)},\ldots,X_n^{(t)})_{t\ge 0}$ is 
\[ L = \Delta + \sum_{i=1}^n \nabla_{x_i}\log f_{*}(x_i) \cdot \nabla_{x_i}, \]
where $\Delta$ is the Laplacian in $\mathbb{R}^{dn}$ and $\nabla_{x_i}$ is the gradient acting on the $x_i$-coordinate. The equilibrium measure of the diffusion is $f_*(x_1)\ldots f_*(x_n)\de x_1\ldots \de x_n$. Using integration by parts, the Dirichlet energy in this case is given by
\begin{equation} \label{eq:Dirichlet-energy}
     \mathcal{E}_{n}(g,h) = \int_{\mathbb{R}^{dn}} \nabla g(x_1,\ldots,x_n) \cdot \nabla h(x_1,\ldots,x_n)\prod_{i=1}^n f_*(x_i)\de x_i.
\end{equation}

Since the equilibrium measure is of the form $\prod_{i=1}^n f_*(x_i)\de x_i$, it follows from the tensorization property of the Poincar\'e inequality that the Poincar\'e constant of $\prod_{i=1}^n f_*(x_i)\de x_i$ is the same as the one of $f_*(x)\de x$. In \cite{Bardet18}, it was shown that $f_*(x)\de x$ satisfies the Poincar\'e inequality whose the Poincar\'e constant depending only on the diameter of the support of $\mu_*$.


\section{{ A detailed analysis of the Gaussian polymer example}} \label{appC}

In this section, we discuss in detail the stability and fluctuation phenomena in the simple settling of the Gaussian random polymer model, with a view towards illustrating their key features and to serve as a motivating instance for the stability results in the NPMLE problem that are the main focus of this paper.
Let the Gaussian environment $(G_v)_{v\in \mathbb{Z}^2}$ evolve according to the Ornstein-Uhlenbeck diffusion \eqref{eq:OU}. 
We denote by $\hat{p}_n^{(t)}$ the ground state polymer of the system under the $t$-perturbed environment $(G_v^{(t)})_{v\in \mathbb{Z}^2}$.
The first result, shown by Chatterjee in \cite{chatterjee2014superconcentration}, says that the ground state energy $\hat{E}_n$ is superconcentrated.

\begin{theorem}[Theorem 1.3. in \cite{chatterjee2014superconcentration}]
  For the $(1+1)$-dimensional Gaussian random polymer of length $n$, we have
\[ \var(\hat{E}_n) \le \frac{C n}{ \log n},\]
where $C$ does not depend on $n$.
\end{theorem}

To see why the above theorem implies superconcentration, we note that the ground state energy $\hat{E}_n$ is absolutely continuous with respect to the variables $(G_v:v\in V_n)$. Moreover,
\[\partial \hat{E}_n/\partial G_v = \mathbf{1}(v\in \text{optimal path}).\] Thus, 
$\|\nabla \hat{E}_n\|^2 = \sum_{v} \mathbf{1}(v\in \text{optimal path}) \le n+1$.
Since the optimal constant of the OU process is $1$, the Poincar\'e inquality immediately yields 
\[\var(\hat{E}_n) \le \mathbb{E}[\|\nabla \hat{E}_n\|^2] \le n+1.\]
Hence, we can take $\varepsilon =O (1/\log n)$ to deduce that $\hat{E}_n$ is superconcentrated. 

By the equivalence of superconcentration and chaos, we deduce that $\hat{E}_n$ is also chaotic (in the sense of \cite{chatterjee2014superconcentration}, i.e., Definition \ref{def:chaos}). However,
this particular notion of chaos is somehow less intuitive due to the emergence of the Dirichlet form. 

To get intuition, we remark that the Dirichlet form can be thought as a way to measure correlations in several probabilistic models. For example, for a discrete Gaussian free field on an undirected graph $G=(V,E)$, the Dirichlet form is defined via the Laplacian $L$ of the graph $G$, which captures the precision operator (the inverse covariance) governing conditional dependencies.

With this intuition, one might interpret $\mathcal{E}(f,P_tf)$ as a kind of "correlation" between 
$f(X)$ and $P_tf(X)$, where $X\sim \nu$. 
Indeed, for Langevin diffusions (which include the OU diffusions), the Dirichlet energy takes the form
\[\mathcal{E}(f,P_tf) = \mathbb{E}_{X\sim\nu} \Big [\nabla f(X) \cdot \nabla P_t f(X) \Big ],\] which is essentially the covariance of $\nabla f(X)$ and $\nabla P_tf(X)$. The presence of the small parameter $\varepsilon$ in Definition \ref{def:chaos} (ii) causes this correlation to decay faster than expected, indicating the chaotic nature of the system.

 However, as we remarked above, it is unclear that chaos in the sense of Definition \ref{def:chaos} is related to the standard notion of chaos in statistical mechanics, which characterized by the high sensitivity of the system to small changes in the environment.
Surprisingly, for the Gaussian polymer model,  \cite{chatterjee2014superconcentration} showed that the chaos of the ground state energy $\hat{E}_n$ in the sense of Definition \ref{def:chaos} (ii) implies the following result: 

\begin{theorem}[Chaos in Gaussian polymer] \label{t:chaotic}
   There exist $t_n \rightarrow 0$ such that
    \[ \mathbb{E} |\hat{p}_n \cap \hat{p}_n^{(t_n)}| = o(n),\]
    where $|p\cap p'|$ is the number of vertices  in the intersection of two paths.
\end{theorem}
In particular, the above theorem says that the ground state polymers $\hat{p}_n$ and $\hat{p}_n^{(t_n)}$ are almost disjoint (relative to their lengths) even when $t_n$ is arbitrarily small. This could be seens as usual notion of chaos in statistical physics: the ground state polymers are highly sensitive to small changes in the environment.

In general, it is still unclear how the formal Definition \ref{def:chaos} (ii)) relates to the usual notion of chaos in statistical physics (e.g. in the sense of Theorem \ref{t:chaotic}).
In the example of Gaussian polymer model, a key ingredient to connect these two notions of chaos is the following formula
\begin{equation} \label{eq:key-formula-polymer}
     \mathcal{E}_{\gamma_n}(\hat{E}_n,P_t \hat{E}_n) 
     = e^{-t} \cdot \mathbb{E}|\hat{p}_n \cap \hat{p}^{(t)}_n |, \quad \forall t \ge 0,
\end{equation}
    We remark that a formula like \eqref{eq:key-formula-polymer} usually depends heavily on the particular structure of the model and the diffusion. For example, to prove formula \eqref{eq:key-formula-polymer}, one has to use the following property of the OU semigroup
     \[ \nabla P_t = e^{-t} P_t \nabla,\]
     which is not available in general.

For the multiple valleys pheonomenon, we recall that the ground state energy $\hat{E}_n$ is obtained by optimizing the energy functional $E_n$, which is a random function of the environment.
Since the ground state energy $\hat{E}_n$ has been shown to be superconcentrated and chaotic, the multiple valleys phenomenon also occurs in this setting.
\begin{theorem}[Multiple valleys in Gaussian polymer]\label{thm:mvgp}
    Let $s_n(p,p'):= |p\cap p'|/n$ be the similarity measure between two paths $p$ and $p'$, where $|p\cap p'|$ denotes the number of vertices in the intersection of the two paths. Then the energy function $E_n$ in the $(1+1)$-dimensional Gaussian polymer model has multiple valleys with the similarity measure as above. In particular, when $n$ is large, then with high probability there are many
paths that all have nearly minimal energy and are all nearly disjoint from each other.
\end{theorem}

\section{{ Proofs of technical lemmas}} \label{appA}


\subsection{Proofs of technical lemmas in subsection \ref{subsec:proof-thm-stability}}

\subsubsection{Proof of Lemma \ref{lm:guntu-entropy}}\label{prf:appguntu-entropy}

\begin{proof}[Proof of Lemma \ref{lm:guntu-entropy}]
    This lemma follows automatically from Theorem 4.1 in \cite{saha2020nonparametric}.
\end{proof}

\subsubsection{Proof of Lemma \ref{lm:Hellinger-technical}}\label{prf:appHellinger-technical}

\begin{proof}[Proof of Lemma \ref{lm:Hellinger-technical}]
By definition of the covering set, there exists a (random) index $J \in [N]$ such that $\|g_J - \est\|_{\infty, \Theta_M} \le n^{-2}$.
On the event $\{\He^2(\est, f_*) \ge t\}$,  we have $J \in \mathbf{J}$. Thus
\[ \est(x) \le h_J(x) + 2n^{-2} , \quad \forall x\in \Theta_M.\]
For $x\notin \Theta_M$, we use the crude bound $f(x) \le 1$ that holds for any $x\in \R^d$ and any density $f\in \mathcal{M}$. 
Thus,  we have 
\begin{eqnarray*}
    \prod_{i=1}^n \est(X_i) 
    &\le&  \prod_{i:X_i \in \Theta_M} \Big (h_J(X_i) + 2n^{-2}\Big ) \\
    &=& \Big [ \prod_{i=1}^n \Big (h_J(X_i) + 2n^{-2}\Big ) \Big ] \cdot \Big [\prod_{i=1}^n  \Big (\frac{1}{h_J(X_i) + 2n^{-2}}\Big )^{\mathbf{1}(X_i\notin\Theta_M)} \Big] \\
    &\le& \Big [ \prod_{i=1}^n \Big (h_J(X_i) + 2n^{-2}\Big ) \Big ] \cdot \Big [\prod_{i=1}^n  \Big (\frac{1}{n^{-2}}\Big )^{\mathbf{1}(X_i\notin\Theta_M)} \Big]\\
     &\le& \Big [\max_{j\in \mathbf{J}}\prod_{i=1}^n \Big (h_j(X_i) + 2n^{-2}\Big )\Big ] 
     \cdot \Big [\prod_{i=1}^n  \Big (1+n^2{\mathbf{1}(X_i\notin\Theta_M)}\Big ) \Big].
\end{eqnarray*}
Therefore
\begin{equation} \label{eq:ineq-2}
    \prod_{i=1}^n \frac{\est(X_i)}{f_*(X_i)} \le  
    \Big [\max_{j\in \mathbf{J}}\prod_{i=1}^n \frac{h_j(X_i) + 2n^{-2}}{f_*(X_i)}\Big ] 
    \cdot \Big [\prod_{i=1}^n  \Big (1+n^2{\mathbf{1}(X_i\notin\Theta_M)}\Big ) \Big].
\end{equation}
Since $|L_n(\est) - \optloss| \le \varepsilon$, we have $L_n(\est) \ge L_n(f_*) - \varepsilon$. This implies 
\[ \prod_{i=1}^n \frac{\est(X_i)}{f_*(X_i)} \ge \exp (-n\varepsilon).\]
The lemma follows.
\end{proof}

\subsubsection{Proof of Lemma \ref{lm:1-term}}\label{prf:app1-term}
\begin{proof}[Proof of Lemma \ref{lm:1-term}]
    By Markov's inequality, we have
\begin{eqnarray*}
    \mathbb{P} \Big (\prod_{i=1}^n \frac{h(X_i) + 2n^{-2}}{f_*(X_i)}  \ge e^{-\gamma}\Big ) 
    &\le& \exp \Big (\frac{\gamma}{2} \Big ) \cdot \mathbb{E} \Big [\prod_{i=1}^n \sqrt{\frac{h(X_i) + 2n^{-2}}{f_*(X_i)} } \Big ] \\
    &=&\exp \Big (\frac{\gamma}{2}  \Big ) \cdot \mathbb{E} \Big [ \sqrt{\frac{h(X_1) + 2n^{-2}}{f_*(X_1)} } \Big ]^n, 
\end{eqnarray*}
where we used the fact that $X_1,\ldots,X_n$ are i.i.d. samples. Now, observe that
\begin{eqnarray*}
    \mathbb{E} \Big [ \sqrt{\frac{h(X_1) + 2n^{-2}}{f_*(X_1)} } \Big ] &\le& \exp \Big (\mathbb{E} \Big [ \sqrt{\frac{h(X_1) + 2n^{-2}}{f_*(X_1)} } \Big ] -1 \Big ) \\
    &=& \exp \Big (\int \sqrt{\frac{h(x)+2n^{-2}}{f_*(x)}}f_*(x)\de x  - 1 \Big )\\
    &\le& \exp \Big (\sqrt{2}n^{-1}\int \sqrt{f_*} + \int \sqrt{hf_*}   - 1 \Big ) \\
    &=& \exp \Big (\sqrt{2}n^{-1}\int \sqrt{f_*} - \frac{1}{2} \He^2(h,f_*) \Big ),
\end{eqnarray*}
where we used  the fact that $\He^2(h,f_*)=2-2\int \sqrt{hf_*}$ in the last line.
Combining all ingredients, we have 
\[ \mathbb{P} \Big (\prod_{i=1}^n \frac{h(X_i) + 2n^{-2}}{f_*(X_i)}  \ge e^{-\gamma}\Big ) 
\le \exp \Big (\frac{\gamma}{2}   - \frac{n}{2}\He^{2}(h,f_*) +\sqrt{2}\int \sqrt{f_*}\Big ),\]
as desired.
\end{proof}

\subsubsection{Proof of Lemma \ref{lm:2-term}}\label{prf:app2-term}

\begin{proof}[Proof of Lemma \ref{lm:2-term}]
We have
    \begin{eqnarray*}
       \E \Big [\prod_{i=1}^n  \Big (1+n^2{\mathbf{1}(X_i\notin\Theta_M)}\Big ) \Big ] = \Big (1+ n^2\mathbb{P}(X_1\notin\Theta_M) \Big )^n.
    \end{eqnarray*}
    where we used the fact that $X_i$'s are i.i.d. Now, using the inequality $1+x \le e^x$, we have
    \begin{eqnarray*}
        \E \Big [\prod_{i=1}^n  \Big (1+n^2{\mathbf{1}(X_i\notin\Theta_M)}\Big ) \Big ] \le \exp \Big (n^3\mathbb{P}(X_1\notin\Theta_M) \Big ).
    \end{eqnarray*}
    Recall that $X_1\sim f_*$, where $f_*$ is a Gaussian location mixture with mixing measure $\mu_*$. Thus, we can write $X_1 =Y +Z$, where $Y$ is a sample from $\mu_*$ and $Z$ is a standard Gaussians in $\R^d$ independent of $Y$. As $\supp(\mu_*) \subset \Theta$, $Y \in \Theta$ almost surely. Thus $\dist(X_1,\Theta) \le \|X_1 - Y\| = \|Z\|$, which implies that
    \[\mathbb{P}(X_1 \notin \Theta_M) \le \mathbb{P}(\|Z\|\ge M) .\]
    
    Note that $\|Z\|^2$ follows $\chi^2$-distribution with $d$ degrees of freedom, hence 
    \[\mathbb{P}(\|Z\|^2 \ge d+ 2\sqrt{dx} + 2x) \le e^{-x}, \quad \forall x\ge 0.\]
    Since $d+2\sqrt{dx}+2x \le 3(d+x)$, we have $\mathbb{P}(\|Z\|^2 \ge 3(d+x)) \le e^{-x}$ for all $x \ge 0$. Thus
     \[\mathbb{P}(X_1 \notin \Theta_M) \le \mathbb{P}(\|Z\|\ge M) \le e^{d-M^2/3} = e^{d - 3\log n} = n^{-3}e^d  .\]

     Combining all ingredients, we deduce that 
     \[ \E \Big [\prod_{i=1}^n  \Big (1+n^2{\mathbf{1}(X_i\notin\Theta_M)}\Big ) \Big ] \le \exp \Big (n^3\mathbb{P}(X_1\notin\Theta_M) \Big ) \le \exp(\exp(d)).\]
\end{proof}


\subsubsection{Proof of Lemma \ref{lm:assumption-M-Theta-tau}}

\begin{proof}[Proof of Lemma \ref{lm:assumption-M-Theta-tau}]

Let $\varepsilon>0$ and $r\ge 0$, we define
 \[ D(r):= \{x \in \mathbb{R}^d: \dist(x, \supp(\mu_*)) \le r \}.\]
    Let $f \in \mathcal{M}$ be a GMM density such that the total variation distance $\TV(f,f_*) \le \varepsilon$. Then, if $\mu$ denotes the mixing measure of $f$, we claim that:
    \[ \mu(D(2r)) \ge 2-\frac{  1 + \varepsilon }{\mathbb{P}(\|Z\|\le r) }  \]
    where $Z$ denotes a standard Gaussian vector in $\mathbb{R}^d$. Indeed, Let $X\sim f_*(x)\de x$ and $X' \sim f(x) \de x$. By definition, we can write $X= Z  + \theta$ where $\theta \sim \mu_*$ and $Z$ is a standard Gaussian in $\mathbb{R}^d$ independent of $\theta$. This implies
\[ \mathbb{P}(X \in D(r)) = \mathbb{P}(Z+\theta \in D(r)) \ge \mathbb{P}(\|Z\|\le r).\]

Similarly, we can write $X'=Z+\theta'$ where $\theta' \sim \mu$ independent of $Z$. 
Observe that on the event $\{\theta' \notin D(2r)\}$, we have $\dist(\theta',D(r)) \ge r$; thus $X' \in D(r)$ would imply $\|Z\| \ge r$. 
Therefore
\begin{eqnarray*}
    \mathbb{P}(X' \in D(r)) &=& \mathbb{P}(X' \in D(r)|\theta' \in D(2r)) \mathbb{P}(\theta' \in D(2r)) \\
    &+& \mathbb{P}(X' \in D(r)|\theta' \notin D(2r)) (1- \mathbb{P}(\theta' \in D(2r))) \\
    &\le& \mathbb{P}(\theta' \in D(2r)) + \mathbb{P}(\|Z\| \ge r) (1-\mathbb{P}(\theta' \in D(2r))\\
    &=& \mathbb{P}(\|Z\| \ge r) + \mathbb{P}(\|Z\|\le r) \mathbb{P}(\theta' \in D(2r)).
\end{eqnarray*}

Since $\TV(f,f_*) \le \varepsilon$, we have
\begin{eqnarray*}
    \mathbb{P}(X \in D(r)) - \mathbb{P}(X' \in D(r)) \le \varepsilon
\end{eqnarray*}
which implies 
\[ \mu(D(2r))=\mathbb{P}(\theta' \in D(2r)) \ge \frac{ 2\mathbb{P}(\|Z\| \le r) - 1 - \varepsilon }{\mathbb{P}(\|Z\|\le r) } = 2-\frac{  1 + \varepsilon }{\mathbb{P}(\|Z\|\le r) } \]
as desired. 

Now note that on $\mathcal{E}$, we have $\He^2(f_*,\est) \le \varepsilon_n + C_1(\log n)^{d+1}/n \le \frac{1}{2}(1-e^{-1})^2 <1/5$. This implies
    \[ \TV(f_*, \est) \le \sqrt{2} \He(f_*,\est) <\sqrt{\frac{2}{5}} <0.65.\]
    Let $r>0$ be such that $\mathbb{P}(\|Z\| \le r) \ge 0.9$.  Then $\est \in \mathcal{M}(D(2r);\tau)$ with
    \[ \tau \ge 2 - \frac{1.65}{0.9} > 0.16.\]
    
\end{proof}


\subsubsection{Proof of Lemma \ref{lm:condition-M-delta}}\label{prf:appcondition-M-delta}

\begin{proof}[Proof of Lemma \ref{lm:condition-M-delta}]
   Let $f\in \mathcal{M}(\widetilde{\Theta};\tau)$. Then
  \begin{eqnarray*}
f(x) 
\ge \frac{1}{(2\pi)^{d/2}} \int_{\widetilde{\Theta}} e^{-\|x-\theta\|^2/2} \mu_f(\de \theta) 
\ge \frac{\tau}{(2\pi)^{d/2}} \cdot e^{-\|x\|^2 - C}, 
\end{eqnarray*}
where $C:= \sup_{\theta \in \widetilde{\Theta}} \|\theta\|^2$. Thus $f(x) \gtrsim  \exp (-\|x\|^2)$ where the implicit constant depends only on $\widetilde{\Theta}, d$ and $\tau$. Thus
    \begin{eqnarray*}
        \int f_*(x) \Big (\frac{f_*(x)}{f(x)} \Big )^{1/2} \de x &\lesssim& \int f_*(x)^{3/2} \exp \Big (\frac{\|x\|^2}{2}\Big ) \de x \\
        &\lesssim& \int \Big (\int \exp \Big (-\frac{\|x-\theta\|^2}{2} + \frac{\|x\|^2}{3}
        \Big ) \mu_*(\de \theta) \Big )^{3/2} \de x \\
         &=& \int \Big (\int \exp \Big (-\frac{\|x\|^2}{6} - \langle x,\theta\rangle - \frac{\|\theta\|^2}{2}
        \Big ) \mu_*(\de \theta) \Big )^{3/2} \de x \\
        &\le&  \int \Big (\int \exp \Big (-\frac{\|x\|^2}{6} +\frac{\|x\|^2}{8}+2\|\theta\|^2 - \frac{\|\theta\|^2}{2}
        \Big ) \mu_*(\de \theta) \Big )^{3/2} \de x \\
        &=& \int \Big (\int \exp \Big (-\frac{\|x\|^2}{24} + \frac{3\|\theta\|^2}{2}
        \Big ) \mu_*(\de \theta) \Big )^{3/2} \de x \\
        &=& \Big ( \int \exp \Big (-\frac{\|x\|^2}{8} \Big ) \de x \Big ) \cdot \Big (\int \exp \Big (\frac{3\|\theta\|^2}{2}\Big ) \mu_*(\de \theta) \Big )^{3/2} < \infty.
    \end{eqnarray*}
\end{proof}

\subsection{Proofs of technical lemmas in subsection \ref{subsubsec:stability_KL_dudley}}

\subsubsection{Proof of Lemma \ref{lm:dudley-bound}}\label{prf:appdudley-bound}

\begin{proof}[Proof of Lemma \ref{lm:dudley-bound}]
    We define 
    \[G(x):= \sup_{f\in \mathcal{M}({\Theta};\tau)} |\log f(x) -\log f_*(x)|\] 
 Note that for $f\in \mathcal{M}(\Theta;\tau)$, we have
 \begin{eqnarray*}
f(x) 
\ge \frac{1}{(2\pi)^{d/2}} \int_{{\Theta}} e^{-\|x-\theta\|^2/2} \mu_f(\de \theta) 
\ge \frac{\tau}{(2\pi)^{d/2}} \cdot e^{-\|x\|^2 - \sup_{\theta \in {\Theta}} \|\theta\|^2}.
\end{eqnarray*}
This implies
     $|\log f(x)| \le \|x\|^2 + C$, for some constant $C$ depending only on ${\Theta}, d$ and $\tau$. Thus $\|G\|_{L^2(f_*)} <\infty$. 

 Let $\mathcal{G}({\Theta};\tau):= \{\log f-\log f_*: f \in \mathcal{M}({\Theta};\tau)\}$
     Since $G(x)$ is an envelope function of $\mathcal{G}({\Theta};\tau)$,
     by Dudley's entropy bound we have
    \begin{equation} \label{eq:lm-dudley-bound}
        \E \Big [\sup_{g\in \mathcal{G}({\Theta};\tau)} \Big |\frac{1}{\sqrt{n}} \sum_{i=1}^n \Big (g(X_i) - \E g(X_i) \Big ) \Big | \Big ] \lesssim \int_0^{\|G\|_{L^2(f_*)}} \sqrt{\log N_{[]}(\varepsilon, \mathcal{G}({\Theta};\tau),L^2(f_*))} \de \varepsilon.
    \end{equation}
    By Theorem \ref{thm:FiniteBracketingIntegral}, $\log N_{[]}(\varepsilon, \mathcal{G}({\Theta};\tau),L^2(f_*)) \lesssim (-\log \varepsilon)^{d+1}$, which is integrable near $0$. Thus, the integral in \eqref{eq:lm-dudley-bound} is finite. The lemma follows.
\end{proof}




\subsection{Proofs of technical lemmas in subsection \ref{subsec:proof-thm-L2mu}} \label{sbsec:prfL2mu_bound}

\subsubsection{Proof of Lemma \ref{lm:not-in-qn}}\label{prf:appnot-in-qn}

\begin{proof}[Proof of Lemma \ref{lm:not-in-qn}]

    Let $p\in \{1,2\}$. By Cauchy-Schwarz's inequality, we have
    \begin{eqnarray*}
        &&\E \Big [|\sqrt{n}(\optloss - L_n(f_*))|^p \mathbf{1}(\npmle \notin \mathcal{Q}_n) \Big ]\\ &\le& n^{p/2} \cdot\E \Big [|\optloss - L_n(f_*)|^{2p} \Big ]^{1/2} \cdot \mathbb{P}(\npmle\notin \mathcal{Q}_n)^{1/2} \\
        &\lesssim& n^{p/2} \cdot \Big (\E[|\optloss|^{2p}]^{1/2p} + \E[|L_n(f_*)|^{2p}]^{1/2p} \Big )^p \cdot e^{- (\log n)^{d+1}}.
    \end{eqnarray*}
    Now we observe that for any $f\in \mathcal{M}$
    \begin{eqnarray*}
        \E[|L_n(f)|^{2p}]^{1/2p} &=& \E \Big [ \Big |\frac{1}{n}\sum_{i=1}^n \log f(X_i) \Big|^{2p} \Big ]^{1/2p} \\
        &\le& \frac{1}{n} \sum_{i=1}^n \E[|\log f(X_i)|^{2p}]^{1/2p} = \E[|\log f(X_1)|^{2p}]^{1/2p}
    \end{eqnarray*}
    since $X_i$'s are identically distributed. 
    
    Let us tak a quick detour to derive an upper bound on the logarithm of the density. Writing $f(x) = (2\pi)^{-d/2} \int \exp(-\|x-\theta\|^2/2)\mu(\de \theta)$ for some $\mu \in \mathcal{P}(\R^d)$, we have
    \begin{eqnarray*}
        f(x) \ge \frac{1}{(2\pi)^{d/2}} \int e^{-\|x\|^2 - \|\theta\|^2} \mu(\de \theta) \ge \frac{1}{(2\pi)^{d/2}}e^{-\|x\|^2 - \sup_{\theta \in \supp(\mu)}\|\theta\|^2}.
    \end{eqnarray*}
    Thus 
    \begin{equation}\label{eq:logf_upbd}
      |\log f(x)| \le \|x\|^2 + \sup_{\theta\in \supp(\mu)} \|\theta\|^2 +C,  
    \end{equation}
    for some constant $C$ independent of $f$. While inequality~\eqref{eq:logf_upbd} is uncontrolled in general, we will still be able to use it for the specific cases of $f = f_*$ and $f = \npmle$ by the following arguments.

    In particular for $f=f_*$, the mixing measure $\mu_*$ is compactly supported, and thus we have by applying \eqref{eq:logf_upbd} that,
    \[ \E[|L_n(f_*)|^{2p}]^{1/2p} \lesssim \E[\|X_1\|^{2p}]^{1/2p} + C = O(1).\]
    
    On the other hand note that for $f=\npmle$ the mixing measure $\hat{\mu}_n$ is supported inside the convex hull of $X_i$'s. Thus by \eqref{eq:logf_upbd} we have
    \begin{eqnarray*}
    \E[|\optloss|^{2p}]^{1/2p} &\lesssim& \E[\|X_1\|^{2p}]^{1/2p} + \E \Big [ \max_{1\le i \le n} \|X_i\|^{2p} \Big ]^{1/2p} +C \\
    &\le&  \E[\|X_1\|^{2p}]^{1/2p} + \E \Big [ \Big (\sum_{i=1}^n \|X_i\|\Big )^{2p} \Big ]^{1/2p} +C \\
    &\le&  \E[\|X_1\|^{2p}]^{1/2p} +  \sum_{i=1}^n \E  [ \|X_i\|^{2p} ]^{1/2p} +C = O(n).
    \end{eqnarray*}
Thus,
\[n^{p/2} \cdot \Big (\E[|\optloss|^{2p}]^{1/2p} + \E[|L_n(f_*)|^{2p}]^{1/2p} \Big )^p \cdot e^{- (\log n)^{d+1}} = O \Big (n^{3p/2} \cdot e^{- (\log n)^{d+1}} \Big ) = o(1),\]
as desired.
\end{proof}

\subsubsection{Proof of Lemma \ref{lem:KLdiv-Qn}}\label{prf:appKLdiv-Qn}

\begin{proof}[Proof of Lemma \ref{lem:KLdiv-Qn}]
    By Lemma \ref{lm:assumption-M-Theta-tau} and Lemma \ref{prop:shenwong}, we have for $n$ large enough:
    \[\KL (f_*\|f) \lesssim  \frac{(\log n)^{d+1}}{n} \log \Big ( \frac{n}{(\log n)^{d+1}} \Big ) \lesssim  \frac{(\log n)^{d+2}}{n}\] 
    as desired.
\end{proof}

\subsubsection{Proof of Lemma \ref{lem:envelopeFunction}}\label{prf:appenvelopeFunction}

\begin{proof}[Proof of Lemma \ref{lem:envelopeFunction}]
By Lemma \ref{lm:assumption-M-Theta-tau},
for $n$ sufficiently large, $\mathcal{Q}_n \subset \mathcal{M}(\widetilde\Theta;\tau)$ for some compact set $\widetilde\Theta$ defined by the support of $\mu_*$ (cf. Lemma \ref{lm:assumption-M-Theta-tau}).
In particular, we would have 
 $|\log f(x)| \le \|x\|^2 + C$ for some constant $C$ depending only on $\mu_*, d$ for all $f\in \mathcal{Q}_n$ for $n$ large.
Thus, for $n$ large, we have
\[ |g_f(x)| = |\log f(x) - \log f_*(x)| \le 2 \|x\|^2 + 2C, \quad \forall f \in \mathcal{Q}_n.\]
This particularly implies $|G_n(x)| \le 2 \|x\|^2 + 2C$. Thus
\[ \|G_n\|^2_{L^2(f_*)} = \E[|G_n(X_1)|^2] \lesssim \E[\|X_1\|^4] + 1 < \infty.\]
Hence $G_n \in L^2(f_*)$.

Since $\mathcal{Q}_n$ is nested, i.e. $\mathcal{Q}_{n+1} \subset \mathcal{Q}_n$ for every $n$, it is clear that $G_n$ is decreasing in $n$. Assuming that there exists $x_0 \in \R^d$ such that $\lim_{n\rightarrow \infty} G_n(x_0) = \delta > 0$, by definition, there exists a sequence of functions $f_n$ such that
\[f_n \in \mathcal{Q}_n \quad\text{ and } \quad |\log f_n(x_0) - \log f_*(x_0) | \ge \delta/2 > 0, \quad\forall n. \]
In particular, there exists $\tilde{\delta}>0$ such that 
\begin{equation} \label{eq:sqrt-fn-x0}
|\sqrt{f_n(x_0)} -\sqrt{f_*(x_0)}| \ge \tilde{\delta} >0, \quad \forall n.
\end{equation}

We claim that for any GMM $f$ of the form
\[f(x) = \frac{1}{(2\pi)^{d/2}}\int_{\R^d} e^{-\|x-\theta\|^2/2}\mu(d \theta), \quad \mu\in \mathcal{P}(\R^d), \]
then $f$ is Lipschitz on $\R^d$ with uniformly bounded Lipschitz constant $\mathrm{Lip}(f)\le (2\pi)^{-d/2}e^{-1/2}$. Indeed, let $\varphi(z):= (2\pi)^{-d/2}e^{-\|z\|^2/2}$, we have
\[\|\nabla \varphi(z)\|=\frac{1}{(2\pi)^{d/2}}\|z\|e^{-\|z\|^2/2} \le(2\pi)^{-d/2}e^{-1/2}, \quad \forall z\in \R^d. \]
Now for every $x,y \in\R^d$, we have
\begin{eqnarray*}
    |f(x)-f(y)|&\le& \int_{\R^d}|\varphi(x-\theta)-\varphi(y-\theta)|\mu(d \theta) \le \frac{\|x-y\|}{(2\pi)^{d/2}e^{1/2}} \cdot
\end{eqnarray*}
The claim follows. Consequently, the square roots of such GMM densities are uniformly $1/2$-H\"older continuous.

By \eqref{eq:sqrt-fn-x0} and the uniform $1/2$-H\"older continuity of square roots of GMM densities, there exists a small neighborhood $B$ of $x_0$ (independent of $n$) such that $|\sqrt{f_n(x)}-\sqrt{f_*(x)}| \ge \tilde{\delta}/2$ uniformly on $B$. Thus
\[ \He^2(f_n,f_*) \ge \int_{B}|\sqrt{f_n(x)}-\sqrt{f_*(x)}|^2 \de x \ge \mathrm{Vol}(B) (\tilde{\delta}/2)^2>0. \]
On the other hand, if $f_n \in \mathcal{Q}_n$ for every $n$, then $\lim_{n\rightarrow \infty}\He^2(f_n, f_*) = 0$. This yields a contradiction, which implies that $G_n$ must decrease to $0$ everywhere. The lemma follows.

\end{proof}

\subsubsection{Proof of Lemma \ref{lm:bracketing-entropy}}
\label{prf:appbracketing-entropy}
\begin{proof}[Proof of \ref{lm:bracketing-entropy}]

Let $\tau\in (0,1)$. We assume the contrapositive statement, i.e., assume that there exists a sequence $n_k \rightarrow \infty$ such that for every $k$, $f_{n_k} \in \mathcal{Q}_{n_k}$ but $f_{n_k} \notin \mathcal{M}(\Theta;\tau)$. By definition of $\mathcal{Q}_{n_k}$, $f_{n_k}$ converges to $f_*$ in the Hellinger distance. Thus
\[ \|f_{n_k}-f_*\|_{L^1(\R^d)} = \int_{\R^d}(\sqrt{f_{n_k}}-\sqrt{f_*})(\sqrt{f_{n_k}}+\sqrt{f_*}) \le 2\|\sqrt{f_{n_k}}-\sqrt{f_*}\|_{L^2(\R^d)} \to 0.\]
In particular, the probability measures $\nu_{n_k}:=f_{n_k}(x) \de x$ converge to $\nu_*:=f_*(x)\de x$ in the total variation distance. Now observe that 
\[ \nu_{n_k}= \gamma\ast \mu_{n_k} \quad,\quad \nu_* =\gamma \ast \mu_*\]
where $\gamma$ denotes the standard Gaussian measure $(2\pi)^{-d/2}\exp(-\|x\|^2/2)\de x$ in $\R^d$. We deduce (since the Gaussian characteristic function is nowhere zero) that the mixing measures $\mu_{{n_k}}$ of $f_{n_k}$ also weakly converges to the the mixing measure $\mu_*$ of $f_{*}$. In particular, $\mu_{n_k}(\Theta)$ converges to $1$, which contradicts to $f_{n_k} \notin \mathcal{M}(\Theta;\tau)$. This contradiction completes the proof.

\end{proof}

\subsection{Proofs of technical lemmas in subsection \ref{subsec:proof-thm-anticon}}

\subsubsection{Proof of Lemma \ref{lm:convex}}\label{prf:appconvex}

\begin{proof}[Proof of Lemma \ref{lm:convex}]
    For each $f\in \mathcal{M}$, let $G_{f}(x) := \log f(x) + \|x\|^2/2$. We claim that $G_{f}$ is convex on $\mathbb{R}^d$ for all $f\in \mathcal{M}$. 
    Indeed,  for $f\in \mathcal{M}$, we can write 
    \[f(x) = (2\pi)^{-d/2} \int \exp(-\|x-\theta\|^2/2)\mu_f( \de\theta)\] 
    where $\mu_f\in \mathcal{P}(\R^d)$ is the mixing measure of $f$.
    It is then easy to see that
 \[ G_{f}(x) = \log \Big ( \int \exp \Big (\langle x,\theta \rangle -\frac{1}{2}\|\theta\|^2 \Big ) \mu_f(\de\theta) \Big ) + \text{constant}. \]
   Since we are interested in the convexity of $G_{f}$, we can ignore the constant term. 
   
    Let $t\in [0,1]$ and $x,y \in \mathbb{R}^d$ be arbitrary, we have
    \begin{eqnarray*}
        &&G_{f}(tx + (1-t)y) \\
        &=& \log \Big \{ \int \exp\Big ( t\langle x, \theta\rangle +(1 -t)\langle y,\theta\rangle -\frac{1}{2}\|\theta\|^2 \Big )\mu_f(\de\theta) \Big \} \\
        &=& \log \Big \{ \int \exp\Big [ t \Big (\langle x, \theta\rangle -\frac{1}{2}\|\theta\|^2\Big )+(1 -t) \Big (\langle y,\theta\rangle -\frac{1}{2}\|\theta\|^2 \Big ) \Big ]\mu_f(\de\theta) \Big \} \\
        &\le& \log \Big \{ \Big [ \int \exp \Big (\langle x, \theta\rangle -\frac{1}{2}\|\theta\|^2\Big )\mu_f (\de\theta)\Big]^t \cdot 
        \Big[ \int \exp \Big (\langle y, \theta\rangle -\frac{1}{2}\|\theta\|^2\Big )\mu_f(\de\theta)\Big]^{1-t} \Big \}\\
        &=& t G_{f}(x) + (1-t)G_{f}(y),
    \end{eqnarray*}
    where we used H\"older inequality and the monotonicity of $\log$. Thus, $G_{f}$ is convex.

    Given the convexity of $G_{f}$, we observe that for each $f\in \mathcal{M}$
    \[L_n(f)(x_1,\ldots,x_n) + \frac{1}{2n} \sum_{i=1}^n \|x_i\|^2= \frac{1}{n} \sum_{i=1}^n  \Big (\log f(x_i) + \frac{\|x_i\|^2}{2} \Big ) = \frac{1}{n} \sum_{i=1}^n G_{f}(x_i),\]
    which is convex in $(x_1,\ldots,x_n)$ due to the observation above.
     Since convexity is closed under taking supremum and convex functions are differentiable almost everywhere, the lemma follows.
\end{proof}

\subsubsection{Proof of Lemma \ref{lm:Dirichlet-control-1}}\label{prf:appDirichlet-control-1}

\begin{proof}[Proof of Lemma \ref{lm:Dirichlet-control-1}]

  We have $n \|\nabla \optloss\|^2 = n^{-1} \sum_{i=1}^n \|\nabla \log \npmle (X_i)\|^2$. Note that for every $f\in \mathcal{M}$, we have 
  \[ \|\nabla f(x)\| \le \frac{1}{(2\pi)^{d/2}} \int \|x-\theta\|e^{-\|x-\theta\|^2/2}\mu_f(\de\theta) \le \Big (\|x\| + \sup_{\theta \in \supp(\mu_f)}\|\theta\| \Big ) f(x).\]
  Since $\npmle$ is supported inside the convex hull of $\{X_i:1\le i \le n\}$, we have 
  almost surely that
  \[ \|\nabla \log \npmle (X_i)\| \le \|X_i\| + \max_{1\le j\le n} \|X_j\| \le \|X_i\|+\sum_{j=1}^n\|X_j\|.\]
  In particular, $\E[\|\nabla \log \npmle (X_1)\|^k] = O(n^k)$ for every $k$.
  Thus
  \begin{eqnarray*}
      \E[n^2\|\nabla \optloss\|^4] &=& \E \Big [\frac{1}{n^2} \Big |\sum_{i=1}^n \|\nabla \log \npmle(X_i)\|^2 \Big |^2 \Big ] \\
      &\lesssim& \E [\|\nabla \log \npmle(X_1)\|^4] \le O(n^4).
  \end{eqnarray*}
  Since $\mathbb{P}(\npmle \notin \mathcal{Q}_n) \lesssim \exp(-(\log n)^{d+1})$, we have $\E[n^2\|\nabla \optloss\|^4] \cdot\mathbb{P}(\npmle \notin \mathcal{Q}_n) = o(1)$ as desired.
    
\end{proof}

\subsubsection{Proof of Lemma \ref{lem:derivsup}}\label{prf:appderivsup}

\begin{proof}[Proof of Lemma \ref{lem:derivsup}]
    It suffices to prove for $m=1$. Note that for any $x,y \in \mathbb{R}$, we have
    \[\sup_{g \in \mathcal{G}} g(y) - \sup_{g \in \mathcal{G}} g(x) \leq \sup_{g \in \mathcal{G}} (g(y) - g(x)).\]
    Exchanging the position of $x$ and $y$, noting that $\hat{g}(x)=\sup_{g\in \mathcal{G}} g(x)$, we deduce that
    \[|\hat{g}(x)-\hat{g}(y)| \le \sup_{g\in\mathcal{G}}|g(x)-g(y)|.\]
    Dividing by $\abs{y-x}$ and sending $y \to x$ completes the proof.
\end{proof}

\subsubsection{Proof of Lemma \ref{lm:Dirichlet-control-2}}\label{prf:appDirichlet-control-2}

\begin{proof}[Proof of Lemma \ref{lm:Dirichlet-control-2}]
    Let $\tau\in (0,1)$ be fixed and $\Theta$ be a compact set whose interior containing the support of $\mu_*$. For large enough $n$, we have $\mathcal{Q}_n \subset \mathcal{M}(\Theta;\tau)$ by Lemma \ref{lm:bracketing-entropy}. Thus, it suffices to show the desired inequality for $f\in \mathcal{M}(\Theta;\tau)$.

    For $f\in \mathcal{M}(\Theta;\tau)$, one has
    \begin{eqnarray*}
        \|\nabla f(x)\| &\le & \frac{1}{(2\pi)^{d/2}}\int\|x-\theta\| e^{-\|x-\theta\|^2/2} \mu_f(\de \theta).
    \end{eqnarray*}

    First, let us consider the case when $\sqrt{2}\|x\|\ge 1$.
    We claim that for every $\theta \in \R^d$
    \[ \|x-\theta\| e^{-\|x-\theta\|^2/2} \le \sqrt{2}\|x\| e^{-\|x\|^2} + \sqrt{2}\|x\|e^{-\|x-\theta\|^2/2}.\]
    Indeed, if $\|x-\theta\| \le \sqrt{2}\|x\|$, the inequality trivially holds. If $\|x-\theta\| \ge \sqrt{2} \|x\|$, we note that the function $t\mapsto te^{-t^2/2}$ is decreasing on $[1,+\infty)$, which implies $\|x-\theta\|e^{-\|x-\theta\|^2/2} \le \sqrt{2} \|x\| e^{-\|x\|^2}$ as desired. Thus, for $\sqrt{2}\|x\| \ge 1$, one has
    \begin{eqnarray*}
        \|\nabla \log f(x)\| &=& \Big \| \frac{\nabla f(x)}{f(x)} \Big \| \\
        &\le& \frac{\sqrt{2} \|x\|e^{-\|x\|^2} + \sqrt{2}\|x\| \int e^{-\|x-\theta\|^2/2} \mu_f(\de \theta)}{ \int e^{-\|x-\theta\|^2/2} \mu_f(\de \theta)} \\
        &=& \frac{\sqrt{2} \|x\|}{ e^{\|x\|^2}\int e^{-\|x-\theta\|^2/2} \mu_f(\de \theta)} + \sqrt{2}\|x\| .
    \end{eqnarray*}
For $f\in \mathcal{M}(\Theta;\tau)$, we have $f(x) \ge C \cdot e^{-\|x\|^2}$ for some constant $C>0$ depending only on $\Theta,d$ and $\tau$. Thus, for $\sqrt{2}\|x\|\ge 1$, we have 
$\|\nabla \log f(x)\| \le C_1\|x\|$ for some $C_1$ depending only on $\Theta,d$ and $\tau$.

Now let us consider the case when $\sqrt{2}\|x\|\le 1$. Then $f(x) \ge C \cdot e^{-1/2}$. We note that the function $t\mapsto t e^{-t^2/2}$ on $[0,+\infty)$ has a global maximizer at $t=1$, which implies
\[ \|\nabla  f(x) \| \le \frac{1}{(2\pi)^{d/2}} \int \|x-\theta\| e^{-\|x-\theta\|^2/2} \mu_f(\de \theta) \le(2\pi)^{-d/2} e^{-1/2}.\]
Thus, for $\sqrt{2}\|x\|\le 1$, we have
\[ \|\nabla \log f(x)\| \le C_2\]
for some $C_2>0$ depending only on $\Theta, d$ and $\tau$. 
Combining all ingredients, we deduce that $\|\nabla \log f(x)\| \le C_1\|x\|+C_2$ for any $f\in \mathcal{M}(\Theta;\tau)$ as desired.
\end{proof}


\section{{ Proof of Theorem~\ref{thm:FiniteBracketingIntegral}}} \label{app:prfFiniteBracketingIntegral}
\subsection{Proof of Theorem~\ref{thm:FiniteBracketingIntegral}}

We first recall the class $\mathcal{M}(\Theta;\tau)$ of GMM densities 
\[  \mathcal{M}(\Theta;\tau):= \Big \{ f\in \mathcal{M}: \mu_f(\Theta) \ge \tau \Big \}, \quad \tau \in (0,1] \]
where $\mu_f$ denotes the mixing measure of the GMM density $f$, and 
\[\log\mathcal{M}(\Theta;\tau):=\Big \{ \log f:f\in \mathcal{M}(\Theta;\tau) \Big \}.\] 

To construct brackets for $\log \mathcal{M}(\Theta;\tau)$, we will need to control the logarithmic densities $\log f(x)$ pointwisely, where $f \in \mathcal{M}(\Theta;\tau)$. To this end, let us first state a simple property of GMM densities in this class, proved later in Appendix~\ref{prf:appbound-density-Mc}.
\begin{lemma} \label{lm:bound-density-Mc}
    There exists a constant $C_{\ref{lm:bound-density-Mc}}$ depending only on $\Theta, d$ and $\tau$ such that
    \[ 0 \le - \log f(x) \le \|x\|^2 +C_{\ref{lm:bound-density-Mc}}, \quad \forall x\in \R^d, \forall f \in \mathcal{M}(\Theta;\tau).\]
\end{lemma}


As remarked earlier, controlling a logarithmic density pointwise is challenging, due to the possibility that it may diverge to infinity when the underlying density approaches zero. To address this issue, we first perform a splitting argument as follows.

\subsubsection{A splitting argument}

Let $R>0$ be a radius to be specified later. We split $\R^d$ into the disjoint union of the ball $B_R:= B(0,R)$ and its complement $B_R^\complement$.
The key observation is that: to construct a bracket $[l,u]$ for $\log f\in \log \mathcal{M}(\Theta;c)$, it suffices to construct brackets inside $B_R$ and outside $B_R$ separately. Namely, we only need to construct functions $l_{\In}(x), u_{\In}(x)$ supported in $B_R$, and $l_{\Out}(x), u_{\Out}(x)$ supported in $B_R^{\complement}$, such that
\begin{eqnarray*}
   l_{\In}(x) &\le& \log f(x) \le u_{\In}(x), \quad \forall x\in B_R;\\
  l_{\Out}(x) &\le& \log f(x) \le u_{\Out}(x), \quad \forall x\in B_R^{\complement}.
\end{eqnarray*}
By gluing 
\begin{equation} \label{eq:gluing}
    l(x):= \begin{cases}l_{\In}(x) \\ l_{\Out}(x) \end{cases}\quad,\quad
u(x):= \begin{cases} u_{\In}(x) &, \quad x\in B_R \\ u_{\Out}(x) &, \quad x\in B_R^{\complement} \end{cases}
\end{equation}
we obtain a bracket $[l,u]$ for $\log f$. Moreover, we have
\begin{eqnarray} \label{eq:braket-norm}
     \|l-u\|^2_{L^2(f_*)} &=& \int_{\R^d} |l(x)-u(x)|^2 f_*(x)\de x \nonumber \\
     &=&\int_{B_R}|l_{\In}(x)-u_{\In}(x)|^2 f_*(x)\de x + \int_{B_R^\complement} |l_{\Out}(x) - u_{\Out}(x)|^2 f_*(x) \de x.
\end{eqnarray}

On the complement $B_R^{\complement}$, by Lemma \ref{lm:bound-density-Mc}, we can choose 
\begin{equation} \label{eq:out-bra}
    \begin{cases}l_{\Out}(x) &= -\|x\|^2 - C_{\ref{lm:bound-density-Mc}} , \quad x\in B_R^{\complement} \\
    u_{\Out}(x) &= 0 \end{cases}
\end{equation}
to form an outside bracket uniformly for any $\log f \in \log \mathcal{M}(\Theta;c)$. To control the $L^2(f_*)$-norm of the outside bracket $[l_{\Out}, u_{\Out}]$, using the fact that $f_*$ has compactly supported mixing measure $\mu_*$, we show in Appendix~\ref{prf:appenvelope-noncompact} that
\begin{lemma} \label{lm:envelope-noncompact}
    Let $\|\Theta\|_{\infty}:=\sup_{\theta\in \Theta}\|\theta\|$ and 
     \[R \ge \max \Big ( \sqrt{48|\log \varepsilon|}, 2\|\Theta\|_{\infty} \Big ).\]
     There exists a constant $C_{\ref{lm:envelope-noncompact}}$ depending only on $\Theta, d$ and $c$ such that
    \[ \int_{B_R^{\complement}} (\|x\|^2+C_{\ref{lm:bound-density-Mc}})^2  f_*(x)\de x \le C_{\ref{lm:envelope-noncompact}} \varepsilon^2,\]
    where $C_{\ref{lm:bound-density-Mc}}$ is the constant in Lemma \ref{lm:bound-density-Mc}.
\end{lemma}
To conclude, for any bracket $[l_{\In}, u_{\In}]$ inside $B_R$, we can always extend to a bracket $[l,u]$ on $\R^d$ as in \eqref{eq:gluing}, where $[l_{\Out},u_{\Out}]$ is given by \eqref{eq:out-bra}. By Equation \eqref{eq:braket-norm} and Lemma \ref{lm:envelope-noncompact}, we have
\[ \|l-u\|_{L^2(f_*)}^2 \le \int_{B_R}|l_{\In}(x)-u_{\In}(x)|^2f_*(x)\de x + C_{\ref{lm:envelope-noncompact}} \varepsilon^2.\]
Therefore, we only need to control the bracketing entropy of $\log \mathcal{M}(\Theta;c)$ inside the ball $B_R$.

To this end we will use Lemma~\ref{lm:approx} which contains the central difficulty of our analysis. The main challenge arises from the fact that the logarithm function is not uniformly Lipschitz. In particular Lemma~\ref{lm:approx} states that if $\varepsilon > 0$ is small enough such that 
    \[R:= \sqrt{80|\log \varepsilon|} \ge 2  \sup_{\theta\in \Theta} \|\theta\|,\] 
    then there exist constants $C_{\ref{lm:approx}}, D_{\ref{lm:approx}}$ depending only on $\Theta, d$ and $c$, and a class of functions $\mathcal{H}$ with 
    \[\log |\mathcal{H}| \le D_{\ref{lm:approx}} |\log \varepsilon|^{d+1}\] 
    satisfying the following property: for every density $f \in \mathcal{M}(\Theta;c)$, there exists $h\in \mathcal{H}$ such that
    \[ \sup_{x\in B_R} |\log f(x) - h(x)| \le C_{\ref{lm:approx}} \varepsilon.\]

We can now proceed to complete the proof of Theorem \ref{thm:FiniteBracketingIntegral}. Let $\varepsilon>0$ be small enough such that Lemma \ref{lm:approx} holds
and $\mathcal{H}$ be the function class in Lemma \ref{lm:approx}. For each $h\in \mathcal{H}$, we define
\begin{equation}
    l_h(x) = \begin{cases}
        h(x)- C_{\ref{lm:approx}}\varepsilon \\
        -\|x\|^2 - C_{\ref{lm:bound-density-Mc}}
    \end{cases} \quad, \quad
    u_h(x) = \begin{cases}
        h(x) + C_{\ref{lm:approx}}\varepsilon &, \quad x \in B_R\\
        0 &, \quad x \in B_R^{\complement} 
    \end{cases} 
\end{equation}
where $C_{\ref{lm:bound-density-Mc}}, C_{\ref{lm:approx}}$ are constants in Lemma \ref{lm:bound-density-Mc} and Lemma \ref{lm:approx}, respectively. 

For each $f\in \mathcal{M}(\Theta;c)$, by Lemma \ref{lm:approx}, there exists $h\in \mathcal{H}$ such that $l_h \le \log f \le u_h$ on $B_R$. Combining with Lemma \ref{lm:bound-density-Mc}, we deduce that $l_h\le \log f \le u_h$ on $\R^d$. Thus, the set of brackets $\{ [l_h,u_h]: h \in \mathcal{H} \}$ covers the class $\log \mathcal{M}(\Theta;c)$. 

Moreover, by Lemma \ref{lm:envelope-noncompact}, we have for each $h\in \mathcal{H}$
\begin{eqnarray*}
    \|l_h-u_h\|_{L^2(f_*)}^2 &=& \int_{B_R} |l_h -u_h|^2f_* + \int_{B_R^{\complement}} |l_h-u_h|^2f_* \\
    &\le& 4C_{\ref{lm:approx}}^2 \varepsilon^2 + C_{\ref{lm:envelope-noncompact}} \varepsilon^2 =: \tilde\varepsilon^2
\end{eqnarray*}
where $\tilde\varepsilon := (4C_{\ref{lm:approx}}^2 + C_{\ref{lm:envelope-noncompact}})^{1/2} \varepsilon$.
Hence, each $[l_h,u_h]$ is an $\tilde\varepsilon$-bracket in $L^2(f_*)$.

Since $\log |\mathcal{H}| \le  D_{\ref{lm:approx}} |\log \varepsilon|^{d+1}$, there are at most $\exp (D_{\ref{lm:approx}} |\log \varepsilon|^{d+1})$ many such brackets. Therefore
\[ \log N_{[]} \Big (\tilde \varepsilon, \log \mathcal{M}(\Theta;c) ,L^2(f_*) \Big ) \lesssim |\log \varepsilon|^{d+1} \lesssim |\log \tilde\varepsilon|^{d+1}, \]
where the implicit constants depend only on $\Theta,d $ and $c$. 
The proof is thus complete. $\square$
 \label{app:logbrack}

\subsection{Proofs of technical lemmas in subsection \ref{app:logbrack}}
\subsubsection{Proof of Lemma \ref{lm:bound-density-Mc}} \label{prf:appbound-density-Mc}

\begin{proof}[Proof of Lemma \ref{lm:bound-density-Mc}]
    For each $f\in \mathcal{M}(\Theta;c)$, we have
\begin{eqnarray*}
f(x) &=& \frac{1}{(2\pi)^{d/2}} \int e^{-\|x-\theta\|^2/2} \mu_f(\de \theta)\\
&\ge& \frac{1}{(2\pi)^{d/2}} \int_{\Theta} e^{-\|x-\theta\|^2/2} \mu_f(\de \theta) \\
&\ge& \frac{c}{(2\pi)^{d/2}} \cdot e^{-\|x\|^2 - \|\Theta\|^2_{\infty}}, 
\end{eqnarray*}
where $\|\Theta\|_{\infty}:= \sup_{\theta \in \Theta} \|\theta\|$. Thus $-\log f(x) \le \|x\|^2 + C_{\ref{lm:bound-density-Mc}}$ for some constant $C_{\ref{lm:bound-density-Mc}}$ depending only on $\Theta, d$ and $c$, as desired.
\end{proof}

\subsubsection{Proof of Lemma \ref{lm:envelope-noncompact}}\label{prf:appenvelope-noncompact}

\begin{proof}[Proof of Lemma \ref{lm:envelope-noncompact}]
    Let $X\sim f_*$. We have
    \begin{eqnarray*}
        \int_{B_R^\complement} (\|x\|^2 + C_{\ref{lm:bound-density-Mc}})^2 f_*(x)\de x &=& \mathbb{E} [(\|X\|^2 + C_{\ref{lm:bound-density-Mc}})^2 \mathbf{1}(\|X\| >R)] \\
        &\le& \mathbb{E}[(\|X\|^2 + C_{\ref{lm:bound-density-Mc}})^4]^{1/2} \cdot \mathbb{P}(\|X\|>R)^{1/2}.
    \end{eqnarray*}
   Note that we can write $X=Y+Z$, where $Y\sim \mu_*$ and $Z$ is a standard Gaussian in $\R^d$. Since $\mu_*$ is supported in the compact set $\Theta$, we have
    \[ \|X\| \le \|Y\|+\|Z\| \le \|\Theta\|_{\infty} + \|Z\| \quad \text{a.s.}\]
    Thus, $ \mathbb{E}[(\|X\|^2 + C_{\ref{lm:bound-density-Mc}})^4] = O(1)$. Moreover, using $R\ge 2\|\Theta\|_{\infty}$, we have
    \[\mathbb{P}(\|X\|>R) \le \mathbb{P}(\|Z\| >R/2) \lesssim e^{-R^2/12} \le e^{-4\log \varepsilon} = \varepsilon^4,\]
    where we used the fact that $R\ge \sqrt{48|\log \varepsilon|}$ and the inequality
    \[ \mathbb{P}(\|Z\|\ge x) \lesssim e^{-x^2/3}, \quad \forall x\ge 0,\]
    for a $d$-dimensional standard Gaussian vector $Z$.
    Thus,
    \[  \int_{B_R^\complement} (\|x\|^2 + C_{\ref{lm:bound-density-Mc}})^2 f_*(x)\de x \lesssim \varepsilon^2\]
    as desired.
\end{proof}

\subsubsection{Proof of Lemma \ref{lm:approx}}\label{prf:applmapprox}

To prove Lemma~\ref{lm:approx}, we can begin by trying to follow in the footsteps of \cite{ghosal2007posterior},\cite{zhang2009generalized} and \cite{saha2020nonparametric} which construct brackets through successively finer approximations. 

For the rest of this proof, for any measure $\mu$ we will use $f_{\mu}$ to denote its density.
Consider some $f_{\psi} \in \pspace{c}{\Theta}$ and let $\psi = \pi*\phi$ be the underlying measure whose density is $f_{\psi}$. 
We will approximate $\psi$ by a measure $\tilde{\psi} = \tilde{\pi}*\phi$ whose mixing measure $\tilde{\pi}$ has a finite number of atoms. This is where we encounter another difficulty. Existing results (e.g., those in \cite{saha2020nonparametric}) discretize $\pi$ to a measure $\tilde{\pi}$ supported on a set roughly of radius $a$ and such that the discrepancy of the density $|f_{\psi} - f_{\tilde{\psi}}|$ is roughly of order $e^{-a^2/2}$. Consequently the value of $a$ is then chosen appropriately to make the discrepancy $|f_{\psi} - f_{\tilde{\psi}}|$ as small as required. This however is problematic for us. This is because to control the logarithms of the densities, we need the density discrepancies not only to be small, but significantly smaller than the densities themselves. However since a density with gaussian decay can become as small as $e^{-a^2/2}$ when restricted to the ball of radius $a$, we cannot naively use the discretization results in \cite{saha2020nonparametric}. 

Therefore we need a finer approach. To do this we will split the measure $\pi$ into two parts: the measure derived by restricting $\pi$ to the compact set $\Theta$ denote by $\pi^{\Theta}$ and the part outside (denoted by $\pi^{\comp{\Theta}}$). We will discretize (i.e., create an approximate measure which is discrete) each measure $\pi^{\Theta}$ and $\pi^{\comp{\Theta}}$ separately and then put the discretized measures back together to construct the final discretization. 

To discretize the measure $\pi^{\Theta}$ we will consider the following Proposition.

\begin{proposition}\label{prop:discretizationcompact}
 Let $\rho$ be some probability measure supported on a compact set $\Theta$.  
 Let $\varepsilon \in (0,1)$ be small enough such that $R := \sqrt{80\log (1/\varepsilon)} \geq 2\norm{\Theta}_{\infty}$. Let $B_R$ be the ball of radius $R$ centered at $0$. Then there is a constant $D_{\ref{prop:discretizationcompact}}$ dependent only on $\Theta$ and $d$ such that there exists a discrete probability measure, $\tilde{\rho}$, with at most $D_{\ref{prop:discretizationcompact}}(\log \frac{1}{\varepsilon})^d$ atoms and supported on $\Theta$ satisfying 
 \begin{align}\label{eq:ideas-compactfdiff} \sup_{x \in B_R} \abs{f_{\rho * \phi} (x) - f_{\tilde{\rho}*\phi}(x)} \leq C_{\ref{prop:discretizationnoncompact}}\varepsilon^{81} 
 \end{align}
\end{proposition}

The detailed proof of Proposition~\ref{prop:discretizationcompact} is in Appendix~\ref{prf:appdiscretizationcompact}. Using Proposition~\ref{prop:discretizationcompact} with $\rho = \pi^{\Theta}$, we will construct the approximate discrete distribution $\tilde{\pi}_1 := \tilde{\rho}$. The crucial point here is that $\tilde{\pi}_1$ is also supported on the compact set $\Theta$ and not on a fattening of $\Theta$. This is in contrast to the Proposition that we will use for the measure $\pi^{\comp{\Theta}}$ stated below. 

\begin{proposition} \label{prop:discretizationnoncompact} 
Let $B_R$ be the ball of radius $R:= \sqrt{80\log (1/\varepsilon)}$ centered at $0$. Let $a = \sqrt{\frac{81}{40}}R$.
There exists a constant $D_{\ref{prop:discretizationnoncompact}}\in \R_{>0}$ (depending on $d$ and $\Theta$) such that 
for any measure $\eta$ there exists a discrete probability distribution $\tilde{\eta}$ with at most $D_{\ref{prop:discretizationnoncompact}} (\log \frac{1}{\varepsilon})^d$ 
atoms and supported on $B_R^a$ such that for some constant $C_{\ref{prop:discretizationnoncompact}}$ depending on $\Theta$ and $d$ we have  
\begin{align}\label{eq:ideas-noncompactfdiff}
\sup_{x\in B_R} \abs{f_{\eta*\phi}(x) - f_{\tilde{\eta}*\phi}(x)} \leq C_{\ref{prop:discretizationnoncompact}}\varepsilon^{81}.\end{align}
\end{proposition}

{Proposition~\ref{prop:discretizationnoncompact} is a consequence of existing results.} We defer the full proof of Proposition~\ref{prop:discretizationnoncompact} to the Appendix~\ref{prf:appdiscretization}. We will use Proposition~\ref{prop:discretizationnoncompact} with $\eta = \pi^{\comp{\Theta}}$ to construct the discrete measure $\tilde{\pi}_2 := \tilde{\eta}$, a discrete approximation of $\pi^{\comp{\Theta}}$. The point here is that we could not have used Proposition~\ref{prop:discretizationnoncompact} to discretize $\pi^{\Theta}$, as that would have pushed the support into the fattened set $B_{R}^a$. This would for instance mean that the density at $0$ could be as small as $e^{-a^2/2} \approx \varepsilon^{81}$ which would cancel the right hand side of equation~\eqref{eq:ideas-noncompactfdiff} when applied to the inequality in Proposition~\eqref{prop:loglips}. (In fact in reality if we were to only use Proposition~\ref{prop:discretizationnoncompact}, the density could be much smaller than in this simple example). 

We will now combine the discrete measures $\tilde{\pi}_1$ and $\tilde{\pi}_2$ to get $\tilde{\pi}$ using the following construction:
\begin{align}\label{eq:ideas-pitilconstruct}
 \tilde{\pi} = \pi(\Theta)\tilde{\pi}_1 +\pi(\comp{\Theta})\tilde{\pi}_2 
\end{align}

Now as $f_{\psi} \in \pspace{c}{\Theta}$, we have $\pi(\Theta) > c$. Thus by the construction in \eqref{eq:ideas-pitilconstruct}, $\tilde{\pi}$ also satisfies $\tilde{\pi}(\Theta) > c$. This fact will now allow us to appropriately lower bound the density of $\tilde{\psi} := \tilde{\pi}*\phi$. 

Finally we will move from the densities to the logarithm of the densities using the following simple inequality:
\begin{proposition}\label{prop:loglips}
 Let $\log x$ denote the natural logarithm of $x$. Then for any positive reals $u,v \in \mathbb{R}_{>0}$ we have
 \begin{align*} 
 \abs{\log u - \log v} \leq \frac{\abs{u-v}}{\min(u,v)}. 
 \end{align*} 
\end{proposition}

A proof of this has been included in Appendix~\ref{prf:apploglips} for posterity. Combining Propositions~\ref{prop:discretizationcompact}, \ref{prop:discretizationnoncompact} and \ref{prop:loglips} we will prove the following Proposition whose proof is deferred to Appendix~\ref{prf:applogdiscretization}.

\begin{proposition}\label{prop:logdiscretization}
 Let $\varepsilon \in (0,1)$ be small enough such that $R := \sqrt{80}\sqrt{\log (1/\varepsilon)} \geq 2\norm{\Theta}_{\infty}$. Let $B_R$ be the ball of radius $R$. Let $\psi =\pi*\phi$ be a measure such that its density, $f_{\psi} \in \pspace{c}{\Theta}$. Then there exists a measure $\tilde{\psi} = \tilde{\pi}*\phi$ with its density, $f_{\tilde{\psi}} \in \pspace{c}{\Theta}$, such that for some constants $C_{\ref{prop:logdiscretization}}, D_{\ref{prop:logdiscretization}}$ depending only on $\Theta$, $c$ and  $d$ such that the following are true 
 \begin{enumerate} 
 \item \label{eq:logdiscnoofatoms} $\tilde{\pi}$ is a discrete measure supported on $B_R^a$ with at most $D_{\ref{prop:logdiscretization}}\left(\log \frac{1}{\varepsilon}\right)^d$ atoms. 
 \item \label{eq:logdiscdiffbound}  
 $\sup_{x \in B_R} \abs{\log f_{\psi}(x) - \log f_{\tilde{\psi}}(x)} \leq C_{\ref{prop:logdiscretization}}\varepsilon$ 
 \end{enumerate}
\end{proposition}

To complete the proof we will then exploit the fact that the discrete measure $\Tilde{\pi}$ has a limited number of atoms and thus is parameterizable. Then for some suitably chosen $\delta$, we will carefully construct a $\delta$-lattice over this space of parameters, and use this lattice along with the envelope function $H$ to finally construct the requisite brackets. In particular we use the following Proposition whose proof is deferred to Appendix~\ref{prf:applattice_cons}.:

\begin{proposition}\label{prop:lattice_cons}
 Let $B_R$ and $B_R^a$ be as defined in Proposition~\ref{prop:discretizationnoncompact}. Let $\varepsilon$ be small enough such that $R \geq 2\norm{\Theta}_{\infty}$.  Then there exists constants $D_{\ref{prop:lattice_cons}}$, $C_{\ref{prop:lattice_cons}}$ and a set of functions $\mc{H}$ whose cardinality is at most $\exp \left(D_{\ref{prop:lattice_cons}}\left(\log \frac{1}{\varepsilon}\right)^{d+1}\right)$ and for which the following is true: 
 Let $\tilde{\psi} = \tilde{\pi}*\phi$ with density $f_{\tilde{\psi}} \in \pspace{c}{\Theta}$ be such that $\tilde{\pi}$ has atmost $D_{\ref{prop:logdiscretization}}\left(\log \frac{1}{\varepsilon}\right)^d$ atoms and $\text{Supp}(\tilde{\pi}) \subseteq B_R^a$. Then $\exists h \in \mc{H}$ such that 
 \begin{align*} 
 \sup_{x \in B_R} \abs{h(x) - \log f_{\tilde{\psi}}(x)} \leq C_{\ref{prop:lattice_cons}} \varepsilon
 \end{align*} 
\end{proposition}

Finally we can combine all the above propositions by a simple triangle inequality to prove Lemma~\ref{lm:approx}.
\begin{proof}[Proof of Lemma~\ref{lm:approx}]
 By Proposition~\ref{prop:logdiscretization} for any density $f \in \pspace{c}{\Theta}$ there exists a density $f_{\tilde{\psi}} \in \pspace{c}{\Theta}$ such that 
 \begin{align} \label{eq:logdensdiffdisc} 
 \sup_{x \in B_R} \abs{\log f_{\psi}(x) - \log f_{\tilde{\psi}}(x)} \leq C_{\ref{prop:logdiscretization}}\varepsilon. 
 \end{align} 
 Additionally by Proposition~\ref{prop:logdiscretization} we can write $\tilde{\psi} = \tilde{\pi}*\phi$ such that $\tilde{\pi}$ is a discrete measure supported on $B_R^a$ with at most $D_{\ref{prop:logdiscretization}}\left(\log \frac{1}{\varepsilon}\right)^d$ atoms. Thus we can use Proposition~\ref{prop:lattice_cons} to conclude that there is some function $h$ in the set of functions $\mc{H}$ (as defined in Proposition~\ref{prop:lattice_cons}) such that 
 \begin{align} \label{eq:logdensdifflattice} 
 \sup_{x \in B_R} \abs{h(x) - \log f_{\tilde{\psi}}(x)} \leq C_{\ref{prop:lattice_cons}} \varepsilon 
 \end{align} 

 Thus by triangle inequality and inequalities~\eqref{eq:logdensdiffdisc} and \eqref{eq:logdensdifflattice} for the constant $C_{\ref{lm:approx}} = C_{\ref{prop:logdiscretization}} + C_{\ref{prop:lattice_cons}}$, we have
 \begin{align*}
 \sup_{x\in B_R} |\log f(x) - h(x)| \le C_{\ref{lm:approx}} \varepsilon. 
 \end{align*} 
 Finally we note that by Proposition~\ref{prop:lattice_cons} the set of functions $\mc{H}$ satisfies 
 \begin{align*} 
 \log \abs{\mc{H}} \leq D_{\ref{prop:lattice_cons}}\left(\log \frac{1}{\varepsilon}\right)^{d+1}. 
 \end{align*} 
 This concludes the proof of Lemma~\ref{lm:approx}.
\end{proof}

\subsubsection{Proof of Proposition~\ref{prop:discretizationcompact}}\label{prf:appdiscretizationcompact}

To prove Proposition~\ref{prop:discretizationcompact} we will first need the following simple proposition.
\begin{proposition}\label{prop:momentmatchTheta} Let $h_1,\cdots,h_m : K \to \R$ be continuous functions 
on a compact metric space $K$. For any probability measure $\pi$ supported on $K$
, there exists a discrete probability measure $\tilde{\pi}$ 
on $K$ with at most $m+1$ atoms such that $\E_{X\sim\pi}[h_j(X)] = \E_{\Tilde{X}\sim\tilde{\pi}}[h_j(\Tilde{X})]$ for all $1\leq j \leq m$.
\end{proposition}

\begin{proof} 
Begin by noting that the vector \[v := (\E_{X\sim\pi}[h_1(X)],\cdots,\E_{X\sim\pi}[h_m(X)]) \in \R^m,\] 
is contained in the convex hull of the compact set
\begin{align}\nonumber C:= \big\{ (h_1(x),\cdots,h_m(x)) \ : \ x\in K\big\}.\end{align}
It then follows from Carath\'{e}odory's theorem that $v$ 
is a convex combination of (at most) $m+1$ elements of $C$. 
Thus there there exists $x_1,\cdots,x_{m+1}\in K$ and $\lambda_1,\cdots,\lambda_{m+1}\in \R_{\geq 0}$ 
 such that \begin{align*}
 \sum_{i=1}^{m+1} \lambda_i = 1 
 \end{align*} and
\begin{align}\label{eq:hcaratheodory}(\E_\pi[h_1(X)],\cdots,\E_\pi[h_m(X)]) 
&= \sum_{i=1}^{m+1} \lambda_i\cdot (h_1(x_i),\cdots,h_m(x_i)). \end{align}
Then $\displaystyle\tilde{\pi} := \sum_{i=1}^{m+1}\lambda_i \cdot \delta_{x_i}$ is a discrete probability measure on $K$ with at most $m+1$ atoms such that by equation~\eqref{eq:hcaratheodory} for any $j \leq m$,
\begin{align*}
 \E_{X\sim\pi}[h_j(X)] = \sum_{i=1}^{m+1} \lambda_i\cdot h_j(x_i) = \E_{\Tilde{X}\sim\tilde{\pi}}[h_j(\Tilde{X})]
\end{align*}
\end{proof}

We can now continue on to the proof of Proposition~\ref{prop:discretizationcompact}.
\begin{proof}[Proof of Proposition~\ref{prop:discretizationcompact}] Let $t\in \Z_{\geq 1}$ be an integer whose exact value will be determined later. 

Consider the exponential function $g(y) := e^{-y}$. As the (mean-value) remainder in the Taylor expansion of order $t-1$ of $g$ at the origin is bounded above by $\frac{|y|^t}{t!}\leq \frac{(e|y|)^t}{t^t}$, 
there exists a polynomial function $P_{t-1} : \R \to \R$ of degree $t-1$ such that for all $y\in \R$,
\begin{align*}|g(y) - P_{t-1}(y)| \leq \frac{(e|y|)^t}{t^t}\end{align*}

Then for any $x\in \R^d$ by setting $y = \norm{x}^2/2$ in the equation above it follows that,
\begin{align} \label{eqn:Nresidue}
\bigg|e^{-\normnorm{x}^2/2} - P_{t-1}\bigg(\frac{\normnorm{x}^2}{2}\bigg)\bigg| \leq \bigg(\frac{e\normnorm{x}^2}{2t}\bigg)^t. 
\end{align}

Now define the indexing set $I$ as
\begin{align}
 \nonumber I = \{k = (k_1, ..., k_d) \in \mathbb{Z}^d : \forall 1\leq j\leq d \text{ we have } 0 \leq k_j \leq 2t-1 \}
\end{align}

Then for any $k \in I$ define the functions \begin{align}\nonumber h_{k}(x_1, ..., x_d) := x_1^{k_1} \cdot x_2^{k_2} \cdots x_d^{k_d}.\end{align} Note that as there are a total of $(2t)^d$ elements in $I$ we only have $(2t)^d$ many functions $h_k$. 

Thus by proposition \ref{prop:momentmatchTheta} with $K = \Theta$ and the functions $h$ as above, there exists a discrete probability measure $\tilde{\rho}$ 
on $\Theta$ with at most $(2t)^d+1$ atoms such that if $X = (X_1,\cdots,X_d) \sim \rho$ 
and $\Tilde{X} = (\Tilde{X}_1,\cdots,\Tilde{X}_d) \sim \tilde{\rho}$, then for every $k \in I$,
\begin{align}\label{eq:hkexpidentity}\E_{X}[h_k(X_1, ..., X_d)] = \E_{\Tilde{X}}[h_k(\Tilde{X}_1,\cdots,\Tilde{X}_d)].\end{align}

On the other hand as $P_{t-1}$ is a polynomial with degree $t$ and $\norm{x}^2 = \sum_{i=1}^d x_i^2$, we must have that $P_{t-1}(\norm{x}^2/2)$ viewed as a function of $x_1, .., x_d$ is a linear combination of the set of functions $\{h_k\}_{k\in I}$.

The above observation combined with equation~\eqref{eq:hkexpidentity} implies that:
\begin{align}\int_{\Theta} P_{t-1}\bigg(\frac{\normnorm{x-\theta}^2}{2}\bigg) \ d\rho(\theta) = 
\int_{\Theta} P_{t-1}\bigg(\frac{\normnorm{x-\theta}^2}{2}\bigg) \ d\tilde{\rho}(\theta). \end{align} 
It then follows that for all $x\in B_R$,
\begin{align}\abs{f_{\rho*\phi} (x) - f_{\tilde{\rho}*\phi}(x)} &\leq \frac{1}{(2\pi)^{d/2}}\bigg|\int_{\Theta} e^{-\normnorm{x-\theta}^2/2} \ d(\rho-\tilde{\rho})(\theta) \bigg| \\
&=\frac{1}{(2\pi)^{d/2}} \bigg|\int_{\Theta} e^{-\normnorm{x-\theta}^2/2} - P_{t-1}\bigg(\frac{\normnorm{x-\theta}^2}{2}\bigg) \ d(\rho - \tilde{\rho})(\theta)\bigg| \\
&\leq \frac{1}{(2\pi)^{d/2}}\cdot 2\sup_{\substack{x\in B_R\\ u\in \Theta}} 
\bigg|e^{-\normnorm{x-\theta}^2/2} - P_{t-1}\bigg(\frac{\normnorm{x-\theta}^2}{2}\bigg)\bigg|.\end{align}

Then using equation~\eqref{eqn:Nresidue} into the inequality above we get that for the constant $C_{\ref{prop:discretizationnoncompact}} = 2\cdot \frac{1}{(2\pi)^{d/2}}$
\begin{align} \label{eq:tmp1}
 \abs{f_{\rho*\phi} (x) - f_{\tilde{\rho}*\phi}(x)} \leq C_{\ref{prop:discretizationnoncompact}} \left(\frac{e\norm{x-\theta}^2}{2t}\right)^t
\end{align}

As $R > \norm{\Theta}_{\infty}$ we have $\norm{x-\theta} \leq 2R = 2\sqrt{80}\sqrt{\log(1/\varepsilon)}.$ 

Using the above observation along with equation~\eqref{eq:tmp1} we get that 
\begin{align}\nonumber\abs{f_{\rho*\phi} (x) - f_{\tilde{\rho}*\phi}(x)} 
&\leq C_{\ref{prop:discretizationnoncompact}}\frac{(160e)^t (\log\frac{1}{\varepsilon})^t}{t^t} 
\\ \label{eq:discretizationcmpcttmp1} &= C_{\ref{prop:discretizationnoncompact}}\exp\big[- t(\log t - \log 160e + \log \varepsilon)\big].\end{align}
Let $C_1 = \max(81, 160e^2)$, and choose $t = \left\lceil C_1\log(1/\varepsilon)\right\rceil$. 
Plugging these values into inequality~\eqref{eq:discretizationcmpcttmp1}
\begin{align}\nonumber \abs{f_{\rho*\phi} (x) - f_{\tilde{\rho}*\phi}(x)} &\leq C_{\ref{prop:discretizationnoncompact}}\exp(-t(\log C_1 - \log 160e))
\\ \nonumber &\leq C_{\ref{prop:discretizationnoncompact}}\exp(-t(\log 160e^2 - \log 160e))
\\ \nonumber &\leq C_{\ref{prop:discretizationnoncompact}}\exp(-t) 
\\ \nonumber &= C_{\ref{prop:discretizationnoncompact}}\exp(- C_1\log(1/\varepsilon)) 
\\ &\leq C_{\ref{prop:discretizationnoncompact}}\varepsilon^{81}. \label{eq:densitydiffdiscretecompact} \end{align}

Inequality~\eqref{eq:densitydiffdiscretecompact} then ensures that $\tilde{\rho}$ satisfies the condition~\eqref{eq:ideas-compactfdiff}.

To complete the argument then we need to show that for some constant $D_{\ref{prop:discretizationcompact}}$, $\tilde{\rho}$ has at most $D_{\ref{prop:discretizationcompact}}(\log \frac{1}{\varepsilon})^d$ atoms. To this end let us remember that $\tilde{\rho}$ was constructed above equation~\eqref{eq:hkexpidentity} using proposition\ref{prop:momentmatchTheta} and hence had at most $(2t)^d+1$ atoms. Since we have set $t = \left\lceil C_1\log(1/\varepsilon)\right\rceil$, $\tilde{\rho}$ has at most $(2C_1)^d(\log \frac{1}{\varepsilon})^d$ atoms. The proof is then complete for $D_{\ref{prop:discretizationcompact}} = (2C_1)^d+1$.
\end{proof}

\subsubsection{Proof of Proposition~\ref{prop:discretizationnoncompact}}\label{prf:appdiscretization}
To prove Proposition~\ref{prop:discretizationnoncompact} we will need the following proposition appearing in the appendix as Lemma D.3 from \cite{saha2020nonparametric}. In the following Lemma, $N(r,S)$ denotes the minimum number of $r$ Euclidean balls needed to cover $S$.
\begin{proposition}[Lemma D.3 in \cite{saha2020nonparametric}]\label{prop:discretizationsaha}
 Let $G$ be an arbitrary probability measure on $\mathbb{R}^d$ and let $S$ denote an arbitrary compact subset of $\mathbb{R}^d$. Also let $a \geq 1.$
 Then there exists a discrete probability measure $G'$ that is supported on 
 $S^a$ and having at most
 \begin{align*}
 l := (2\lfloor(13.5)a^2\rfloor + 2)^d N(a, S^a) + 1
 \end{align*}
 atoms such that
 \begin{align*}
 \sup_{x\in S} \abs{f_{G*\phi}(x) - f_{G'*\phi}(x)} \leq \left(1 + \frac{1}{\sqrt{2\pi}}\right)(2\pi)^{-d/2}e^{-a^2/2}.
 \end{align*} 
\end{proposition}

\begin{proof}[Proof of Proposition~\ref{prop:discretizationnoncompact}]
Using Proposition~\ref{prop:discretizationsaha} for $a = \sqrt{\frac{81}{40}}R$, $S = B_R$ and $G =\eta$, we have a discrete measure supported on $B_R^a$, denoted as $\Tilde{\eta}$ (or $G'$ in Proposition~\ref{prop:discretizationsaha}), having at most 
\begin{align}\label{eq:atomsupper}
 l := (2\lfloor(13.5)a^2\rfloor + 2)^d N(a, B_R^a) + 1
 \end{align}
 atoms such that
 \begin{align}\label{eq:discretdensdiff}
 \sup_{x\in B_R} \abs{f_{\eta*\phi}(x) - f_{\tilde{\eta}*\phi}(x)} \leq \left(1 + \frac{1}{\sqrt{2\pi}}\right)(2\pi)^{-d/2}e^{-a^2/2}.
 \end{align}

Plugging the value of $a =\sqrt{\frac{81}{40}}R = \sqrt{162\log \frac{1}{\varepsilon}}$ into equation~\eqref{eq:discretdensdiff} we get that for the constant $C_{\ref{prop:discretizationnoncompact}} = \left(1 + \frac{1}{\sqrt{2\pi}}\right)(2\pi)^{-d/2}$, we have for any $x \in B_R$,
\begin{align}\label{eq:discretizationtmp1}
 \abs{f_{\eta*\phi}(x) - f_{\tilde{\eta}*\phi}(x)} \leq C_{\ref{prop:discretizationnoncompact}} e^{-\frac{a^2}{2}} \leq C_{\ref{prop:discretizationnoncompact}} \varepsilon^{81}.
\end{align}
Let us now see why the number of atoms is of the order of $(-\log \varepsilon)^d$. To this end note 
the minimum number of $r$-euclidean balls needed to cover $S$ is of the order $\left(\frac{\norm{S}_{\infty}}{r}\right)^d$. In particular there is some constant $C_1$ such that
\begin{align}\nonumber
 N(a,B_R^a) \leq C_1\left(\frac{\norm{B_R^a}_{\infty}}{a}\right)^d.
\end{align}
But $\norm{B_R^a}_{\infty} \leq R + a.$ Thus we have
\begin{align}\label{eq:Nbound1}
 N(a,B_R^a) \leq C_1\left(\frac{R + a}{a}\right)^d \leq C_1 \left(\frac{\left(1+\sqrt{\frac{81}{40}}\right)R}{\sqrt{\frac{81}{40}} R}\right)^d = C_1 \left(\frac{\left(1+\sqrt{\frac{81}{40}}\right)}{\sqrt{\frac{81}{40}}}\right)^d.
\end{align}

Note that $a = \sqrt{\frac{81}{40}}R = \sqrt{162 \log \frac{1}{\varepsilon}}$. Thus plugging the above equation~\eqref{eq:Nbound1} into \eqref{eq:atomsupper} we get that for some constant $D_{\ref{prop:discretizationnoncompact}}$, we have that the number of atoms of $\Tilde{\eta}$ is at most $D_{\ref{prop:discretizationnoncompact}}^d(-\log \varepsilon)^d$ as required.

\end{proof}

\subsubsection{Proof of Proposition~\ref{prop:loglips}}\label{prf:apploglips}
\begin{proof}[Proof of Proposition~\ref{prop:loglips}]
 WLOG assume $u \leq v$. Then simply note that 
 \begin{align}\nonumber
 \frac{d \log x}{dx} = \frac{1}{x}.
 \end{align}
 Therefore we have
 \begin{align*}
 \abs{\log u - \log v} \leq \abs{\int_{u}^v \frac{d \log x}{dx}~dx} \leq \int_{u}^v \abs{\frac{1}{x}}~dx = \int_{u}^v\frac{1}{x}~dx \leq \frac{\abs{v-u}}{\min(u,v)}.
 \end{align*}
\end{proof}

\subsubsection{Proof of Proposition~\ref{prop:logdiscretization}}\label{prf:applogdiscretization}

\begin{proof}[Proof of Proposition~\ref{prop:logdiscretization}]
 We will first split the measure $\pi$ into two parts, the part supported on $\Theta$ and the part outside. To do this formally define the measure $\pi^{\Theta}$ as
 \begin{align}\label{eq:pithetadefn}
 \pi^{\Theta}(A) = \frac{\pi(A\cap \Theta)}{\pi(\Theta)}. 
 \end{align}
 Similarly if $\comp{\Theta}$ denotes the set $\mathbb{R}^d\exclude \Theta$, define the measure $\pi^{\comp{\Theta}}$ as
 \begin{align}\label{eq:pinotthetadefn}
 \pi^{\comp{\Theta}}(A) = \frac{\pi(A\cap \comp{\Theta})}{\pi(\comp{\Theta})}. 
 \end{align}
 Note that combining equations~\eqref{eq:pithetadefn} and \eqref{eq:pinotthetadefn} we get
 \begin{align}
 \pi = \pi(\Theta) \pi^{\Theta} + \pi(\comp{\Theta}) \pi^{\comp{\Theta}}.
 \end{align}
 Thus if $\varphi$ denotes the density of the standard normal distribution we have,
 \begin{align}
 \nonumber f_{\psi}(x) &= \mathbb{E}_{\theta\sim\pi}\left[\varphi(x - \theta) \right]
 \\\nonumber &= \pi(\Theta)\mathbb{E}_{\theta_1\sim\pi^{\Theta}}\left[\varphi(x - \theta_1) \right] + \pi(\comp{\Theta})\mathbb{E}_{\theta_2\sim\pi^{\comp{\Theta}}}\left[\varphi(x - \theta_2) \right]
 \\&= \pi(\Theta)f_{\pi^{\Theta}*\phi}(x) + \pi(\comp{\Theta})f_{\pi^{\comp{\Theta}}*\phi}(x).\label{eq:fpsiconvexdecomp}
 \end{align}
 
 Now we note that the measure $\pi^{\Theta}$ is supported in $\Theta$. As such we can use Proposition~\ref{prop:discretizationcompact} (with $\rho = \pi^{\Theta}$) to conclude that there exists a discrete measure $\tilde{\pi}_1$ supported on $\Theta$ such that $\tilde{\pi}_1$ has at most $D_{\ref{prop:discretizationcompact}}(\log \frac{1}{\varepsilon})^d$ atoms and satisfying 
 \begin{align}\label{eq:logdisccompactfdiff} \sup_{x \in B_R} \abs{f_{\pi^{\Theta} * \phi} (x) - f_{\tilde{\pi}_1*\phi}(x)} \leq C_{\ref{prop:discretizationnoncompact}}\varepsilon^{81}.
 \end{align}

 Similarly then we can use Proposition~\ref{prop:discretizationnoncompact} (with $\eta = \pi^{\comp{\Theta}}$) to conclude that there exists a discrete measure $\tilde{\pi}_1$ supported on $B_R^a$ such that $\tilde{\pi}_2$ has at most $D_{\ref{prop:discretizationnoncompact}}(\log \frac{1}{\varepsilon})^d$ atoms and satisfying 
 \begin{align}\label{eq:logdiscnoncompactfdiff} \sup_{x \in B_R} \abs{f_{\pi^{\comp{\Theta}}*\phi}(x) - f_{\tilde{\pi}_2*\phi}(x)} \leq C_{\ref{prop:discretizationnoncompact}}\varepsilon^{81}.
 \end{align}
 Then define the measure $\tilde{\pi}$ as
 \begin{align}\label{eq:logdisc-pitildedefn}
 \tilde{\pi} = \pi(\Theta)\tilde{\pi}_1 + \pi(\comp{\Theta})\tilde{\pi}_2.
 \end{align}
 
 We note that $\tilde{\pi}$ has at most as many atoms as the sum of the number of atoms in $\tilde{\pi}_1$ and $\tilde{\pi}_2$. As the $\tilde{\pi}_1$ and $\tilde{\pi}_2$ have at most $D_{\ref{prop:discretizationcompact}}(\log \frac{1}{\varepsilon})^d$ and $D_{\ref{prop:discretizationnoncompact}}(\log \frac{1}{\varepsilon})^d$ atoms respectively. We can conclude that for the constant $D_{\ref{prop:logdiscretization}} = D_{\ref{prop:discretizationcompact}} + D_{\ref{prop:discretizationnoncompact}}$, the measure $\tilde{\pi}$ has atmost $D_{\ref{prop:logdiscretization}}\left(\log \frac{1}{\varepsilon}\right)^d$ atoms. Thus $\tilde{\pi}$ satisfies item~\ref{eq:logdiscnoofatoms} of the Proposition~\ref{prop:discretizationnoncompact}.

 To prove item~\ref{eq:logdiscdiffbound}, define $\tilde{\psi} = \tilde{\pi}*\phi$. Then we have,
 \begin{align}
 \nonumber f_{\tilde{\psi}}(x) &= \mathbb{E}_{\theta\sim\tilde{\pi}}\left[\varphi(x - \theta) \right]
 \\\nonumber &= \pi(\Theta)\mathbb{E}_{\theta_1\sim\tilde{\pi}_1}\left[\varphi(x - \theta_1) \right] + \pi(\comp{\Theta})\mathbb{E}_{\theta_2\sim\tilde{\pi}_2}\left[\varphi(x - \theta_2) \right]
 \\&= \pi(\Theta)f_{\tilde{\pi}_1*\phi}(x) + \pi(\comp{\Theta})f_{\tilde{\pi}_2*\phi}(x). \label{eq:ftildepsiconvexdecomp}
 \end{align}

 Thus using equations~\eqref{eq:fpsiconvexdecomp} and \eqref{eq:ftildepsiconvexdecomp} we get
 \begin{align}\label{eq:fdiffdecomp}
 \abs{f_{\psi}(x) - f_{\tilde{\psi}}(x)} \leq \pi(\Theta)\abs{f_{\pi^{\Theta}*\phi}(x)-f_{\tilde{\pi}_1*\phi}(x)} + \pi(\comp{\Theta})\abs{f_{\pi^{\comp{\Theta}}*\phi}(x)-f_{\tilde{\pi}_2*\phi}(x)}. 
 \end{align}
 Then plugging in inequalities~\eqref{eq:logdisccompactfdiff} and \eqref{eq:logdiscnoncompactfdiff} into the above inequality \eqref{eq:fdiffdecomp} we get for any $x \in B_R$
 \begin{align}\label{eq:fdiffineq1}
 \abs{f_{\psi}(x) - f_{\tilde{\psi}}(x)} \leq (\pi(\Theta)C_{\ref{prop:discretizationnoncompact}} + \pi(\comp{\Theta})C_{\ref{prop:discretizationnoncompact}})\varepsilon^{81}.
 \end{align}

 Now as $\psi = \pi *\phi \in \pspace{c}{\Theta}$, we have that 
 \begin{align}\label{eq:logdisc-pithetac}
 \pi(\Theta) \geq c.
 \end{align}

 Remember that $\tilde{\pi}_1$ was constructed via Proposition~\ref{prop:discretizationcompact} and is thus supported on $\Theta$. Thus by the construction of the measure $\tilde{\pi}$ in equation~\eqref{eq:logdisc-pitildedefn} and \eqref{eq:logdisc-pithetac} we have
 \begin{align*}
 \tilde{\pi}(\Theta) \geq \pi(\Theta)\tilde{\pi}_1(\Theta) = \pi(\Theta) \geq c.
 \end{align*} 
 Thus we have that $\tilde{\psi} = \tilde{\pi} * \phi$ is such that $f_{\tilde{\psi}}$ is in the set $\pspace{c}{\Theta}$.

 Then as $f_{\psi}, f_{\tilde{\psi}} \in \pspace{c}{\Theta}$, we can use \ref{lm:bound-density-Mc} to conclude that $\forall x \in B_R$,
 \begin{align}\label{eq:logdiscmaxlogupbd}
 \max(- \log f_{\psi}(x), - \log f_{\tilde{\psi}}(x)) \le \|x\|^2 + C_{\ref{lm:bound-density-Mc}}.
 \end{align}
 Exponentiating the negative of both sides of equation~\eqref{eq:logdiscmaxlogupbd} we conclude that $\forall x \in B_R$,
 \begin{align}\label{eq:logdisc-mindenslwbd1}
 \min(f_{\psi}(x),f_{\tilde{\psi}}(x)) \ge e^{-C_{\ref{lm:bound-density-Mc}}}e^{-\norm{x}^2}.
 \end{align}
 We note that as $x \in B_R$, we have that $\norm{x} \leq R = \sqrt{80\log \frac{1}{\varepsilon}}$. Plugging this observation into \eqref{eq:logdisc-mindenslwbd1} we get
 \begin{align}\label{eq:logdisc-mindenslwbd2}
 \min(f_{\psi}(x),f_{\tilde{\psi}}(x)) \ge e^{-C_{\ref{lm:bound-density-Mc}}}\varepsilon^{80}.
 \end{align}

 We can now use Proposition~\ref{prop:loglips} with $u= f_{\psi}(x)$, $v=f_{\tilde{\psi}}(x)$ along with inequalities \eqref{eq:fdiffineq1} and \eqref{eq:logdisc-mindenslwbd2} to conclude
 \begin{align}\label{eq:logdisctmpfinal}
 \abs{\log f_{\psi}(x) - \log f_{\tilde{\psi}}(x)} \leq \frac{\abs{f_{\psi}(x) - f_{\tilde{\psi}}(x)}}{\min(f_{\psi}(x),f_{\tilde{\psi}}(x))} \leq \frac{(\pi(\Theta)C_{\ref{prop:discretizationnoncompact}} + \pi(\comp{\Theta})C_{\ref{prop:discretizationnoncompact}})\varepsilon^{81}}{e^{-C_{\ref{lm:bound-density-Mc}}}\varepsilon^{80}}
 \end{align}
 Rewriting the above equation with $C_{\ref{prop:logdiscretization}} = \frac{\pi(\Theta)C_{\ref{prop:discretizationnoncompact}} + \pi(\comp{\Theta})C_{\ref{prop:discretizationnoncompact}}}{e^{-C_{\ref{lm:bound-density-Mc}}}}$,we conclude
 \begin{align}
 \abs{\log f_{\psi}(x) - \log f_{\tilde{\psi}}(x)} \leq C_{\ref{prop:logdiscretization}}\varepsilon.
 \end{align}
\end{proof}

\subsubsection{Proof of Proposition~\ref{prop:lattice_cons}}\label{prf:applattice_cons}
Before we proceed onto the Proof of Proposition~\ref{prop:lattice_cons}, we will need the following technical proposition which is well known in the literature but whose proof we have included for completeness:
\begin{proposition}\label{prop:gausslips}
 Let $\varphi = \frac{1}{(2\pi)^{d/2}}e^{-\norm{x}^2/2} $ denote the density of the standard $d$-dimensional gaussian. Then for any $u,v \in \mathbb{R}^d$ we have for some constant $C_{\ref{prop:gausslips}}$
 \begin{align*}
 \abs{\varphi(u) - \varphi(v)} \leq C_{\ref{prop:gausslips}}\norm{u-v}_1
 \end{align*}
\end{proposition}

\begin{proof}
 Let us denote $x = (x_1, ..., x_d)$ and $y = (y_1, ..., y_d)$.
 
 Since $\varphi(x) = \frac{1}{(2\pi)^{d/2}}e^{-\norm{x}^2/2}$, we note that 
 \begin{align}\label{eq:gausslips1}
 \frac{\partial \varphi(x)}{\partial x_i} = -\frac{x_i}{(2\pi)^{d/2}}e^{-\norm{x}^2/2}
 \end{align}

 Now note that there is some constant $C_{\ref{prop:gausslips}}$ such that $\forall x \in \mathbb{R}^d$, we have:
 \begin{align}\label{eq:gausslips2}
 \frac{\abs{x_i}}{(2\pi)^{d/2}}e^{-\norm{x}^2/2} \leq C_{\ref{prop:gausslips}}
 \end{align}

 Combining equations~\eqref{eq:gausslips1} and \eqref{eq:gausslips2} we then have $\forall x \in \mathbb{R}^d$,
 \begin{align} \label{eq:gausslips3}
 \abs{\frac{\partial \varphi(x)}{\partial x_i}} \leq C_{\ref{prop:gausslips}}.
 \end{align}

 For any $1\leq i \leq d$, define the vector $x^{(i)} = (y_1, y_2, ..., y_{i}, x_{i+1}, ..., x_d)$. Let $x^{(0)} = x$. Also note that $x^{(d)} = y$. Then using equation~\eqref{eq:gausslips3}, we can conclude that 
 \begin{align}\nonumber
 \abs{\varphi(x^{(i)}) - \varphi(x^{(i-1)})} \leq C_{\ref{prop:gausslips}}\abs{x_i -y_i}.
 \end{align}

 The proof is complete by using triangle inequality along with the following telescoping sum
 \begin{align*}
 y - x = x^{(d)} - x^{(0)} = \sum_{i=1}^d x^{(i)} -x^{(i-1)}.
 \end{align*}
\end{proof}

Armed with the Proposition above we can continue on to prove Proposition~\ref{prop:lattice_cons}.

\begin{proof}[Proof of Proposition~\ref{prop:lattice_cons}]
For conciseness of notation let us write \begin{align}\label{eq:mdefn}
m = D_{\ref{prop:logdiscretization}}\log \left(\frac{1}{\varepsilon}\right)^d.\end{align}
Then $\exists z_1, ..., z_m \in B_R^a$ and $p_1, ...., p_m \in [0,1]$ such that
\begin{align}\label{eq:finitebrack-tildepisumdefn}
 \tilde{\pi} = \sum_{j=1}^m p_j\delta_{z_j}
\end{align}

Now for some fixed $\delta>0$ (to be chosen appropriately later), define the set
\begin{align}\Delta := \big\{ (k_1\cdot\delta, k_2\cdot\delta,\cdots, k_d \cdot\delta) \ : \ k_1,\cdots,k_d\in \Z\} \cap B_R^a.\end{align}
And for some fixed $\gamma>0$ such that $1/\gamma \in \Z$ (to be chosen appropriately later), define the set
\begin{align}\label{eq:gammaprobmass}
 \Gamma = \{\gamma \cdot k\ : \ k\in \Z_{\geq 0}\} \cap [0,1]. 
\end{align}

Now let $\mc{P}_{\delta,\gamma,m}$ to be the collection of all discrete distributions with at most $m+1$ atoms, with support in $\Delta$ and with the probability mass of each atom in the set $\Gamma$.

To approximate $\tilde{\pi}$ by an element in $\mc{P}_{\delta,\gamma,m}$ let us remind ourselves that by equation~\eqref{eq:finitebrack-tildepisumdefn} we can write $\tilde{\pi} = \sum_{j=1}^m p_j\delta_{z_j}$. Now for each $1 \leq j \leq m$, define $y_j$ to be the closest point in $\Delta$ to $z_j$. Similarly define $q_j$ to be the closest point in $\Gamma$ to $p_j$ that is smaller than $p_j$. Then define $\tilde{\tilde{\pi}}$, the "projection" of $\tilde{\pi}$ onto $\mc{P}_{\delta,\gamma,m}$ as
\begin{align*}
 \tilde{\tilde{\pi}} = \sum_{j=1}^m q_j \delta_{y_j} + \left(1 - \sum_{j=1}^m q_j\right) \delta_0.
\end{align*}

Note that as $\Delta$ is a $\delta$-lattice of $B_R^a$ and $z_j \in B_R^a$, we must have $\norm{y_j-z_j}_1 \leq d\delta$. This observation combined with Propositiion~\ref{prop:gausslips} (which amounts to proving that the gaussian density is Lipschictz), implies that if $\varphi(x) = \frac{1}{(2\pi)^{d/2}}e^{-\norm{x}^2/2}$ denotes the density of the standard gaussian we have:
\begin{align}\label{eq:phil1deldiff}
 \abs{\varphi(x-z_j) - \varphi(x-y_j)} \leq C_{\ref{prop:gausslips}}\norm{(x-z_j) -(x-y_j)}_1 \leq C_{\ref{prop:gausslips}}d\delta.
\end{align}
On the other hand as $\Gamma$ forms a $\gamma$-lattice and $q_j$ is chosen to be smaller than $p_j$ we conclude,
\begin{align}\label{eq:gammadiff}
 0 \leq p_j -q_j \leq \gamma.
\end{align}

Finally note that $\forall x, ~\varphi(x) \leq 1$. Combining this observation with equations~\eqref{eq:phil1deldiff} and \eqref{eq:gammadiff} we conclude that for $\tilde{\tilde{\psi}} = \tilde{\tilde{\pi}}*\phi$ and $\forall x\in B_R$ we have,
\begin{align}\nonumber
|f_{\tilde{\psi}} (x) - f_{\tilde{\tilde{\psi}}}(x)| 
&= \abs{\sum_{j=1}^m p_j \varphi(x-z_j) - \left(\sum_{j=1}^m q_j \varphi(x-y_j) + \sum_{j=1}^m (p_j-q_j)\varphi(x) \right)}
\\\nonumber&\leq \sum_{j=1}^m p_j|\varphi(x-z_j) - \varphi(x-y_j)| 
+ \sum_{j=1}^m |p_j - q_j| \varphi(x - y_j) + \sum_{j=1}^m |p_j - q_j| \varphi(x)
\\ \label{eq:ftildegriddiff} &\leq m(C_{\ref{prop:gausslips}}d\delta + 2\gamma).\end{align}

Let us now set \begin{align}\label{eq:finitebrac- deltavalue}\delta = \min\left(\frac{\varepsilon^{81}}{2C_{\ref{prop:gausslips}}dm},1\right),\end{align} 
and choose $\gamma$ such that \begin{align}\label{eq:finitebrac- gammavalue}\frac{1}{\gamma} = \left\lceil \frac{1}{\min\left(\frac{\varepsilon^{81}}{4m},1\right)}\right\rceil.\end{align} Plugging the values in equations~\eqref{eq:finitebrac- deltavalue} and \eqref{eq:finitebrac- gammavalue} in the inequality~\eqref{eq:ftildegriddiff} we conclude $\forall x\in B_R$ we have,
\begin{align}\label{eq:ftilftiltil-diff}
 \abs{f_{\tilde{\psi}} (x) - f_{\tilde{\tilde{\psi}}}(x)} \leq \varepsilon^{81}.
\end{align}
On the other hand by Lemma~\ref{lm:bound-density-Mc} as $\tilde{\psi} \in \pspace{c}{\Theta}$, we conclude $\forall x\in B_R$ we have, 
\begin{align}\label{eq:finitebrack-ftillwbd}
 f_{\tilde{\psi}}(x) \geq e^{-C_{\ref{lm:bound-density-Mc}}}e^{-\norm{x}_2^2} \geq e^{-C_{\ref{lm:bound-density-Mc}}}e^{-R^2} \geq e^{-C_{\ref{lm:bound-density-Mc}}}\varepsilon^{80}
\end{align}

Note that as $R \geq \norm{\Theta}_{\infty}$, we have $\Theta \subseteq B_R \subset B_R^a$. Thus since $\Delta$ is a $\delta$-lattice for $B_R^a$ and $\delta \leq 1$, for any $j$ such that $z_j \in \Theta$ we must have $y_j \in \Theta^{\sqrt{d}}$. Thus \begin{align}\nonumber\tilde{\tilde{\pi}}(\Theta^{\sqrt{d}}\cup\{0\}) \geq \sum_{j:z_j\in\Theta} q_j + 1 - \sum_{j=1}^m q_j = \sum_{j:z_j\in\Theta} q_j + \sum_{j=1}^m p_j - q_j\end{align}
By construction $p_j \geq q_j$, thus we conclude from the above inequality,
\begin{align}
\nonumber \tilde{\tilde{\pi}}(\Theta^{\sqrt{d}}\cup\{0\}) \geq \sum_{j:z_j\in\Theta} q_j + \sum_{j:z_j\in\Theta} p_j - q_j \geq \sum_{j:z_j\in\Theta} p_j = \tilde{\pi}(\Theta) \geq c.
\end{align}

Then as $\Theta^{\sqrt{d}}\cup\{0\}$ is also a compact set we can use Lemma~\ref{lm:bound-density-Mc} with $\Theta$ replaced by $\Theta^{\sqrt{d}}\cup\{0\}$ to conclude that for some constant $C'$ and $\forall x \in B_R$, we have
\begin{align}\label{eq:finitebrack-ftiltillwbd}
 f_{\tilde{\tilde{\psi}}}(x) \geq e^{-C'}e^{-\norm{x}_2^2} \geq e^{-C'}e^{-R^2} \geq e^{-C'}\varepsilon^{80}.
\end{align}

Thus using Proposition~\ref{prop:loglips} with $u=f_{\tilde{\psi}}$ and $v=f_{\tilde{\tilde{\psi}}}$ and using inequalities~\eqref{eq:ftilftiltil-diff}, \eqref{eq:finitebrack-ftillwbd} and \eqref{eq:finitebrack-ftiltillwbd} we get $\forall x \in B_R$ and for some constant $C_{\ref{prop:lattice_cons}}$,
\begin{align}\label{eq:logftilftiltildiff}
 \abs{\log f_{\tilde{\psi}}(x) - \log f_{\tilde{\tilde{\psi}}}(x)} \leq \frac{\abs{ f_{\tilde{\psi}}(x) - f_{\tilde{\tilde{\psi}}}(x)}}{\min(f_{\tilde{\psi}}(x), f_{\tilde{\tilde{\psi}}}(x))} \leq \frac{\varepsilon^{81}}{\min( e^{-C_{\ref{lm:bound-density-Mc}}},e^{-C'})\varepsilon^{80}} \leq C_{\ref{prop:lattice_cons}}\varepsilon
\end{align}

Then given $m, \delta, \gamma$ as defined in equations \eqref{eq:mdefn}, \eqref{eq:finitebrac- deltavalue}, \eqref{eq:finitebrac- gammavalue} we will now define the class of functions $\mc{H}$ as,
\begin{align} \label{eq:Fclassdefn}
 \mc{H} := \{\log f_{\tilde{\tilde{\psi}}} : \tilde{\tilde{\psi}} \in \mc{P}_{\delta,\gamma,m} \}
\end{align}

Now equation~\eqref{eq:logftilftiltildiff} already implies the first condition that $\mc{H}$ needs to satisfy. It remains to bound the cardinality of $\mc{H}$. To that end note that $\abs{\mc{H}} = \abs{\mc{P}_{\delta,\gamma,m}}$. The rest is a simple counting argument as follows.

We begin by noting that by plugging in the values of $m$ and $\delta$ from equations~\eqref{eq:mdefn} and \eqref{eq:finitebrac- deltavalue} we get
 \begin{align}\label{eq:deltalogbd}
 \log\frac{1}{\delta} = \log \frac{2C_{\ref{prop:gausslips}}dm}{\varepsilon^{81}} \lesssim d\log\left(\log\frac{1}{\varepsilon}\right) + \log\frac{1}{\varepsilon}\lesssim \log\frac{1}{\varepsilon}.\end{align}

Also note that \begin{align}\label{eq:Bepsainfnorm}\norm{B_R^a}_{\infty} \leq R+a \lesssim \sqrt{\log (1/\varepsilon)}\end{align} 

Next, observe that $|\Delta| \lesssim \frac{\norm{B_R^a}_{\infty}^d}{\delta^d}$. Combining this observation with equations~\eqref{eq:deltalogbd} and \eqref{eq:Bepsainfnorm} we conclude
\begin{align}\label{eq:deltalatticeupbd}
 \log \abs{\Delta} \lesssim d\left(\log \norm{B_R^a}_{\infty} + \log \frac{1}{\delta} \right) \lesssim \log \frac{1}{\varepsilon}.
\end{align}
 
Similarly plugging in the value of $\gamma$ from equation~\eqref{eq:finitebrac- gammavalue} we also have $\log \frac{1}{\gamma} \lesssim \log\frac{1}{\varepsilon}.$ Thus we have that 
\begin{align}\label{eq:gammalatticeupbd}
 \log \abs{\Gamma} \lesssim \log \frac{1}{\varepsilon}.
\end{align}

 Thus as $|\mc{P}_{\delta,\gamma,m}| \leq |\Delta|^m\cdot|\Gamma|^m$, we have by using equations~\eqref{eq:deltalatticeupbd}, \eqref{eq:gammalatticeupbd} and ~\eqref{eq:mdefn},
\begin{align}\label{eq:logpdelgam}
\log |\mc{P}_{\delta,\gamma,m}| \lesssim m \left(\log \abs{\Delta} + \log \abs{\Gamma} \right) \lesssim m \log \frac{1}{\varepsilon} \lesssim
D_{\ref{prop:logdiscretization}} \left(\log\frac{1}{\varepsilon}\right)^{d+1}\leq D_{\ref{prop:lattice_cons}}\left(\log\frac{1}{\varepsilon}\right)^{d+1}. \end{align}
Rewriting the above we thus have that 
\begin{align*}
 \abs{\mc{H}} = \abs{\mc{P}_{\delta,\gamma,m}} \leq \exp \left(D_{\ref{prop:lattice_cons}}\left(\log\frac{1}{\varepsilon}\right)^{d+1} \right).
\end{align*}
\end{proof}

\section{Proof of statistical results in \ref{subsec:stat_impl_L2mu}}
\subsection{Proof of Corollary~\ref{cor:clt_Lhat}}\label{app:prf_clt_Lhat}
\begin{proof}
Before we begin let us quickly comment on why \(\sigma_*^2>0\) holds.  Indeed, if \(\sigma_*^2=0\), then
\(\log f_*(X)\) is constant \(f_*\)-almost surely.  Since
\(f_*(x)>0\) for every \(x\in\mathbb R^d\), this implies that
\(\log f_*\) is constant Lebesgue-a.e.  By continuity of \(f_*\),
\(f_*\) would be constant on \(\mathbb R^d\), which is impossible for
a probability density on \(\mathbb R^d\).  

Since $\Theta \supseteq \operatorname{supp}(\mu_*)$, we have $\mu_* \in \mathcal{M}(\Theta,1)$. As such by Lemma~\ref{lm:bound-density-Mc} we have that
\[
    |\log f_*(x)| \le C_{\ref{lm:bound-density-Mc}}(1+\|x\|^2)
\]
for a constant \(C_{\ref{lm:bound-density-Mc}}<\infty\) depending only on \(d\) and \(R\).

If \(X\sim f_*\), then \(X\) admits the representation
\[
    X=\Theta+Z,
\]
where \(\Theta\sim\mu_*\), \(Z\sim N(0,I_d)\), and \(\Theta\) and
\(Z\) are independent.  Since \(\Theta\) is in a compact set almost surely, it follows
that \(\mathbb E\|X\|^4<\infty\).  Consequently,
\[
    \mathbb E_{f_*}\big[|\log f_*(X)|^2\big]<\infty.
\]
Thus the ordinary central limit theorem applies and gives
\[
    \sqrt n\big(L_n(f_*)-\mathbb E_{\mu_*}\big[\log f_*(X)\big]\big)
    =
    \frac{1}{\sqrt n}\sum_{i=1}^n
    \left(\log f_*(X_i)-\mathbb E_{\mu_*}\big[\log f_*(X)\big]\right)
    \rightsquigarrow
    N(0,\sigma_*^2).
\]

It remains only to compare \(L_n(f_*)\) and \(\hat L_n\).  By
Theorem~\ref{thm:L2mu_bound}, applied with \(p=2\),
\[
    \mathbb E\left[
        \left|\hat L_n-L_n(f_*)\right|^2
    \right]
    =
    o(n^{-1}).
\]
Equivalently,
\[
    \mathbb E\left[
        \left|\sqrt n\,(\hat L_n-L_n(f_*))\right|^2
    \right]
    \to 0.
\]
Hence
\[
    \sqrt n\,(\hat L_n-L_n(f_*))\to 0
    \quad\text{in }L^2,
\]
and therefore also in probability.  Since
\[
    \sqrt n(\hat L_n-\mathbb E_{\mu_*}\big[\log f_*(X)\big])
    =
    \sqrt n\big(L_n(f_*)-\mathbb E_{\mu_*}\big[\log f_*(X)\big]\big)
    +
    \sqrt n\big(\hat L_n-L_n(f_*)\big),
\]
Slutsky's theorem yields
\[
    \sqrt n(\hat L_n-\mathbb E_{f_*}\big[\log f_*(X)\big])
    \rightsquigarrow
    N(0,\sigma_*^2).
\]

\end{proof}

\subsection{Proof of Lemma \ref{sigma_est}}\label{prf:appsigma_est}

\begin{proof}[Proof of Lemma \ref{sigma_est}]
Write $\hat\ell_i := \log \hat f_n(X_i)$, $\ell_i^* := \log f_*(X_i)$, and
\[
\hat L_n = \frac{1}{n}\sum_{i=1}^n \hat\ell_i, \qquad 
L_n^* := L_n(f_*) = \frac{1}{n}\sum_{i=1}^n \ell_i^*, \qquad 
S_n^2 := \frac{1}{n}\sum_{i=1}^n (\ell_i^* - L_n^*)^2.
\]
We decompose
\[
\hat\sigma_n^2 - \sigma_*^2 \;=\; (\hat\sigma_n^2 - S_n^2) + (S_n^2 - \sigma_*^2)
\]
and show both pieces are $o_p(1)$.

We first note that by the weak law of large numbers applied to the i.i.d.\ variables
$\ell_i^*$ and $(\ell_i^*)^2$ (note that Corollary \ref{cor:clt_Lhat}
asserts that $\sigma_*^2 < \infty$.) we have
\[
L_n^* \xrightarrow{P} \mathbb{E}[\log f_*(X)], 
\qquad 
\frac{1}{n}\sum_{i=1}^n (\ell_i^*)^2 \xrightarrow{P} \mathbb{E}[(\log f_*(X))^2].
\]
By the continuous mapping theorem,
\begin{equation} \label{eq:stp1_F2}
S_n^2 \;=\; \frac{1}{n}\sum_{i=1}^n (\ell_i^*)^2 - (L_n^*)^2 
\;\xrightarrow{P}\; \operatorname{Var}_{f_*}\log f_*(X) \;=\; \sigma_*^2.    
\end{equation}

We will now show that 
\begin{equation}
\hat\sigma_n^2 - S_n^2 \xrightarrow{P} 0. \label{eq:stp2_F2}\end{equation}
Now expanding each as a second moment minus a squared mean,
\[
\hat\sigma_n^2 - S_n^2 \;=\; 
\underbrace{\frac{1}{n}\sum_{i=1}^n \bigl(\hat\ell_i^{\,2} - (\ell_i^*)^2\bigr)}_{=: A_n} 
\;-\; 
\underbrace{\bigl(\hat L_n^{\,2} - (L_n^*)^2\bigr)}_{=: B_n}.
\]

\emph{(2a) $B_n \to 0$ in probability.} Factor $B_n = (\hat L_n - L_n^*)(\hat L_n + L_n^*)$. By Theorem~\ref{thm:L2mu_bound} with $p=2$,
\[
\mathbb{E}\bigl[(\hat L_n - L_n^*)^2\bigr] = o(n^{-1}) 
\;\Longrightarrow\; \hat L_n - L_n^* \xrightarrow{L^2} 0 
\;\Longrightarrow\; \hat L_n - L_n^* \xrightarrow{P} 0.
\]
Since $L_n^* \xrightarrow{P} \mathbb{E}[\log f_*(X)]$ by Step~1, the sum $\hat L_n + L_n^* = O_p(1)$, so $B_n = o_p(1)\cdot O_p(1) = o_p(1)$.

\emph{(2b) $A_n \to 0$ in probability.} Using $a^2 - b^2 = (a-b)(a+b)$
and the Cauchy--Schwarz inequality, 
\[
|A_n| \;\le\; \sqrt{\frac{1}{n}\sum_{i=1}^n (\hat\ell_i - \ell_i^*)^2} 
\;\cdot\; \sqrt{\frac{1}{n}\sum_{i=1}^n (\hat\ell_i + \ell_i^*)^2}.
\]
The first factor is $o_p(1)$ by Corollary~\ref{cor:L2prob}.

For the second factor, $(a+b)^2 \le 2a^2 + 2b^2$ gives
\[
\frac{1}{n}\sum_{i=1}^n (\hat\ell_i + \ell_i^*)^2 
\;\le\; 2\cdot\frac{1}{n}\sum_{i=1}^n \hat\ell_i^{\,2} 
\;+\; 2\cdot\frac{1}{n}\sum_{i=1}^n (\ell_i^*)^2.
\]
The term $\frac{1}{n}\sum_i (\ell_i^*)^2$ is $O_p(1)$ by Step~1. For the $\hat\ell_i$ term, expand $\hat\ell_i = \ell_i^* + (\hat\ell_i - \ell_i^*)$ and apply $(a+b)^2 \le 2a^2 + 2b^2$ again:
\[
\frac{1}{n}\sum_{i=1}^n \hat\ell_i^{\,2} 
\;\le\; 2\cdot\frac{1}{n}\sum_{i=1}^n (\ell_i^*)^2 
\;+\; 2\cdot\frac{1}{n}\sum_{i=1}^n (\hat\ell_i - \ell_i^*)^2 
\;=\; O_p(1) + o_p(1) \;=\; O_p(1),
\]
where we used Corollary~\ref{cor:L2prob} once more. Therefore $\frac{1}{n}\sum_i (\hat\ell_i + \ell_i^*)^2 = O_p(1)$, and
\[
|A_n| \;\le\; \sqrt{o_p(1)} \cdot \sqrt{O_p(1)} \;=\; o_p(1).
\]

Combining (2a) and (2b), $\hat\sigma_n^2 - S_n^2 = A_n - B_n = o_p(1)$.

\noindent\textbf{Conclusion.} Adding equations~\eqref{eq:stp1_F2} and \eqref{eq:stp2_F2},
\[
\hat\sigma_n^2 - \sigma_*^2 \;=\; (\hat\sigma_n^2 - S_n^2) + (S_n^2 - \sigma_*^2) \;=\; o_p(1),
\]
i.e.\ $\hat\sigma_n^2 \xrightarrow{P} \sigma_*^2$.
\end{proof}

\subsection{Proof of Corollary \ref{cor:L2prob}}\label{prf:appL2prob}

\begin{proof}[Proof of Corollary \ref{cor:L2prob}]
  We shall use the following notation: 
  \begin{equation}\label{not_cor}
    \begin{split}
    & r(x) := \log \hat{f}_n(x) - \log f_*(x) \\ &\Delta_n :=  \left|\sup_{f \in
    \mathcal{M}} \frac{1}{n} \sum_{i=1}^n \log f(X_i) - \frac{1}{n}
    \sum_{i=1}^n \log f_*(X_i) \right| = \left| \hat{L}_n - L_n(f_*)
                                                   \right|.
     \end{split}                                                   
  \end{equation}
 Also $P_n g = \frac{1}{n} \sum_{i=1}^n g(X_i)$ denote the expectation
 of $g$ with respect to the empirical probability measure $P_n$ of
 $X_1, \dots, X_n$.

 Our main claim is that
 \begin{align}\label{mainclaim}
   P_n r^2 \leq e^2 \left(\Delta_n + n \Delta_n^2 \right). 
 \end{align}
 This proves \eqref{cor:L2prob.eq} because, according to Theorem
 \ref{thm:L2mu_bound}, both $\Delta_n$ and $n
 \Delta_n^2$ converge to zero in probability as $n \rightarrow
 \infty$.

 We prove \eqref{mainclaim} below. Because $\hat{f}_n$ maximizes
 likelihood over $\mathcal{M}$ and because $(1 - t) \hat{f}_n + t f_*
 \in \mathcal{M}$ for every $t \in [0, 1]$, we have
 \begin{align*}
   0 \geq \lim_{t \downarrow 0} \frac{L_n((1 - t) \hat{f}_n + t f_*) -
   L_n(\hat{f}_n)}{t} 
 \end{align*}
 Calculating the limit explicitly, we obtain
 \begin{align*}
   \frac{1}{n} \sum_{i=1}^n \frac{f_*(X_i)}{\hat{f}_n(X_i)} \leq 1. 
 \end{align*}
 Using the notation from \eqref{not_cor}, the above is same as
 \begin{align*}
   P_n e^{-r} \leq 1. 
 \end{align*}
 Letting
 \begin{align*}
   h(u) := e^{-u} - 1 + u \qt{for $u \in \R$}, 
 \end{align*}
 we then get
 \begin{align}\label{h_del}
   P_n h(r) = P_n e^{-r} - 1 + P_n r \leq P_n r = \frac{1}{n}
   \sum_{i=1}^n \log \hat{f}_n(X_i) - \frac{1}{n} \sum_{i=1}^n \log
   f_*(X_i) \leq \Delta_n. 
 \end{align}
We are now ready to prove \eqref{mainclaim}. We shall use the
inequalities for $h(u)$ given in Lemma  \ref{h_ineq}.  Write
\begin{align}\label{mainclaim_start}
  P_n r^2 = P_n r^2 I\{r \le 1\} + P_n r^2 I\{r > 1\}. 
\end{align}
For the first term above, we use the first part of Lemma \ref{h_ineq}
(and \eqref{h_del}) which gives
\begin{align}\label{mainclaim_1}
  P_n r^2 I\{r \leq 1\} \leq 2 e P_n h(r) \leq 2 e \Delta_n. 
\end{align}
For the second term, write
\begin{align*}
  P_n r^2 I\{r > 1\} &= \frac{1}{n} \sum_{i=1}^n r^2(X_i) I\{r(X_i) >
                       1\} \\
                     &\leq \left[ \max_{i : r(X_i) > 1} r(X_i) \right]
                     \left[ \frac{1}{n} \sum_{i=1}^n r(X_i) I\{r(X_i) >
                       1\}\right] \\
  &\leq \left[\sum_{i=1}^n r(X_i) I\{r(X_i) > 1\} \right] \left[
         \frac{1}{n} \sum_{i=1}^n r(X_i) I\{r(X_i) >
    1\}    \right] \\
                     &= n \left(\frac{1}{n} \sum_{i=1}^n r(X_i)
                       I\{r(X_i) > 1\} \right)^2 = n \left(P_n r I\{r
                       > 1\}  \right)^2. 
\end{align*}
Now using the second part of Lemma \ref{h_ineq} and \eqref{h_del}, we deduce
\begin{align}\label{mainclaim_2}
  P_n r^2 I\{r > 1\} \leq n e^2 \left(P_n h(r) \right)^2 \leq e^2 n \Delta_n^2
\end{align}
Combining \eqref{mainclaim_1} and \eqref{mainclaim_2} with
\eqref{mainclaim_start}, we deduce \eqref{mainclaim}. This completes
the proof of Corollary \ref{cor:L2prob}. 
\end{proof}

\begin{lemma}\label{h_ineq}
  Let $h(u) := e^{-u} - 1 + u$ for $u \in \R$. Then
  \begin{enumerate}
  \item $u^2 \leq 2 e h(u)$ for $u \leq 1$,
  \item $u \leq e h(u)$   for $u > 1$. 
  \end{enumerate}
\end{lemma}

\begin{proof}[Proof of Lemma \eqref{h_ineq}]
  Here is the proof of $u^2 \leq 2 e h(u)$ for $u \leq 1$. First
  assume that $u \leq 0$. Then writing $a = -u \geq 0$, we see that
  \begin{align*}
    h(u) = e^{-u} - 1 + u = e^{a} - 1 - a \geq \frac{a^2}{2} = \frac{u^2}{2}
  \end{align*}
  which shows $u^2 \leq 2 h(u)$ for $u \leq 0$.

  Now let $0 < u \leq 1$. Then
  \begin{align*}
    h(u) &= e^{-u} - 1 + u = \int_0^u (u - t) e^{-t} dt \geq
                                e^{-1} \int_0^u (u - t) dt =
                                \frac{u^2}{2e},            
  \end{align*}
  where, in the penultimate inequality, we used $e^{-t} \geq e^{-1}$
  because $0 \leq t \leq u \leq 1$. We have thus proved $u^2 \leq 2 e
  h(u)$ for all $u \leq 1$.

  Now let $u > 1$. Let
  \begin{align*}
    g(u) := h(u) - \frac{u}{e}. 
  \end{align*}
  It is easy to check that $g(1) = 0$. Also
  \begin{align*}
    g'(u) = -e^{-u} + 1 - \frac{1}{e} > -\frac{1}{e} + 1 - \frac{1}{e}
    = 1 - \frac{2}{e} > 0. 
  \end{align*}
  This proves that $g$ is increasing on $[1, \infty)$ which proves
  $g(u) > g(1) = 0$ for $u > 1$. But $g(u) > 0$ is equivalent to $u \leq e
  h(u)$. This completes the proof of Lemma \ref{h_ineq}. 
\end{proof}

\end{appendix}

\end{document}